\documentclass{amsart}
\usepackage{amssymb,amsthm,amsmath,amscd}
\usepackage[all]{xy}

\newcommand{\Z}{\mathbb{Z}}
\newcommand{\Q}{\mathbb{Q}}
\newcommand{\A}{\mathbb{A}}
\newcommand{\R}{\mathbb{R}}
\newcommand{\C}{\mathbb{C}}
\newcommand{\K}{\mathbb{K}}
\newcommand{\OO}{\mathfrak{o}}
\newcommand{\D}{\mathcal{D}}
\newcommand{\p}{\mathfrak{p}}
\newcommand{\cc}{\mathfrak{c}}
\newcommand{\G}{\mathcal{G}}
\newcommand{\W}{\mathcal{W}}
\newcommand{\KK}{\mathcal{K}}
\newcommand{\Ha}{\mathfrak{H}}
\newcommand{\Ht}{\mathcal{H}_2}
\newcommand{\Lam}{\Lambda}
\newcommand{\lam}{\lambda}
\newcommand{\Gam}{\Gamma}
\newcommand{\sD}{\sqrt{-D}}
\newcommand{\SD}{S_{2\kappa+1}(\Gamma_0(D),\chi)}
\newcommand{\Cl}{{\it Cl}_K}
\newcommand{\SL}{\mathrm{SL}}
\newcommand{\GL}{\mathrm{GL}}
\newcommand{\M}{\mathrm{M}}
\renewcommand{\O}{\mathrm{O}}
\newcommand{\GO}{\mathrm{GO}}
\newcommand{\GSO}{\mathrm{GSO}}
\newcommand{\PGSO}{\mathrm{PGSO}}
\newcommand{\SO}{\mathrm{SO}}
\newcommand{\U}{\mathrm{U}(2,2)}
\newcommand{\Her}{\mathrm{Her}_2}
\newcommand{\SU}{\mathrm{SU}(2,2)}
\newcommand{\GU}{\mathrm{GU}(2,2)}
\newcommand{\GSU}{\mathrm{GSU}(2,2)}

\newcommand{\GUplus}{\mathrm{GU}(2,2)(\R)^+}
\newcommand{\diag}{{\rm diag}}
\newcommand{\vol}{{\rm vol}}
\newcommand{\ord}{{\rm ord}}
\newcommand{\Tr}{{\rm Tr}}
\newcommand{\Lift}{{\it Lift}^{(2)}(f)}
\newcommand{\bs}{{\backslash}}
\newcommand{\bchi}{\underline{\chi}}
\newcommand{\FF}{{\bf F}}
\newcommand{\ff}{{\bf f}}
\newcommand{\g}{{\bf g}}
\newcommand{\I}{\sqrt{-1}}
\newcommand{\ii}{{\bf i}}
\newcommand{\re}{{\rm Re}}
\newcommand{\fin}{{\rm fin}}
\newcommand{\iif}{&\quad&{\rm if\ }}
\newcommand{\pair}[1]{\langle #1 \rangle}
\newcommand{\period}{\pair{F_\cc|_{\Ha\times\Ha},g\times g_C}}

\newtheorem{thm}{Theorem}[section]
\newtheorem{lem}[thm]{Lemma}
\newtheorem{prop}[thm]{Proposition}
\newtheorem{cor}[thm]{Corollary}

\newtheorem{rem}[thm]{Remark}
\newtheorem{example}[thm]{Example}

\title{Pullbacks of Hermitian Maass lifts}
\author{Hiraku Atobe}
\date{}
\address{Department of mathematics, Kyoto University, Kitashirakawa-Oiwake-cho, Sakyo-ku, Kyoto, 606-8502, Japan}
\email{atobe@math.kyoto-u.ac.jp}

\begin{document}
\maketitle

\section{Introduction}
Pullbacks of Siegel Eisenstein series have been studied by B$\ddot{\rm o}$cherer \cite{B}, Garrett \cite{G} and Heim \cite{H}.
Pullbacks of hermitian Eisenstein series have been studied by Furusawa \cite{F}, Harris \cite{Harris} and Saha \cite{Saha}.
These pullbacks have been used to study the algebraicity of critical values of certain automorphic $L$-functions.
Moreover, one might consider pullbacks of cusp forms.
The (Gan--)Gross--Prasad conjecture \cite{GP}, \cite{GGP} would relate critical values of certain $L$-functions and
the pullbacks of an automorphic representation of $\SO(n+1)$ to $\SO(n)$ or one of ${\rm U}(n+1)$ to ${\rm U}(n)$.
For example, in \cite{W}, \cite{Ichino duke} and \cite{GI}, 
the pullbacks of an automorphic representation of $\SO(n+1)$ to $\SO(n)$ for small $n$ were studied.
In \cite{Z1} and \cite{Z2}, Zhang studied the Gan--Gross--Prasad conjecture for ${\rm U}(n+1)$ to ${\rm U}(n)$ 
for general $n$ assuming some additional conditions.
On the other hand, 
Ichino \cite{Ichino} gave an explicit formula for pullbacks of Saito--Kurokawa lifts, which are Siegel cusp forms of degree $2$,
in terms of central critical values of $L$-functions for $\SL_2\times \GL_2$.
Ichino and Ikeda \cite{II} gave an explicit formula for the restriction of hermitian Maass lifts of degree $2$ 
to the Siegel upper half space of degree $2$ in terms of central critical values of triple product $L$-functions.
These results may be also regarded as special cases of the Gross--Prasad conjecture.
In this paper, we relate pullbacks of hermitian Maass lifts of degree $2$ to central values of $L$-functions for $\GL_2\times\GL_2$.
\par

Let us describe our results.
Let $K=\Q(\sD)$ be an imaginary quadratic field with discriminant $-D<0$.
We denote the ideal class group of $K$ by $\Cl$ and the class number of $K$ by $h_K$.
The primitive Dirichlet character corresponding to $K/\Q$ is denoted by $\chi$.
Let $\kappa$ be a positive integer and $f\in\SD$ be a normalized Hecke eigenform.
For an integral ideal $\cc$ of $K$ which is prime to $D$, 
we denote by $F_\cc$ the hermitian Maass lift of $f$ which satisfies the Maass relation for $\cc$. 
The lift $F_\cc$ is an automorphic form on the hermitian upper half space $\Ht$ of degree $2$ 
with respect to a certain arithmetic subgroup $\Gamma_K^{(2)}[\cc]\subset \U(\Q)$.
See Sect.~\ref{maass lift} for details.
Let $C=N(\cc)$ be the ideal norm of $\cc$ and $d(C)=\diag(1,C)\in\GL_2(\Q)$.
The pullback $F_\cc|_{\Ha\times\Ha}$ is in $S_{2\kappa+2}(\SL_2(\Z))\otimes S_{2\kappa+2}(d(C)^{-1}\SL_2(\Z)d(C))$.
For each normalized Hecke eigenform $g\in S_{2\kappa+2}(\SL_2(\Z))$, we put
$g_C(z)=g(z/C)\in S_{2\kappa+2}(d(C)^{-1}\SL_2(\Z)d(C))$ and consider the period integral
$\pair{F_\cc|_{\Ha\times\Ha},g\times g_C}$ given by
\begin{align*}
&\pair{F_\cc|_{\Ha\times\Ha},g\times g_C}\\
&=\int_{d(C)^{-1}\SL_2(\Z)d(C)\bs\Ha}\int_{\SL_2(\Z)\bs\Ha}F_\cc\left(
\begin{pmatrix}
z_1&0\\0&z_2
\end{pmatrix}
\right)
\overline{g(z_1)g_C(z_2)}y_1^{2\kappa}y_2^{2\kappa}dz_1dz_2.
\end{align*}
Let $L(s,f\times g)$ and $L(s,f\times g\times\chi)$ be the Rankin--Selberg $L$-function
and its twist
given by $f$ and $g$ of degree $4$.
We put $L_\infty(s)=\Gamma_\C(s+2\kappa+1/2)\Gamma_\C(s+1/2)$ 
with $\Gamma_\C(s)=2(2\pi)^{-s}\Gamma(s)$.
They satisfy the functional equation
\[
L_\infty(s)L(s,f\times g)=-D^{1-2s+2\kappa}a_f(D)^{-2}L_\infty(1-s)L(1-s,f\times g\times \chi),
\]
where $a_f(D)$ is the $D$-th Fourier coefficient of $f$.
Let $L(s,\chi)$ be the Dirichlet $L$-function associated with $\chi$.
\par
Our main result is as follows.\\
{\bf Theorem \ref{main}.} {\it The identity}
\[
L\left(\frac{1}{2},f\times g\right)=\frac{L(1,\chi)(4\pi)^{2\kappa+1}}{a_f(D)(2\kappa)!}\cdot
\frac{1}{h_K}
\sum_{[\cc]\in {\it Cl}_K}\frac{\pair{F_\cc|_{\Ha\times\Ha},g\times g_C}}{\pair{g_C,g_C}}
\]
{\it holds.}\par
Note that the period integrals do not appear in the (Gan--)Gross--Prasad conjecture.
We remark that the period integrals 
appearing
in the right hand side are not square.
This is a difference to the results of \cite{Ichino} or \cite{II},
which relate the central critical values of certain $L$-functions
to the square of the absolute values of period integrals.
\par
The sketch of proof is as follows.
The lifts $\{F_\cc\}$ give an automorphic form $\Lift$ on $\U(\A_\Q)$.
We consider the restriction of $\Lift$ to $({\rm U}(1,1)\times {\rm U}(1,1))(\A_\Q)$.
The group $\U$ is closely related to $\O(4,2)$ (see Sect.~\ref{GU-GL}).
So we may regard a certain theta lift on $\O(4,2)(\A_\Q)$ as a function on $\U(\A_\Q)$.
First, we prove that the theta lift is equal to $\Lift$ up to scalar multiplication (Proposition \ref{theta}).
The group ${\rm U}(1,1)\times {\rm U}(1,1)$ is closely related to $\O(2,2)\times\O(2)$
and the function $g\times g_C$ gives an automorphic form on $\O(2,2)(\A_\Q)$ 
(see Sect.~\ref{GL-GO(2,2)}).
Therefore, by the following seesaw identity, we find that a sum of period integrals of $\{F_\cc\}$ 
is equal to a sum of central values of $L(s,f\times g)$ and its twists (Corollary \ref{L(1/2)=<F,g*g>}).
$$
\xymatrix{
   \SL_2\times\SL_2   \ar@{-}[d]    &  &  \O(4,2) \ar@{-}[d]     \\
     \SL_2 \ar@{-}[urr]      &  &  \O(2,2)\times\O(2)\ar@{-}[ull]   \\
}.
$$
Finally, we show that these equations and the genus theory imply the main theorem.
\par
This paper is organized as follows.
In Sect.~\ref{maass lift}, we review the theory of hermitian Maass lifts.
In Sect.~\ref{Statement}, we state our main result.
In Sects.~\ref{GL2} and \ref{Weil rep}, we recall the basic facts about automorphic forms on $\GL_2$ and theta lifts, respectively.
In Sect.~\ref{GU-GL}, we study the hermitian Maass lifting.
In Sects.~\ref{GL-GO(2,2)} and \ref{SW}, we recall the theta correspondence for $(\GL_2,\GO(2,2))$ and $(\SL_2,\O(2))$, respectively.
In Sect.~\ref{main proof}, we prove identities for the above seesaw
and we show that these identities and the genus theory imply the main result.

{\bf Acknowledgments.}
The author would like to thank my advisor, Prof. Atsushi Ichino. 
Without his helpful support, this work would not have been completed.
The author is also thankful to Prof. Tamotsu Ikeda for useful discussions.

{\bf Notation.}
Let $K=\Q(\sD)$ be an imaginary quadratic field with discriminant $-D<0$.
We denote by $\OO$ the ring of integers of $K$.
Let $x\mapsto \overline{x}$ be the non-trivial Galois automorphism of $K$ over $\Q$.
The primitive Dirichlet character corresponding to $K/\Q$ is denoted by $\chi$.
We regard $K^1=\{\alpha\in K^\times|N_{K/\Q}(\alpha)=1\}$ as an algebraic group over $\Q$.
We denote by $J_K^D$ (resp.~$J_\OO^D$) the set of fractional ideals (resp.~integral ideals) of $K$ which are prime to $D$.
Here, we say that
a fractional ideal $\cc$ is prime to $D$ if $\ord_\p(\cc)=0$ for each prime ideal $\p\mid D$. 
Let ${\it Cl}_K$ be the ideal class group of $K$ and $h_K=\#{{\it Cl}_K}$ the ideal class number of $K$.
\par

We define the algebraic group $\Her$ of hermitian matrices of size $2$ with entries in $K$ by
\[
\Her(R)=\left\{\left.
\begin{pmatrix}
a&b+\sD c\\
b-\sD c &d
\end{pmatrix}
\right|a,b,c,d\in R
\right\}
\]
for any $\Q$-algebra $R$.
\par

For a number field $F$, we denote the adele ring of $F$ by $\A_F$.
The finite part of the adele ring (resp.~the idele group) of $F$ is denoted by $\A_{F,\fin}$ (resp.~$\A_{F,\fin}^\times$). 
Let $\psi_0=\otimes_v\psi_v$ be the non-trivial additive character of $\A_\Q/\Q$ defined as follows:
\begin{itemize}
\item If $v=p$, then $\psi_p(x)=e^{-2\pi\I x}$ for $x\in\Z[p^{-1}]$.
\item If $v=\infty$, then $\psi_\infty(x)=e^{2\pi\I x}$ for $x\in\R$. 
\end{itemize}
We call $\psi_0$ (resp.~$\psi_v$) the standard additive character of $\A_\Q$ (resp.~$\Q_v$).
We put $\hat\Z=\prod_p\Z_p$.
\par

Let $\bchi=\otimes_v\bchi_v$ be the character of the idele class group $\A_\Q^\times/\Q^\times$ determined by $\chi$.
Then $\bchi_v$ is the character of $\Q_v^\times$ corresponding to $\Q_v(\sD)/\Q_v$
and is given by the Hilbert symbol $\bchi_v(x)=(-D,x)_{\Q_v}$.
\par

Let
\[
B=\left\{
\begin{pmatrix}
*&*\\0&*
\end{pmatrix}
\in\SL_2
\right\}
\quad {\rm and}\quad
N=\left\{\left.
n(x)=
\begin{pmatrix}
1&x\\0&1
\end{pmatrix}
\right|x\in \mathbb{G}_a
\right\}
\]
be the standard Borel subgroup of $\SL_2$ and the unipotent radical of $B$, respectively.
We write
\[
a(x)=\begin{pmatrix}
x&0\\0&1
\end{pmatrix}
,\quad 
d(x)=\begin{pmatrix}
1&0\\0&x
\end{pmatrix}
,\quad 
t(x)=\begin{pmatrix}
x&0\\0&x^{-1}
\end{pmatrix}
\]
and
\[
\diag(x,y)=\begin{pmatrix}
x&0\\0&y
\end{pmatrix}
,\quad 
k_\theta=\begin{pmatrix}
\cos\theta&\sin\theta\\-\sin\theta&\cos\theta
\end{pmatrix}
.\]
We put $\GL_2(\R)^+=\{g\in\GL_2(\R)|\det(g)>0\}$.
For $N\in\Z_p$, we define
\[
{\bf K}_0(N;\Z_p)=\left\{\left.
\begin{pmatrix}
a&b\\c&d
\end{pmatrix}\in\GL_2(\Z_p)
\right|c\in N\Z_p
\right\}
\]
and $K_0(N;\Z_p)={\bf K}_0(N;\Z_p)\cap\SL_2(\Z_p)$.
\par
Let $\Ha=\{z\in\C|{\rm Im}(z)>0\}$ be the complex upper half plane.
For $z=x+\I y\in\Ha$, we put $dz=dxdy$ and $q=e^{2\pi\I z}$.
Here $dx,dy$ are the Lebesgue measures.
Note that
$\vol(\SL_2(\Z)\bs\Ha,y^{-2}dz)=\pi/3$.
\par

For an algebraic group $G$ over $\Q$, we put
$[G]=G(\Q)\bs G(\A_\Q)$.
\par

We put $\Gamma_\R(s)=\pi^{-s/2}\Gamma(s/2)$, $\Gamma_\C(s)=2(2\pi)^{-s}\Gamma(s)$ and 
$\xi_\Q(s)=\Gamma_\R(s)\zeta(s)$.

{\bf Measures.}
Let $dx_\infty$ be the Lebesgue measure on $\R$.
For each prime $p$, let $dx_p$ be the Haar measure on $\Q_p$ with
$\vol(\Z_p,dx_p)=1$.
We take the Haar measure $d^\times x_v=|x_v|_v^{-1}dx_v$ on $\Q_v^\times$.
\par

We normalize the Haar measures on $\SL_2(\Z_p)$ and $\SO(2)$ so that the total volumes are equal to $1$. 
For a place $v$ of $\Q$, we define a Haar measure $dg_v$ on $\SL_2(\Q_v)$ by
\[
dg_v=|a_v|_v^{-2}dx_vd^\times a_vdk_v
\]
for $g_v=n(x_v)t(a_v)k_v$ with $x_v\in \Q_v, a_v\in \Q_v^\times$, and
\[
k_v\in\left\{
\begin{aligned}
&\SL_2(\Z_p)\iif v=p,\\
&\SO(2)\iif v=\infty.
\end{aligned}
\right.
\]
We take the product measure $dg=\prod_vdg_v$ on $\SL_2(\A_\Q)$.
Note that the measure $\xi_\Q(2)^{-1}dg$ is the Tamagawa measure on $\SL_2(\A_\Q)$. 
For another connected linear algebraic group $G$ over $\Q$, we take the Tamagawa measure on $G(\A_\Q)$.
In particular, for a quadratic space $V$ over $\Q$ which is neither a hyperbolic plane nor $\dim(V)=1$, we have $\vol([\SO(V)])=2$.
We normalize the Haar measure on $\O(V)(\A_\Q)$ so that
$\vol([\O(V)])=1$.

\section{Hermitian Maass lifts}\label{maass lift}
In this section, we review the theory of hermitian modular forms and hermitian Maass lifts. 
See \cite{Ikeda}. 
\subsection{Hermitian modular forms}
The similitude unitary group $\GU$ is an algebraic group over $\Q$ defined by
\[
\GU(R)=\{ g\in \GL_4(K\otimes R) |{^t\overline{g}}Jg=\lam(g)J,\ \lam(g)\in R^\times\}
\]
with
\[
J=\begin{pmatrix}
0&-{\bf 1}_2\\ {\bf 1}_2&0
\end{pmatrix}\in \GL_4
\]
for any $\Q$-algebra $R$.
The homomorphism $\lam\colon\GU\rightarrow \GL_1$ is called the similitude norm.
Let $\U=\ker(\lam)$ be the unitary group and $\SU=\U\cap{\rm Res}_{K/\Q}(\SL_4)$ the special unitary group.
\par

We define the hermitian upper half space $\Ht$ of degree $2$ by
\[
\Ht=\left\{ Z\in \M_2(\C)\left| \frac{1}{2\sqrt{-1}}(Z-{^t\overline{Z}})>0\right.\right\}.
\]
For a fractional ideal $\cc$ of $K$, we define a subgroup $\Gamma_K^{(2)}[\cc]$ of $\U(\Q)$ by
\[
\Gamma_K^{(2)}[\cc]=\left\{g\in\U(\Q)\left|
g\begin{pmatrix}
\OO \\ \cc \\ \OO \\ \overline{\cc}^{-1}
\end{pmatrix}
=\begin{pmatrix}
\OO \\ \cc \\ \OO \\ \overline{\cc}^{-1}
\end{pmatrix}
\right.\right\},
\]
where $\overline{\cc}$ is the conjugate ideal of $\cc$.
Let $C=N(\cc)\in\Q_{>0}$ be the ideal norm of $\cc$.
We put
\[
\Lam_2^{\cc}(\OO)=\left\{\left.
\begin{pmatrix}
n&\alpha\\\overline{\alpha}&m/C
\end{pmatrix}\in \Her(\Q)
\right|
n, m\in\Z, \alpha\in \sD^{-1}\cc^{-1}
\right\}.
\]
The set of positive definite elements of $\Lam_2^\cc(\OO)$ is denoted by $\Lam_2^\cc(\OO)^+$.
\par

Let $\GUplus=\{g\in\GU(\R)|\lam(g)>0\}$.
Note that $\GUplus$ is generated by its center and $\U(\R)$.
We put
\[
g\pair{Z}=(AZ+B)(CZ+D)^{-1}\quad{\rm and}\quad
j(g,Z)=\det(CZ+D)
\]
for $Z\in\Ht$ and
\[
g=\begin{pmatrix}
A&B\\C&D
\end{pmatrix}\in\GUplus.
\]
We find that $g\pair{\cdot}$ gives an action of $\GUplus$ on $\Ht$ and the center of $\GUplus$ acts trivially.
For a holomorphic function $F$ on $\Ht$, an even integer $l$ and $g\in \GUplus$, we define
\[
(F\parallel_lg)(Z)=\det(g)^{l/2}F(g\pair{Z})j(g,Z)^{-l}.
\]
We put
\[
M_l(\Gam_K^{(2)}[\cc],{\det}^{-l/2})=\left\{F\left|F\parallel_l\gamma=F\ {\rm for\ any\ } \gamma\in\Gam_K^{(2)}[\cc]\right.\right\}.
\]
Then $F\in M_l(\Gam_K^{(2)}[\cc],{\det}^{-l/2})$ has a Fourier expansion of the form
\[
F(Z)=\sum_{H\in\Lam_2^\cc(\OO),H\geq 0}
A(H)\exp(2\pi\I\Tr(HZ)).
\]
The space of cusp forms $S_l(\Gam_K^{(2)}[\cc],{\det}^{-l/2})$ is defined by
\[
S_l(\Gam_K^{(2)}[\cc],{\det}^{-l/2})=\{F\in M_l(\Gam_K^{(2)}[\cc],{\det}^{-l/2})|
A(H)=0\ {\rm unless\ }H\in\Lam_2^\cc(\OO)^+\}.
\]

\subsection{Hermitian Maass lifts}
For $x\in\Q^\times$ and each prime $p$, we put
\[
x'=\prod_{p\nmid D}p^{\ord_p(x)}\quad {\rm and}\quad
x_p=p^{\ord_p(x)}.
\]
Let $Q_D$ be the set of all primes which divide $D$. 
We define a primitive Dirichlet character $\chi_p$ by
\[
\chi_p(n)=\left\{
\begin{aligned}
&\chi(m)\iif(n,p)=1,\\
&0\iif p\mid n,
\end{aligned}
\right.
\]
where $m$ is an integer such that
\[
m\equiv\left\{
\begin{aligned}
&n&\quad&\bmod D_p,\\
&1&\quad&\bmod D_p^{-1}D.
\end{aligned}
\right.
\]
One should not confuse $\chi_p$ with $\bchi_p$.
For $Q\subset Q_D$, we set
\[
\chi_Q=\prod_{p\in Q}\chi_p\quad{\rm and} \quad
\chi'_Q=\prod_{p\in Q_D\setminus Q}\chi_p.
\]
Note that $\chi_{\emptyset}=1$ and $\chi_{Q_D}=\chi$.
\par

Let $\kappa$ be a positive integer.
We fix a normalized Hecke eigenform $f=\sum_{n>0} a_f(n)q^n\in\SD$.
By Theorem 4.6.16 of \cite{M}, for each subset $Q\subset Q_D$, there exists a normalized Hecke eigenform
\[
f_Q=\sum_{n>0}a_{f_Q}(n)q^n\in\SD
\]
such that for each prime $p$, the Fourier coefficient $a_{f_Q}(p)$ satisfies
\[
a_{f_Q}(p)=\left\{
\begin{aligned}
&\chi_Q(p)a_f(p)\iif p\not\in Q,\\
&\chi'_Q(p)\overline{a_f(p)}\iif p\in Q.
\end{aligned}
\right.
\]
We fix $\cc\in J_\OO^D$ and put $C=N(\cc)\in\Z_{>0}$.
Note that $\chi(C)=1$.
Following \cite{Ikeda} Definition 15.2, we define
\[
f^{\cc*}=\sum_{Q\subset Q_D}\chi_Q(-C)f_Q.
\]
\begin{lem}\label{pure}
The Fourier coefficients of $f^{\cc*}$ are purely imaginary.
\end{lem}
\begin{proof}
In the case when $\cc=\OO$, the assertion is Lemma 1.1 of \cite{II}.
Let $Q\subset Q_D$ and put $Q'=Q_D\setminus Q$. 
Then we find that
\[
a_{f_{Q'}}(n)=\overline{a_{f_Q}(n)}
\]
for all $n\in\Z_{>0}$.
Since $(D,-C)=1$, 
we have $\chi_{Q'}(-C)=\chi(-C)\chi_Q(-C)^{-1}=\chi(-1)\chi_{Q}(-C)=-\chi_{Q}(-C)$.
Therefore, the Fourier coefficients of $\chi_Q(-C)f_{Q}+\chi_{Q'}(-C)f_{Q'}$
are purely imaginary.
\end{proof}
By \cite{Ikeda} Corollary 15.5, the $n$-th Fourier coefficient of $f^{\cc*}$ is given by
\[
a_{f^{\cc*}}(n)={\bf a}_D^\cc(n)\alpha_{F_\cc}(n),
\]
where 
\begin{align*}
{\bf a}_D^\cc(n)=\prod_{p\mid D}(1+\chi_p(-Cn)),\quad
\alpha_{F_\cc}(n)=
a_f(n')\prod_{p\mid (D,n)}\left(a_f(n_p)+\bchi_p(-Cn)\overline{a_f(n_p)}\right).
\end{align*}
By \cite{Ikeda} \S 16, 
we can define $F_\cc\in S_{2\kappa+2}(\Gamma_K^{(2)}[\cc],{\det}^{-\kappa-1})$ by
\[
F_\cc(Z)=\sum_{H\in\Lam_2^\cc(\OO)^+}
\left(\sum_{d\mid\varepsilon(H)}d^{2\kappa+1}\alpha_{F_\cc}\left(\frac{CD\det(H)}{d^2}\right)\right)
\exp(2\pi\I\Tr(HZ)).
\]
Here
\[
\varepsilon(H)=\varepsilon_{\cc}(H)=\max\{m\in\Z_{>0}|m^{-1}H\in\Lam_2^\cc(\OO)\}.
\]
We call $F_\cc$ the hermitian Maass lift of $f$ which satisfies the Maass relation for $\cc$.
\begin{lem}\label{Lem1}
Let $\cc\in J_\OO^D$ and $\alpha\in K^\times$ such that $\alpha\cc\in J_\OO^D$. Then
\[
F_{\cc}(Z)=F_{\alpha\cc}(d(\alpha)Zd(\overline{\alpha})).
\]
\end{lem}
\begin{proof}
Since ${^t\overline{(d(\alpha)Zd(\overline{\alpha}))}}=d(\alpha){^t\overline{Z}}d(\overline{\alpha})$, 
we find that $d(\alpha)Zd(\overline{\alpha})\in\Ht$ for all $Z\in\Ht$.
Note that 
$\Tr(Hd(\alpha)Zd(\overline{\alpha}))=\Tr(d(\overline{\alpha})Hd(\alpha)Z)$.
The map
\[
H\mapsto d(\overline{\alpha})Hd(\alpha)
\]
gives a bijection $\Lam_2^{\alpha\cc}(\OO)^+\rightarrow \Lam_2^{\cc}(\OO)^+$
which satisfies
\[
\varepsilon_\cc(d(\overline{\alpha})Hd(\alpha))=\varepsilon_{\alpha\cc}(H)
\quad{\rm and}\quad
N(\alpha\cc)D\det(H)=N(\cc)D\det(d(\overline{\alpha})Hd(\alpha)).
\]
This completes the proof.
\end{proof}
By Lemma \ref{Lem1}, we can define $F_\cc\in S_{2\kappa+2}(\Gamma_K^{(2)}[\cc],{\det}^{-\kappa-1})$
for a fractional ideal $\cc\in J_K^D$ by the same formula.
Moreover the function
\[
\Ha\times\Ha\ni(z_1,z_2)\mapsto 
F_\cc\left(
{\rm diag}(z_1,Cz_2)
\right)
\]
depends only on the ideal class of $\cc$ and defines an element in 
$S_{2\kappa+2}(\SL_2(\Z))\otimes S_{2\kappa+2}(\SL_2(\Z))$.
\begin{lem}\label{Tp}
Let $T(p)$ be the Hecke operator on $S_{2\kappa+2}(\SL_2(\Z))$. 
Then
\[
(T(p)\otimes {\rm id})(F_\cc(\diag(z_1,Cz_2)))=({\rm id}\otimes T(p))(F_\cc(\diag(z_1,Cz_2)))
\]
for all primes $p$.
\end{lem}
\begin{proof}
The proof is similar to that of Lemma 1.1 of \cite{Ichino}.
\end{proof}

\section{Statement of the main theorem}\label{Statement}
In this section, we state the main theorem and we give numerical examples.
\par

Let $\kappa$ be a positive integer and
\[
f(z)=\sum_{n=1}^\infty a_f(n)q^n\in\SD
\]
be a normalized Hecke eigenform. 
For $\cc \in J_K^D$, let $F_\cc\in S_{2\kappa+2}(\Gam_K^{(2)}[\cc],\det^{-\kappa-1})$
be the hermitian Maass lift of $f$ defined in the previous section.
Let $C=N(\cc)\in\Q_{>0}$ be the ideal norm of $\cc$.
For each normalized Hecke eigenform
\[
g(z)=\sum_{n=1}^\infty a_g(n)q^n\in S_{2\kappa+2}(\SL_2(\Z)),
\]
we set $g_C(z)=g(z/C)$.
Note that $g_C\in S_{2\kappa+2}(d(C)^{-1}\SL_2(\Z)d(C))$.
We consider the period integral
$\period$ given by
\begin{align*}
&\period
\\&=
\int_{d(C)^{-1}\SL_2(\Z)d(C)\bs\Ha}\int_{\SL_2(\Z)\bs\Ha}
F_\cc\left(
\begin{pmatrix}
z_1&0\\0&z_2
\end{pmatrix}
\right)
\overline{g(z_1)g_C(z_2)}y_1^{2\kappa}y_2^{2\kappa}dz_1dz_2.
\end{align*}
Define the Petersson norms of $g$ and $g_C$ by
\begin{align*}
\pair{g,g}=\int_{\SL_2(\Z)\bs\Ha}
|g(z)|^2y^{2\kappa}dz,\quad
\pair{g_C,g_C}=\int_{d(C)^{-1}\SL_2(\Z)d(C)\bs\Ha}
|g_C(z)|^2y^{2\kappa}dz.
\end{align*}
By Lemma \ref{Lem1}, we get the following lemma.
\begin{lem}\label{indep}
The map
\[
J_K^D\ni\cc\mapsto \frac{\pair{F_\cc|_{\Ha\times\Ha},g\times g_C}}{\pair{g,g}\pair{g_C,g_C}}\in\C
\]
factors through the ideal class group ${\it Cl}_K$.
\end{lem}
\par

For $p\not\in Q_D$, we define the Satake parameter $\{\alpha_{f,p}, \chi(p)\alpha_{f,p}^{-1}\}$
of $f$ at $p$ by
\[
1-a_f(p)X+\chi(p)p^{2\kappa}X^2=(1-p^\kappa\alpha_{f,p}X)(1-p^\kappa\chi(p)\alpha_{f,p}^{-1}X).
\]
For $p\in Q_D$, we put $\alpha_{f,p}=p^{-\kappa}a_f(p)$.
For each prime $p$, we define the Satake parameter $\{\alpha_{g,p}, \alpha_{g,p}^{-1}\}$
of $g$ at $p$ by
\[
1-a_g(p)X+p^{2\kappa+1}X^2=(1-p^{\kappa+1/2}\alpha_{g,p}X)(1-p^{\kappa+1/2}\alpha_{g,p}^{-1}X).
\]
The Ramanujan conjecture proved by Deligne states that 
$|\alpha_{f,p}|=|\alpha_{g,p}|=1$ for all $p$.
In particular, we have $|D^{-k}a_f(D)|=1$ and
$a_g(n)\in\R$ for all $n\in\Z_{>0}$.
We put
\[
A_p=\left\{
\begin{aligned}
&\begin{pmatrix}
\alpha_{f,p}&0\\0&\chi(p)\alpha_{f,p}^{-1}
\end{pmatrix}
\iif p\nmid D,\\
&\alpha_{f,p} \iif p\mid D,
\end{aligned}\right.
\quad 
{\rm and}\quad
B_p=\begin{pmatrix}
\alpha_{g,p}&0\\0&\alpha_{g,p}^{-1}
\end{pmatrix}.
\]
Define the $L$-functions $L(s,f\times g)$ and $L(s,f\times g\times \chi)$ by Euler products
\begin{align*}
L(s,f\times g)&=\prod_{p \nmid D}\det({\bf 1}_4-A_p\otimes B_p\cdot p^{-s})^{-1}
\times\prod_{p\mid D}\det({\bf 1}_2-A_p\otimes B_p\cdot p^{-s})^{-1},\\
L(s,f\times g\times \chi)&=\prod_{p \nmid D}\det({\bf 1}_4-A_p^{-1}\otimes B_p\cdot p^{-s})^{-1}
\times\prod_{p\mid D}\det({\bf 1}_2-A_p^{-1}\otimes B_p\cdot p^{-s})^{-1}.
\end{align*}
for $\re(s)\gg 0$. 
Note that 
$L(s,f\times g\times \chi)=\overline{L(\bar{s}, f\times g)}$ by the Ramanujan conjecture.
We also define
\[
\D(s,f,g)=\sum_{n=1}^\infty n^{-s}a_f(n)a_g(n)
=\sum_{n=1}^\infty n^{-s}a_f(n)\overline{a_g(n)}.
\]
Then, by \cite{Shimura} Lemma 1, we have
\[
\D(s+2\kappa+1,f,g)=L(2s+1,\chi)^{-1}L(s+1/2,f\times g),
\]
where $L(s,\chi)$ is the Dirichlet $L$-function associated to $\chi$.
\par

Let $\Lam(s,f\times g)$ and $\Lam(s,f\times g\times\chi)$
be the completed $L$-functions given by
\begin{align*}
\Lam(s,f\times g)&=\Gam_\C(s+1/2)\Gam_\C(s+2\kappa+1/2)L(s,f\times g),\\
\Lam(s,f\times g\times \chi)&=\Gam_\C(s+1/2)\Gam_\C(s+2\kappa+1/2)L(s,f\times g\times\chi).
\end{align*}
By \cite{J} Theorem 19.14, they have meromorphic continuations to the whole $s$-plane
and satisfy the functional equation
\[
\Lam(s,f\times g)=
\varepsilon(s,f\times g)\Lam(1-s,f\times g\times \chi).
\]
Here, $\varepsilon(s,f\times g)$ is the $\varepsilon$-factor which will be defined in the next section.
\par

Our main result is as follows.
\begin{thm}\label{main}
The identity
\[
L\left(\frac{1}{2},f\times g\right)=\frac{L(1,\chi)(4\pi)^{2\kappa+1}}{a_f(D)(2\kappa)!}
\cdot\frac{1}{h_K}
\sum_{[\cc]\in {\it Cl}_K}\frac{ \langle F_{\cc}|_{\Ha\times\Ha},g\times g_C\rangle}{\langle g_C,g_C \rangle}
\]
holds.
\end{thm}

\begin{rem}
A special case of a result of Shimura $($\cite{Shimura} Theorem $3$$)$ asserts that
\[
\pi^{-(2\kappa+1)}\pair{g,g}^{-1}\D(2\kappa+1,f,g)\in\Q(f)\Q(g)\subset\overline{\Q}
\]
and for all $\sigma\in{\rm Aut}(\C)$, one has
\[
\left[\pi^{-(2\kappa+1)}\pair{g,g}^{-1}\D(2\kappa+1,f,g)\right]^{\sigma}
=\pi^{-(2\kappa+1)}\pair{g^\sigma,g^\sigma}^{-1}\D(2\kappa+1,f^\sigma,g^\sigma).
\]
Here $\Q(f)$ $($resp.~$\Q(g))$ is the algebraic number field 
generated by the coefficients $\{a_f(n)\}$ of $f$ $($resp.~$\{a_g(n)\}$ of $g)$.
Note that in Shimura's paper, one takes the measure on $\SL_2(\Z)\bs\Ha$ so that
$ \vol(\SL_2(\Z)\bs\Ha)=1$.
Since 
\[
\frac{ \langle F_{\cc}|_{\Ha\times\Ha},g\times g_C\rangle}{\langle g_C,g_C \rangle\pair{g,g}}
\in\Q(f)\Q(g)
\]
and for all $\sigma\in{\rm Aut}(\C)$, one has
\[
\left[\frac{ \langle F_{\cc}|_{\Ha\times\Ha},g\times g_C\rangle}{\langle g_C,g_C \rangle\pair{g,g}}\right]^\sigma
=\frac{ \langle F^\sigma_{\cc}|_{\Ha\times\Ha},g^\sigma\times g^\sigma_C\rangle}
{\langle g^\sigma_C,g^\sigma_C \rangle\pair{g^\sigma,g^\sigma}},
\]
we find that Theorem $\ref{main}$ is compatible with this result.
\end{rem}

\begin{rem}
Let $g,g_1,g_2\in S_{2\kappa+2}(\SL_2(\Z))$ be normalized Hecke eigenforms.
If $g_1\not=g_2$, then
\[
\pair{F_\cc|_{\Ha\times\Ha},g_1\times (g_{2})_C}=0
\]
by Lemma \ref{Tp} and the multiplicity one theorem.
On the other hand, by Lemma \ref{pure}, we find that
\[
\pair{F_\cc|_{\Ha\times\Ha},g\times g_{C}}\in\I\R.
\]
\end{rem}

\begin{example}
We discuss the case $D=3$ and $\kappa=5$. Then the class number of $K=\Q(\sqrt{-3})$ is $h_K=1$.
Let $f\in S_{11}(\Gamma_0(3),\chi)$ be the Hecke eigenform such that
\[
f(z)=q + aq^2 + (9a - 27)q^3 + 304q^4 - 106aq^5 + (-27a - 6480)q^6 + 
    17234q^7 + O(q^8),
\]
with $a=12\sqrt{-5}$. Here we have used MAGMA \cite{BC1}.
Let $g\in S_{12}(\SL_2(\Z))$ be the normalized Hecke eigenform. 
We put
$\mathcal{A}=\{\alpha\in\sD^{-1}\OO | N(\alpha)<1\}$.
Then we have
\begin{align*}
\frac{\pair{F_\OO|_{\Ha\times\Ha},g\times g}}{\pair{g,g}^2}
=\sum_{\alpha\in\mathcal{A}}
\alpha_F\left(D\det
\begin{pmatrix}
1&\alpha\\\overline{\alpha}&1
\end{pmatrix}
\right).
\end{align*}
It is easy to see that
\[
\mathcal{A}=\left\{\left. \frac{a}{\sqrt{-D}}+\frac{(1+\sqrt{-D})b}{2\sqrt{-D}}\right|
(a,b)\in\{(0,\pm 1),(1,-1),(\pm 1,0),(0,0),(-1,1)\}\right\}.
\]
Since $\bchi_3(-3)=-1$, we have
\begin{align*}
\frac{\pair{F_\OO|_{\Ha\times\Ha},g\times g}}{\pair{g,g}^2}
=\alpha_F(3)+6\alpha_F(2)
=\left(a_f(3)+\bchi_3(-3)\overline{a_f(3)}\right)+6a_f(2)
=24a.
\end{align*}
On the other hand, by the Dirichlet class number formula, we have
\[
L(1,\chi)
=\frac{2\pi}{w_K\sqrt{D}}h_K=\frac{2\pi}{6\sqrt{3}}=\frac{\pi}{3\sqrt{3}},
\]
where $w_K$ is the number of roots of unity contained in $K$.
\par

Next, by using Dokchitser's program \cite{D}, we have computed
\begin{align*}
\pair{g,g}&=
0.00000 10353 62056 80432 09223 47816 81222 51645 93224 90796 \cdots,\\
L(1/2.f\times g)&=
(0.56063 39681 29898 43884 12923 20977 82681 50597 97544 70872 \cdots)\\
-&
(0.06268 07831 61695 17780 41809 51552 44859 45959 04505 29843 \cdots)\I.
\end{align*}
Therefore we have
\begin{align*}
&\frac{L(1/2,f\times g)a_f(D)(2\kappa)!}{L(1,\chi)(4\pi)^{2\kappa+1}\pair{g,g}}\\
&=
(643.98 75775 19939 43256 58420 16594 60755 58068 98087 56811 94085 90018 
\cdots)\I.
\end{align*}
This numerical value coincides with
$24a=24\times12\sqrt{-5}$.
\end{example}
\begin{example}
We discuss the case $D=15$ and $\kappa=5$. 
Then the class number of $K=\Q(\sqrt{-15})$ is $h_K=2$.
Let $\cc=\p=(17,(-11+\sqrt{-15})/2)$. 
This is a prime ideal above $p=N(\p)=17$.
The set $\{\OO,\p\}\subset J_K^D$ gives a complete system of representatives of ${\it Cl}_K$.
By using {\rm MAGMA} \cite{BC1}, we find that
there is  a normalized Hecke eigenform $f\in S_{11}(\Gamma_0(15),\chi)$ given by
\begin{align*}
f(z)&=q+(50.905\cdots)q^2+((190.983\cdots) + (150.247\cdots)\I)q^3\\
&+(1567.405\cdots)q^4+((-553.573\cdots) + (3075.578\cdots)\I)q^5+O(q^6).
\end{align*}
This satisfies $[\Q(f):\Q]=16$.
Let $g\in S_{12}(\SL_2(\Z))$ be the normalized Hecke eigenform. 
Then
\begin{align*}
a_f(15)&=
-(56782 2.2227 09865 28973 31440 49620 89700 18019 65209 77555 \cdots)\\
&+(50421 0.5849 95821 10937 05802 75205 40332 84135 04318 20609 \cdots\I)
\end{align*}
and
\begin{align*}
&\frac{\pair{F_\OO|_{\Ha\times\Ha},g\times g}}{\pair{g,g}^2}\\
&=-(18354 47.760 70085 24487 88571 78425 86039 65801 25241 46122 81062 
\cdots)\I,\\
&\frac{\pair{F_{\p}|_{\Ha\times\Ha},g\times g_p}}{\pair{g,g}\pair{g_p,g_p}}\\
&=(30037 50.957 75739 16600 06591 29262 95444 75281 08362 24854 36259 
\cdots)\I.
\end{align*}
Next, by using Dokchitser's program \cite{D}, we have computed
\begin{align*}
L(1,\chi)&=1.6223 11470 38944 47587 81184 30811 91756 19982 00362 52694 
\cdots,\\
L(1/2,f\times g)&=
(0.2917 40614 25112 91654 22127 46078 68148 91372 94585 49395 
\cdots)
\\&
- (0.3285 46859 12670 89836 52596 44096 76435 44616 81055 16850 
\cdots)\I.
\end{align*}
Therefore the numerical values of
\[
\frac{L(1/2,f\times g)a_f(D)(2\kappa)!}{L(1,\chi)(4\pi)^{2\kappa+1}\pair{g,g}}
\quad{\rm and}\quad
\frac{1}{h_K}\left(\frac{\pair{F|_{\Ha\times\Ha},g\times g}}{\pair{g,g}^2}
+\frac{\pair{F_{\p}|_{\Ha\times\Ha},g\times g_p}}{\pair{g,g}\pair{g_p,g_p}}\right)
\]
are both
\begin{align*}
&(58415 1.5985 28269 60560 90097 54185 47025 47399 15603 93657 
75983 44426 
\cdots)\I.
\end{align*}
\end{example}

\section{Automorphic forms on $\GL_2$}\label{GL2}
In this section, we recall the theory of automorphic forms on $\GL_2$.
\subsection{Automorphic forms and representations}
Let $f$ be an automorphic form on $\GL_2(\A_\Q)$ and 
$\chi$ be a character of $\A_\Q^\times/\Q^\times$.
We say that $f$ has a central character $\chi$ if $f$ satisfies
\[
f(ag)=\chi(a)f(g)
\]
for $a\in\A_\Q^\times$ and $g\in\GL_2(\A_\Q)$.
Let $\psi=\psi_0$ be the standard character of $\A_\Q$.
For $\xi\in\Q$, we define the $\xi$-th Fourier coefficient $W_{f,\xi}$ of $f$ by
\[
W_{f,\xi}(g)=\int_{\Q\bs\A_\Q}f(n(x)g)\overline{\psi(\xi x)}dx.
\]
\par

Let $f =\sum_{n>0} a_f(n)q^n \in \SD$ and $g=\sum_{n>0} a_g(n)q^n \in S_{2\kappa+2}(\SL_2(\Z))$ 
be normalized Hecke eigenforms. 
The automorphic form $f$ gives a cusp form \ff\ on $\GL_2(\A_\Q)$ by the formula
\[
\ff(\alpha)=\underline{\chi}(d)(f|\alpha_\infty)(\I)
\]
for $\alpha=\gamma \alpha_\infty k\in \GL_2(\A_\Q)$ with 
$\gamma\in \GL_2(\Q),\alpha_\infty\in \GL_2(\R)^+$ and
\[
k=
\begin{pmatrix}
a&b\\
c&d
\end{pmatrix}
\in {\bf K}_0(D;\hat{\Z}),
\quad{\rm where}\quad
{\bf K}_0(D;\hat{\Z})=\left\{
\begin{pmatrix}
a&b\\
c&d
\end{pmatrix}
\in \GL_2(\hat{\Z})
\left| c\in D\hat{\Z}
\right.\right\}.
\]
Note that the central character of $\ff$ is $\bchi$. 
The automorphic form $g$ gives a cusp form $\g$ on $\GL_2(\A_\Q)$ by the formula
\[
\g(\beta)=(g|\beta_\infty)(\I)
\]
for $\beta=\gamma'\beta_\infty k'$ with 
$\gamma'\in \GL_2(\Q),\beta_\infty\in \GL_2(\R)^+$ and
$k'\in\GL_2(\hat{\Z})$.
Note that the central character of $\g$ is the trivial character.
\par

Let $\pi_f\cong\otimes'_v\pi_{f,v}$ (resp.~$\pi_g\cong\otimes'_v\pi_{g,v}$)
be the irreducible cuspidal automorphic representation of $\GL_2(\A_\Q)$
generated by $\ff$ (resp.~$\g$).
Then the central character of $\pi_f$ (resp.~$\pi_g$) is $\bchi$ (resp.~the trivial character).
The $\infty$-component $\pi_{f,\infty}$ (resp.~$\pi_{g,\infty}$) is the (limit of) discrete series representation of $\GL_2(\R)$
with minimal weight $\pm(2\kappa+1)$ (resp.~$\pm(2\kappa+2)$). 
The $p$-component $\pi_{f,p}$ (resp.~$\pi_{g,p}$) is the principal series representation
\[
\pi_{f,p}\cong {\rm Ind}_{{\bf B}_p}^{{\bf G}_p}(|\cdot|_p^{s_{f,p}} \boxtimes \underline{\chi}_p|\cdot|_p^{-s_{f,p}}),
\quad \left({\rm resp.~}
\pi_{g,p}\cong {\rm Ind}_{{\bf B}_p}^{{\bf G}_p}(|\cdot|_p^{s_{g,p}} \boxtimes |\cdot|_p^{-s_{g,p}})\right).
\]
Here 
we put ${\bf G}_p=\GL_2(\Q_\p)$ and we denote ${\bf B}_p$ the Borel subgroup of ${\bf G}_p$ consisting all upper triangle matrices,
and ${s_{f,p},\ s_{g,p}\in\C}$ satisfy 
$|p|_p^{s_{f,p}}=\alpha_{f,p}$ and $|p|_p^{s_{g,p}}=\alpha_{g,p}$.
The restriction of $\pi_{f,\infty}$ to $\SL_2(\R)$ is the direct sum of two irreducible representations $\pi_f^+$ and $\pi_f^-$,
i.e., $\pi_{f,\infty}|_{\SL_2(\R)}\cong \pi_f^+\oplus\pi_f^-$.
Here, $\pi_f^+$ (resp.~$\pi_f^-$) is the holomorphic (resp.~anti-holomorphic) discrete series representation of $\SL_2(\R)$
with minimal weight $2\kappa+1$ (resp.~$-(2\kappa+1)$).
Note that $\ff\in \pi_f^+\otimes(\otimes_p'\pi_{f,p})$.

\subsection{The $L$-functions}
Let $\pi=\pi_f\times\pi_g$. 
Then by \cite{J} \S14, \S17 and \S19,
we can define the $L$-functions and the $\varepsilon$-factor
\[
L(s,\pi)=\prod_{v}L(s,\pi_v),\quad 
L(s,\pi^{\vee})=\prod_{v}L(s,\pi^{\vee}_v),\quad 
\varepsilon(s,\pi)=\prod_v \varepsilon(s,\pi_v,\psi_v)
\]
where $\pi^{\vee}$ is the contragredient representation of $\pi$.
By \cite{J} Theorem 19.14, we have the functional equation
\[
L(s,\pi)=\varepsilon(s,\pi)L(1-s,\pi^\vee).
\]
Let $\Lam(s,f\times g)$ and $\Lam(s,f\times g \times\chi)$ be the completed $L$-functions defined in Sect.~\ref{Statement}.
\begin{lem}
We have
\[
\Lam(s,f\times g)=L(s,\pi),\quad \Lam(s,f\times g\times \chi)=L(s,\pi^{\vee}).
\]and
\[
\varepsilon(s,\pi)=-D^{1-2s+2\kappa}a_f(D)^{-2}.
\]
\end{lem}
\begin{proof}
Since $\prod_{p<\infty}\bchi_p(-1)=\bchi(-1)\bchi_\infty(-1)^{-1}=-1$,
it is enough to show the equations
\begin{align*}
L(s,\pi_p)&=\det({\bf 1}_r-A_p\otimes B_p\cdot p^{-s})^{-1},\\
L(s,\pi^\vee_p)&=\det({\bf 1}_r-A_p^{-1}\otimes B_p\cdot p^{-s})^{-1},\\
\varepsilon(s,\pi_p,\psi_p)&=\bchi_p(-1)(p^d)^{1-2s+2\kappa}a_f(p^d)^{-2}
\end{align*}
for $p<\infty$ with $d=\ord_p(D)$ and
\[
r=\left\{\begin{aligned}
&4\iif p\nmid D,\\
&2\iif p\mid D,\\
\end{aligned}
\right.
\]
and
\begin{align*}
L(s,\pi_\infty)=L(s,\pi^\vee_\infty)&=\Gamma_\C(s+1/2)\Gamma_\C(s+2\kappa+1/2),\\
\varepsilon(s,\pi_\infty,\psi_\infty)&=1.
\end{align*}
\par

For $v=p<\infty$, by \cite{J} Theorem 15.1, we have
\begin{align*}
\varepsilon(s,\pi_p,\psi_p)
&=\varepsilon(s,\pi_{f,p}\otimes |\cdot|_p^{s_{g,p}},\psi_p)
	\varepsilon(s,\pi_{f,p}\otimes |\cdot|_p^{-s_{g,p}},\psi_p)\\
&=\varepsilon(s,|\cdot|_p^{s_{f,p}+s_{g,p}},\psi_p)
	\varepsilon(s,|\cdot|_p^{s_{f,p}-s_{g,p}},\psi_p)
\\&\quad\times
	\varepsilon(s,\bchi_p|\cdot|_p^{-s_{f,p}+s_{g,p}},\psi_p)
	\varepsilon(s,\bchi_p|\cdot|_p^{-s_{f,p}-s_{g,p}},\psi_p)\\
&=\varepsilon(s-s_{f,p}+s_{g,p},\bchi_p,\psi_p)\varepsilon(s-s_{f,p}-s_{g,p},\bchi_p,\psi_p).
\end{align*}
If $p\nmid D$, then the character $\bchi_p$ is unramified.
So we have
\[
\varepsilon(s,\pi_p,\psi_p)=1=\bchi_p(-1)(p^d)^{1-2s+2\kappa}a_f(p^d)^{-2}.
\]
\par

If $p\mid D$, then $d=\ord_p(D)>0$.
Let $U_p^{(0)}=\Z_p^\times$ and $U_p^{(n)}=1+p^n\Z_p$ for $n\geq 1$. 
We have
$\bchi_p|U_p^{(d)}=1$ and $\bchi_p|U_p^{(d-1)}\not=1$.
Then it is well-known that 
\[
\varepsilon(s,\bchi_p,\psi_p)=p^{d(1/2-s)}\varepsilon(1/2,\bchi_p,\psi_p)
\]
(see e.g., \cite{BH} 23.5 Theorem).
So we have
\[
\varepsilon(s,\pi_p,\psi_p)
=p^{d(1-2s)}p^{2ds_{f,p}}\varepsilon(1/2,\bchi_p,\psi_p)^2.
\]
It is well-know that $\varepsilon(1/2,\bchi_p,\psi_p)^2=\bchi_p(-1)$ 
(see e.g., \cite{BH} 23.4 Corollary 2).
On the other hand, we have
$p^{2ds_{f,p}}=\alpha_{f,p}^{-2d}=(p^{-\kappa}a_f(p))^{-2d}=(p^d)^{2\kappa}a_f(p^d)^{-2}$.
This gives the desired formula for $\varepsilon(s,\pi_p,\psi_p)$.
By calculations similar to that of $\varepsilon(s,\pi_p,\psi_p)$, we have
the desired formulas for $L(s,\pi_p)$ and $L(s,\pi^\vee_p)$.
\par

Next, we assume that $v=\infty$.
Let $W_\R=\C^\times\cup j\C^\times$ be the Weil group of $\R$ with
$j^2=-1$ and $jzj^{-1}=\bar{z}$
for $z\in\C^\times$.
Let $\rho_{f,\infty}$ (resp $\rho_{g,\infty}$) be the two dimensional representation of $W_\R$
which corresponds to $\pi_{f,\infty}$ (resp.~$\pi_{g,\infty}$) by the local Langlands correspondence.
Then by \cite{J} Proposition 17.3, we have
\begin{align*}
L(s,\pi_\infty)=L(s,\rho_{f,\infty}\otimes\rho_{g,\infty})&,
\quad L(s,\pi^{\vee}_\infty)=L(s,\rho^{\vee}_{f,\infty}\otimes\rho^{\vee}_{g,\infty}),
\\
\varepsilon(s,\pi_\infty,\psi_\infty)&=\varepsilon(s,\rho_{f,\infty}\otimes\rho_{g,\infty},\psi_\infty).
\end{align*}
For $n\in\Z$, we define the two dimensional representation $\rho_n$ of $W_\R$ by
\begin{align*}
\rho_n \colon W_\R=\C^\times\cup j\C^\times &\rightarrow \GL_2(\C),\\
re^{\I \theta}\mapsto 
\begin{pmatrix}
e^{n \I\theta} &0\\0&e^{-n \I\theta}
\end{pmatrix},&\quad 
j\mapsto
\begin{pmatrix}
0&(-1)^{n}\\1&0
\end{pmatrix}.
\end{align*}
Then we have $\rho_{-n}\cong\rho_n$ for $n\in\Z$.
Hence $\rho_n$ is a self-dual representation.
Moreover, for $n_1,n_2\in\Z$, we have
$\rho_{n_1}\otimes\rho_{n_2}\cong \rho_{n_1+n_2}\oplus\rho_{n_1-n_2}$.
For $n\in\Z_{>0}$, the $L$-function and the $\varepsilon$-factor associated to $\rho_n$ are given by
\[
L(s,\rho_n)=\Gamma_\C(s+n/2),\quad
\varepsilon(s,\rho_n,\psi_\infty)=\I^{n+1}.
\]
See \cite{Tate}.
It is well-known that 
$\rho_{f,\infty}\cong \rho_{2\kappa}$ and $\rho_{g,\infty}\cong \rho_{2\kappa+1}$.
So we have
\begin{align*}
L(s,\pi_\infty)=L(s,\pi^{\vee}_\infty)&=\Gamma_\C(s+2\kappa+1/2)\Gamma_\C(s+1/2),\\
\varepsilon(s,\pi_\infty,\psi_\infty)&=\I^{(4\kappa+2)+2}=1.
\end{align*}
This completes the proof.
\end{proof}

We put $\varepsilon(s,f\times g)=\varepsilon(s,\pi)$. 
Then we have the functional equation noted in Sect.~\ref{Statement}.
In particular, this functional equation and the Ramanujan conjecture imply that
\[
a_f(D)D^{-\kappa}\Lam(1/2,f\times g)
=-(a_f(D)D^{-\kappa})^{-1}\Lam(1/2,f\times g\times\chi)
=-\overline{a_f(D)D^{-\kappa}\Lam(1/2,f\times g)}.
\]
So we find that 
\[
a_f(D)L(1/2,f\times g)\in\I\R.
\]

\subsection{The Whittaker function on $\GL_2$}\label{Whittaker}
Let $\psi=\psi_0$ be the standard character of $\A_\Q$.
The Whittaker function $W_\ff$ (resp.~$W_\g$) of $\ff$ (resp.~$\g$) is defined by
\[
W_\ff(\alpha)=\int_{\Q\backslash\A_\Q}\ff\left(
n(x)
\alpha\right)\overline{\psi(x)}dx
\quad \left({\rm resp.} \quad 
W_\g(\alpha)=\int_{\Q\bs\A_\Q}\g(n(x)\alpha)
\overline{\psi(x)}dx
\right).
\]
Note that $W_{\ff,\xi}(\alpha)=W_{\ff}(a(\xi)\alpha)$ and 
$W_{\g,\xi}(\alpha)=W_{\g}(a(\xi)\alpha)$ for $\xi\in\Q^\times$.
The function $W_\ff$ (resp.~$W_\g$) has a product expansion $W_\ff=\prod_vW_{\ff,v}$ 
(resp.~$W_\g=\prod_vW_{\g,v}$), 
where
\begin{align*}
W_{\ff,\infty}\left(
n(y)a(x)k_\theta
\right)
&=\left\{
\begin{aligned}
&e^{2\pi\I y}x^{\kappa+(1/2)}e^{-2\pi x}e^{\I(2\kappa+1)\theta}
 \quad &{\rm if}\ x>0,\\
&0		  \quad &{\rm if}\ x<0,
\end{aligned}\right.\\
W_{\g,\infty}\left(
n(y)a(x)k_\theta
\right)
&=\left\{
\begin{aligned}
&e^{2\pi\I y}x^{\kappa+1}e^{-2\pi x}e^{\I(2\kappa+2)\theta}
 \quad &&{\rm if}\ x>0,\\
&0		  \quad &&{\rm if}\ x<0.
\end{aligned}\right.
\end{align*}
For $p\nmid D$, the function $W_{\ff,p}$ satisfies
$W_{\ff,p}(\alpha k)=W_{\ff,p}(\alpha)$
for all $\alpha\in \GL_2(\Q_p)$ and $k\in \GL_2(\Z_p)$. 
For $p\mid D$, the function $W_{\ff,p}$ satisfies
$W_{\ff,p}(\alpha k)=\underline{\chi}_p(k)W_{\ff,p}(\alpha)$
for all $\alpha\in \GL_2(\Q_p)$ and $k\in {\bf K}_0(D;\Z_p)$.
Here, we define
\[
\bchi_p(k)=\bchi_p(d),
\quad{\rm for}\quad 
k=
\begin{pmatrix}
a&b\\
c&d
\end{pmatrix}
\in {\bf K}_0(D;\Z_p).
\]
For each prime $p$, the function $W_{\g,p}$ satisfies
$W_{\g,p}(\alpha k)=W_{\g,p}(\alpha)$
for all $\alpha\in \GL_2(\Q_p)$ and $k\in \GL_2(\Z_p)$. 
\par

The following lemma is a reformulation of Lemma 6.3 of \cite{II}
in terms of the Fourier coefficients of $f$ and $g$.
Here, we put $a_f(x)=0$ and $a_g(x)=0$ unless $x\in\Z_{>0}$.
For $x=p^nu\in \Q_p^\times$ with $u\in \Z_p^\times$, we set $x_p=p^n$.
\begin{lem}\label{explicit}
We put $d=\ord_p(D)$, $\zeta_8=\exp(\pi\I/4)$, 
\[
\varepsilon=\left\{
\begin{split}
&1&\quad&{\rm if\ }\underline{\chi}_p(-1)=1,\\
&\I&\quad&{\rm if\ }\underline{\chi}_p(-1)=-1,
\end{split}
\right.
\]
and
\[
w=\begin{pmatrix}
0&-1\\1&0
\end{pmatrix}
,\quad 
k_1=\begin{pmatrix}
1&0\\p^{d-1}&1
\end{pmatrix}
,\quad 
k_2=\begin{pmatrix}
1&0\\p^{d-2}&1
\end{pmatrix}.
\]
\begin{enumerate}
\item
For each prime $p$, we have
\[
W_{\g,p}(a(x))=|x|_p^{\kappa+1}a_g(x_p).
\]
\item
For $p\nmid D$, we have
\[
W_{\ff,p}\left(
a(x)
\right)
=|x|_p^{\kappa+1/2}a_f(x_p).
\]
\item
For $p\mid D$, we have
\[
W_{\ff,p}\left(
a(x)
w\right)
=\underline{\chi}_p(-D_px)\varepsilon^{-1}p^{-d/2}a_f(D_p)^{-1}
|x|_p^{\kappa+1/2}\overline{a_f((Dx)_p)}.
\]
\item
For $p=2\mid D$, we have
\[
W_{\ff,p}\left(
a(x)
k_1\right)
=\left\{
\begin{aligned}
&|x|_2^{\kappa+1/2}a_f(2)^{-1}&\quad&{\rm if\ }\ord_2(x)=-1,\\
&0&\quad&{\rm otherwise}.
\end{aligned}
\right.
\]
\item
For $p=2$, $d=3$ and $x=2^nu$ with $u\in\Z_2^\times$, we have
\[
W_{\ff,p}\left(
a(x)
k_2\right)
=\left\{
\begin{aligned}
&2^{-1/2}\underline{\chi}_p(u)\varepsilon \zeta_8^{-u}
|x|_2^{\kappa+1/2}a_f(2)^{-2}&\quad&{\rm if\ }\ord_2(x)=-2,\\
&0&\quad&{\rm otherwise}.
\end{aligned}
\right.
\]
\end{enumerate}
\end{lem}

\section{Weil representations and theta lifts}\label{Weil rep}
In this section, we recall the theory of Weil representations and theta lifts.
\subsection{Quadratic spaces}
Let $F$ be a field of characteristic not $2$ 
and $V$ a quadratic space over $F$.
Namely, $V$ is a vector space over $F$ of dimension $m$
equipped with a non-degenerate symmetric bilinear form $(\ ,\ )$.
We assume that $m$ is even.
Let $Q$ denote the associated quadratic form on $V$.
Then 
\[
Q[x]=\frac{1}{2}(x,x)
\]
for $x\in V$.
We fix a basis $\{v_1,\dots,v_m\}$ of $V$ and identify $V$ with the space of column vectors $F^m$.
Define $Q\in\GL_m(F)$ by
\[
Q=((v_i,v_j))_{i,j}
\]
and let $\det(V)$ denote the image of $\det(Q)$ in $F^\times/F^{\times2}$.
Let 
\[
\GO(V)=\{h\in\GL_m|{}^thQh=\nu(h)Q,\ \nu(h)\in\GL_1\}
\]
be the orthogonal similitude group and 
$\nu\colon\GO(V)\rightarrow\GL_1$
the similitude norm.
We put
\[
\GSO(V)=\{h\in\GO(V)|\det(h)=\nu(h)^{m/2}\}.
\]
Let $\O(V)=\ker(\nu)$ be the orthogonal group and 
$\SO(V)=\O(V)\cap\SL_m$ the special orthogonal group.

\subsection{Weil representations}
Let $F$ be a local field of characteristic not $2$ and 
$V$ a quadratic space over $F$ of dimension $m$.
We assume that $m$ is even.
We fix a non-trivial additive character $\psi$ of $F$.
For $a\in F^\times$, we define a non-trivial additive character $a\psi$ of $F$ by
$(a\psi)(x)=\psi(ax)$.
We define a quadratic character $\chi_V$ of $F^\times$ by
\[
\chi_V(a)=((-1)^{m/2}\det(V),a)_F
\]
for $a\in F^\times$,
where $(x,y)_F$ is the Hilbert symbol of $F$.
We define an $8$-th root of unity $\gamma_V$ by
\[
\gamma_V=\gamma_F\left(\det(V),\frac{1}{2}\psi\right)\gamma_F\left(\frac{1}{2}\psi\right)^mh_F(V).
\]
Here $h_F(V)$ is the Hasse invariant of $V$.
Note that $\chi_V$ and $\gamma_V$ depend only on the anisotropic kernel of $V$.
To calculate $\gamma_F(a,b\psi)$ and $\gamma_F(a\psi)$, see \cite{Ichino} A.1.
\par

Let $\omega=\omega_{V,\psi}$ denote the Weil representation of
$\SL_2(F)\times \O(V)$ on $\mathcal{S}(V)$ with respect to $\psi$.
Let $\varphi \in \mathcal{S}(V)$ and $x \in V$. 
Then,
\begin{align*}
\omega(1,h)\varphi(x) &=\varphi(h^{-1}x),\\
\omega\left(t(a)
,1\right)\varphi(x)
&= \chi_V(a)|a|_F^{m/2}\varphi(xa),\\
\omega\left(n(b)
,1\right)\varphi(x)
&= \varphi(x)\psi(bQ[x]),\\
\omega\left(w
,1\right)\varphi(x)
&= \gamma_V^{-1}\int_{V}\varphi(y)\psi(-(x,y))dy
\end{align*}
for $a \in F^\times$, $b\in F$ and $h \in \O(V)(F)$.
Here, $dy$ is the self-dual measure on $V$ with respect to $\psi((x,y))$ given by
\[
dy=|\det(Q)|^{1/2}\prod_{j}dy_{j}
\]
where $dy_{j}$ is the self-dual measure on $F$ with respect to $\psi$.
\par

Following \cite{HK} \S5.1, we extend the Weil representation $\omega$.
We put
\[
R={\rm G}(\SL_2\times\O(V))
=\{(g,h)\in\GL_2\times\GO(V)|\det(g)=\nu(h)\}.
\]
For $h\in\GO(V)$ and $\varphi\in\mathcal{S}(V)$, we put
\[
L(h)\varphi(x)=|\nu(h)|_F^{-m/4}\varphi(h^{-1}x)
\]
for $x\in V$.
Then we define the Weil representation $\omega$ of $R(F)$ on $\mathcal{S}(V)$ by
\[
\omega(g,h)=\omega(g\cdot d(\det(g)^{-1}),1)\circ L(h)
\]
for $(g,h)\in R(F)$.

\subsection{Theta functions and theta lifts}
Let $V$ be a quadratic space over $\Q$ of even dimension $m$ and
$\psi=\psi_0$ the standard additive character of $\A_\Q$.
Let $\omega$ denote the Weil representation of $R(\A_\Q)$ on 
$\mathcal{S}(V(\A_\Q))$ with respect to $\psi$.
For $(g,h)\in R(\A_\Q)$ and $\varphi \in \mathcal{S}(V(\A_\Q))$, we put
\[
\theta (g,h;\varphi) = \sum_{x\in V(\Q)}\omega(g,h)\varphi(x).
\]
Then $\theta(g,h;\varphi)$ is an automorphic form on $R(\A_\Q)$.
The function $\theta(g,h;\varphi)$ is called a theta function.
\par

Let $f$ be a cusp form on $\GL_2(\A_\Q)$.
For $h \in \GO(V)(\A_\Q)$, choose $g' \in \GL_2(\A_\Q)$ such that 
$\det(g')=\nu(h)$, and put
\[
\theta(h;f,\varphi)=\int_{[\SL_2]}
\theta(gg',h;\varphi)f(gg')dg.
\]
Note that this integral does not depend on the choice of $g'$.
Then $\theta(f,\varphi)$ is an automorphic form on $\GO(V)(\A_\Q)$.
The function $\theta(f,\varphi)$ is called a theta lift.
Similarly, we define $\theta(f',\varphi)$ for a cusp form $f'$ on $\GO(V)(\A_\Q)$.
More precisely, see \cite{HK}.

\subsection{Change of polarizations}\label{change}
Let $F$ be a field of characteristic not $2$ and
$V$ a quadratic space over $F$ of even dimension $m$.
We assume that the matrix $Q\in \GL_m(F)$ associated to $V$ is of the form
\[
Q=
\begin{pmatrix}
0&0&1\\
0&Q_1&0\\
1&0&0
\end{pmatrix}
\]
for some $Q_1\in \GL_{m-2}(F)$.
Let $V_1=F^{m-2}$ be the quadratic space with bilinear form
\[
(v,w)={}^tvQ_1w.
\]
The associated quadratic form on $V_1$ is also denoted by $Q_1$.
For $v \in V_1$, we define an element $\ell(v)\in \O(V)(F)$ by
\[
\ell(v)=
\begin{pmatrix}
1&-{}^tvQ_1&-Q_1[v]\\
0&{\bf 1}_{m-2}&v\\
0&0&1
\end{pmatrix}.
\]
\par

Let $F=\Q_v$ (resp.~$F=\Q$) and $\psi=\psi_v$ (resp.~$\psi=\psi_0$) 
be the standard character of $\Q_v$ (resp.~$\A_\Q$).
For $\varphi \in \mathcal{S}(V)$ (resp.~$\varphi \in \mathcal{S}(V(\A_F))$),
we define the partial Fourier transform by the formula
\[
\hat{\varphi}(x_1;y_1,y_2)=
\int_{X_F}\varphi
\begin{pmatrix}
z\\x_1\\y_1
\end{pmatrix}
\psi(y_2z)dz
\]
for $x_1\in V_1$, $y_1,y_2\in \Q_v$ and $X_F=\Q_v$ 
(resp.~$x_1\in V_1(\A_\Q)$, $y_1,y_2\in \A_\Q$ and $X_F=\A_\Q$). 
Here $dz$ is the self-dual measure on $X_F$ with respect to $\psi$.
We define a representation $\hat{\omega}$ of $R(X_F)$ on 
$\mathcal{S}(V_1(X_F))\otimes \mathcal{S}(X_F^2)$ by
\[
\hat{\omega}(g,h)\hat{\varphi}=(\omega(g,h)\varphi)\hat{}.
\]
If $\hat{\varphi}=\varphi_1\otimes \varphi_2$ with $\varphi_1\in\mathcal{S}(V_1)(X_F)$ and
$\varphi_2\in\mathcal{S}(X_F^2)$, then
\[
\hat{\omega}(g,1)\hat{\varphi}(x_1;y_1,y_2)
=\omega_{V_1,\psi}(g,1)\varphi_1(x_1)\cdot \varphi_2((y_1,y_2)g)
\]
for $g\in\SL_2(X_F)$.
See also \cite{Ichino} \S4.2.
\par

Let $f$ be a cusp form on $\GL_2(\A_\Q)$ and $\varphi \in \mathcal{S}(V(\A_\Q))$.
For $\Xi \in V_1(\Q)$,
define the $\Xi$-th Fourier coefficient 
$\W_\Xi=\W_{\theta(f,\varphi),\Xi}$
of $\theta(f,\varphi)$ by
\[
\W_\Xi(h)=\int_{V_1(\Q)\backslash V_1(\A_\Q)}
\theta(\ell(v)h;f,\varphi)\overline{\psi((\Xi,v))}dv
\]
for $h\in\GO(V)(\A_\Q)$.
We need a modification of Lemma 4.2 of \cite{Ichino}.
\begin{lem}\label{Wxi}
\begin{enumerate}
\item
If $\Xi \not= 0$, then for $h \in \GO(V)(\A_\Q)$, we have
\begin{align*}
&\W_\Xi(h)=
\int_{N(\A_\Q)\backslash \SL_2(\A_\Q)}
\hat{\omega}\left(g
\cdot d(\nu(h))
,h\right)\hat{\varphi}(-\Xi;0,1)
W_{f,-Q_1[\Xi]}
\left(g
\cdot d(\nu(h))
\right)dg.
\end{align*}
Here, $W_{f,\xi}$ is the $\xi$-th Fourier coefficient of $f$ defined in Sect.~$\ref{GL2}$.
\item
For $h\in\GO(V)(\A_\Q)$, we have
\begin{align*}
&\W_0(h)
=\int_{[\SL_2]}
\sum_{x_1\in V_1(\Q)}\hat{\omega}\left(g
\cdot d(\nu(h))
,h\right)\hat{\varphi}(x_1;0,0)
f
\left(g
\cdot d(\nu(h))
\right)dg.
\end{align*}
\end{enumerate}
\end{lem}
\begin{proof}
For $h \in \O(V)(\A_\Q)$, apply Lemma 4.2 of \cite{Ichino} to $f|_{\SL_2(\A_\Q)}$.
In general, we put
$f'(g)=f\left(g
\cdot d(\nu(h))
\right)$.
Then by \cite{HK} Lemma 5.1.7, we have
\begin{align*}
\theta(\ell(v)h;f,\varphi)&=\theta(\ell(v);f',L(h)\varphi),\\
\hat{\omega}(g,1)(L(h)\varphi)\hat{}\ (-\Xi,y_1,y_2)
&=\hat{\omega}(g\cdot d(\nu(h)),h)\hat{\varphi}(-\Xi,y_1,y_2).
\end{align*}
On the other hand, we have
\[
W_{f',\xi}(g)=W_{f,\xi}\left(g
\cdot d(\nu(h))
\right).
\]
These equations imply the formulas.
\end{proof}

\section{Theta correspondence for $(\GU,\GL_2)$}\label{GU-GL}
In this section, we study the hermitian Maass lifting in terms of theta lifts.
The main result of this section is Proposition \ref{theta} which states that 
the automorphic form $\Lift$ on $\U(\A_\Q)$ given by the hermitian Maass lifts $\{F_\cc\}$ of $f$
 can be written by a certain theta lift via a homomorphism
$\phi\colon\GU(\A_\Q)\rightarrow \PGSO(4,2)(\A_\Q)$.
This proposition is a key point of the proof of Theorem \ref{main}.
\subsection{Preliminaries}\label{pre}
For a $\Q$-algebra $R$, we define the similitude unitary group $\GU$ and the similitude orthogonal groups $\GO(4,2)$ by
\begin{align*}
\GU(R)&=\{ g\in \GL_4(K\otimes R) |{^t\bar{g}}Jg=\lam(g)J,\ \lam(g)\in R^\times\},\\
\GO(4,2)&=\{g \in \GL_6|^tgQg=\nu(g)Q,\ \nu(g)\in \GL_1\}
\end{align*}
with
\[Q=
\begin{pmatrix}
&&&&&1\\
&&&&1&\\
&&2&0&&\\
&&0&2D&&\\
&1&&&&\\
1&&&&&\\
\end{pmatrix}
\in \GL_6\quad {\rm and} \quad 
J=\begin{pmatrix}
0 & -{\bf 1}_2 \\
{\bf 1}_2 & 0 \\
\end{pmatrix}\in \GL_4.
\]
The homomorphisms $\lam\colon\GU\rightarrow \GL_1$ and 
$\nu\colon\GO(4,2)\rightarrow \GL_1$
are the similitude norms.
We define the subgroups $\GSU$ of $\GU$ and $\GSO(4,2)$ of $\GO(4,2)$ by
\begin{align*}
\GSU=\{g \in \GU|\det(g)=\lam(g)^2\},\\
\GSO(4,2)=\{g \in \GO(4,2)|\det(g)=\nu(g)^3\}.
\end{align*}
For $\alpha \in {\rm Res}_{K/\Q}(\GL_1)$, we put
\[
r_{\alpha}=
\begin{pmatrix}
1&0&0&0\\
0&\alpha&0&0\\
0&0&1&0\\
0&0&0&\bar{\alpha}^{-1}
\end{pmatrix}
\in \U.
\]
Then we have an exact sequence
\[
\begin{CD}
1 @>>> \GL_1 @>\iota>> \GSU\rtimes {\rm Res}_{K/\Q}(\GL_1) @>\rho>>\GU @>>>1,
\end{CD}
\]
where 
\[
\iota(a)=(r_a,a^{-1}) \quad {\rm and}\quad \rho(h,\alpha)=hr_\alpha.
\]
\par

For each $p$, we put $\GU(\Z_\p)=\GU(\Q_p)\cap\GL_4(\OO\otimes_\Z \Z_p)$
and $\GSU(\Z_p)=\GU(\Z_p)\cap \GSU(\Q_p)$.
\begin{lem}\label{GU(Zp)}
For $g\in\GU(\Z_p)$, there exist $h\in\GSU(\Z_p)$ and $\alpha\in(\OO\otimes_\Z \Z_p)^\times$ such that
\[
g=hr_\alpha.
\]
\end{lem}
\begin{proof}
Let $\OO_\p=\OO\otimes_\Z\Z_p$.
We define the subgroups $H$ and $\OO_\p^1$ of $K^1(\Q_p)$ by
\begin{align*}
H=\{{\det(g)}/{\lam(g)^2}|g\in\GU(\Z_p)\}\quad{\rm and}\quad
\OO_\p^1=\{\alpha/\overline{\alpha}|\alpha\in\OO_\p^\times\}.
\end{align*}
Note that $H\supset\OO_\p^1$.
It is enough to show that $H=\OO_\p^1$.
If $g\in\GU(\Z_p)$, then
\[
g'=g\begin{pmatrix}
{\bf 1}_2&0\\0&\lam(g) \cdot {\bf 1}_2
\end{pmatrix}^{-1}\in\U(\Z_p)
\]
and we have ${\det(g')}/{\lam(g')^2}=\det(g')={\det(g)}/{\lam(g)^2}$.
Hence we find that $H=\det(\U(\Z_p))$.
\par

If $p$ is split in $K/\Q$, we have
\[
K\otimes\Q_p\cong\Q_p\times\Q_p,\quad \OO\otimes_\Z\Z_p\cong \Z_p\times\Z_p
\]
and the non-trivial element of ${\rm Gal}(K/\Q)$ acts by $(x,y)\mapsto(y,x)$.
If $g\in\U(\Z_p)$, then we have $\det(g)=(x,y)$ for some $x,y\in\Z_p^\times$.
Since $\det(g)\in K^1(\Q_p)$, we have $xy=1$, i.e., $y=x^{-1}$.
So we have $\det(g)=\alpha/\overline{\alpha}$ with $\alpha=(x,1)\in \OO_\p^\times$.
Therefore we have $H=\OO_\p^1$.
\par

If $p$ is inert in $K/\Q$, we can regard $p$ as a uniformizer of $K\otimes \Q_p$.
By Hilbert $90$, we have
$K^1(\Q_p)=\{\alpha/\overline{\alpha}|\alpha\in K(\Q_p)^\times=(K\otimes \Q_p)^\times\}$.
Hence we have $K^1(\Q_p)=\OO^1_\p$. So we get $H=\OO_\p^1$.
\par

Now we assume that $p$ is ramified in $K/\Q$.
Let $\p$ be the prime ideal of $K$ above $p$.
Then we have $\OO_\p/\p\cong \mathbb{F}_p$ and
the non-trivial element of ${\rm Gal}(K/\Q)$ acts on $\mathbb{F}_p$ by trivially.
So we may consider the reduction map
\[
\pi\colon\U(\Z_p)\rightarrow {\rm Sp}_2(\mathbb{F}_p).
\]
Let 
\[
B=\left\{\left(
\begin{array}{cc|cc}
*&*&*&*\\
0&*&*&*\\\hline
0&0&*&0\\
0&0&*&*
\end{array}
\right)\in\U
\right\}=TN
,\quad
{\bf B}=\left\{\left(
\begin{array}{cc|cc}
*&*&*&*\\
0&*&*&*\\\hline
0&0&*&0\\
0&0&*&*
\end{array}
\right)\in{\rm Sp}_2
\right\}
\]
be Borel subgroups of $\U$ and ${\rm Sp}_2$, respectively. 
We denote by $T$ the torus of diagonal matrices, and by $N$ the unipotent radical of $B$.
Let $I=\pi^{-1}({\bf B}(\mathbb{F}_p))\subset \U(\Z_p)$ denote the Iwahori subgroup of $\U(\Z_p)$.
Then, the Bruhat decomposition of ${\rm Sp}_2$ shows that there is a subset $W$ of $\U(\Z_p)$ such that
$\U(\Z_p)$ is generated by $I$ and $W$.
We can take $W$ such that $\det(w)=1$ for all $w\in W$.
Hence we have $H=\det(\U(\Z_p))=\det(I)$.
Moreover, by the Iwahori decomposition, we have
\[
I=(I\cap N^-)(I\cap T)(I\cap N),
\]
where $N^-$ is the opposite of $N$.
Since $\det(N)=\det(N^-)=\{1\}$, we have
$\det(I)=\det(I\cap T)$.
Clearly, this is equal to $\OO^1_\p$.
Therefore we have $H=\OO_p^1$.
\end{proof}
\par

As in \cite{Mori} \S 2.1, we set the six dimensional vector space over $\Q$ by
\[
V=\{B(x_1,\dots,x_6)\in\M_4(K)
|x_i \in \Q\ (1\leq i \leq 6)
\}
\]
where $B(x_1,\dots,x_6)$ is defined by
\[
B(x_1,\dots,x_6)=
\begin{pmatrix}
0&x_1&x_3+x_4\sqrt{-D}&x_2\\
-x_1&0&x_5&-x_3+x_4\sqrt{-D}\\
-x_3-x_4\sqrt{-D}&-x_5&0&x_6\\
-x_2&x_3-x_4\sqrt{-D}&-x_6&0\\
\end{pmatrix}.
\]
We define a mapping $\Psi \colon V \rightarrow \Q$ by
$\Psi(B)={\rm Tr}(BJ^t\overline{B}J)$.
As in \cite{Mori} \S 2.1, we have
\[
\Psi(B(x_1,\dots,x_6))=-4Q[{}^t(x_1,\dots,x_6)].
\]
As a basis for $V$, we may take
$e_i=B(\delta_{i,1},\dots,\delta_{i,6})$ for $1\leq i \leq 6$.
Then with respect to the basis $\{e_i|1\leq i \leq 6\}$, we may identify $\GO(V)$ as $\GO(4,2)$.
\par

We define homomorphisms
\begin{align*}
\phi_1\colon\GSU\rightarrow \GSO(V)\cong \GSO(4,2),
\quad \phi_1(g)\colon &V \ni B \to gB^tg\in V,\\
\phi_2\colon K^\times \rightarrow \GSO(V)\cong \GSO(4,2),
\quad \phi_2(\alpha)\colon &V \ni B \to \overline{\alpha}r_{\alpha}Br_{\alpha}\in V.
\end{align*}
Note that $\nu(\phi_1(g))=\lam(g)^2$ for $g\in \GSU$
and $\nu(\phi_2(\alpha))=N_{K/\Q}(\alpha)$ for $\alpha\in K^\times$. 
For $\alpha=x+y\sqrt{-D}\in K^\times$, we have
\[
\phi_2(\alpha)=
\begin{pmatrix}
N_{K/\Q}(\alpha)&&&&&\\
&1&&&&\\
&&x&Dy&&\\
&&-y&x&&\\
&&&&N_{K/\Q}(\alpha)&\\
&&&&&1
\end{pmatrix}
\in\GSO(4,2).
\]
We find that $\phi_1(\GSU(\Z_p))\subset\GSO(V)(\Z_p)$.
However, in general, it is not true that
\[
\alpha\in\OO\otimes\Z_p\Rightarrow \phi_2(\alpha)\in\GSO(V)(\Z_p).
\] 
For $g\in\GU$, we decompose $g=hr_\alpha$ with $h\in\GSU$ and $\alpha\in K^\times$.
Then the image of $\phi_1(h)\phi_2(\alpha)$ to $\PGSO(4,2)$ 
does not depend on the choice of $(h,\alpha)$.
We denote this element by $\phi(g)\in \PGSO(4,2)$. 
Note that for all $\alpha\in K^\times$, we have $\phi(\alpha {\bf 1}_4)=1$ in $\PGSO(4,2)$
and $\phi$ induces an isomorphism
\[
\phi\colon{\rm PGU}(2,2)\xrightarrow{\cong} \PGSO(4,2).
\]

\subsection{Theta lifts}
Let $\KK=\GU(\A_{\Q,\fin})\cap(\prod_p\GL_4(\OO_p))$, where $\OO_p=\OO\otimes \Z_p$.
Put $\KK_0=\KK\cap\U(\A_\Q)$.
We easily find that 
the canonical injection $\U(\A_\Q)\hookrightarrow \GU(\A_\Q)$ induces a bijection
\[
\U(\Q)\bs\U(\A_\Q)/\U(\R)\KK_0\rightarrow
\GU(\Q)\bs\GU(\A_\Q)/\GUplus\KK.
\]
By the Chebotarev density theorem, as a complete system of representatives of ${\it Cl}_K$, 
we can take integral ideals $\cc_1=\OO,\cc_2,\dots,\cc_h$ 
which are prime to $2D$.
We choose $t_i=(t_{i,\p})_\p\in\A_{K,\fin}^\times$ such that
$\ord_\p(t_{i,\p})=\ord_\p(\cc_i)$ for all prime ideals $\p$ of $K$. 
We put $\gamma_i=r_{t_i}\in\U(\A_\Q)$ and $C_i=N(\cc_i)\in\Z_{>0}$.
Then by Lemma 13.1 of \cite{Ikeda}, the set $\{\gamma_1,\dots,\gamma_{h}\}$
gives a complete system of $\U(\Q)\bs\U(\A_\Q)/\U(\R)\KK_0$.
\par

Let $f\in\SD$ be a normalized Hecke eigenform and 
$F_{\cc_i}$ the hermitian Maass lift of $f$ defined in Sect.~\ref{maass lift}.
Then by a result of Ikeda (\cite{Ikeda} Theorem 13.2 and Theorem 15.18), 
there is an automorphic form $\Lift$ on $\U(\A_\Q)$ defined by
\[
\Lift(u\gamma_ixk)=C_i^{-\kappa-1}(F_{\cc_i}\parallel_{2\kappa+2} x)(\ii)
=C_i^{-\kappa-1}F_{\cc_i}(x\pair{\ii})j(x,\ii)^{-2\kappa-2}(\det x)^{\kappa+1}
\]
for $u\in \U(\Q)$, $x\in\U(\R)$ and $k\in\KK_0$.
Here, we put $\ii=\I\cdot {\bf 1}_4\in\Ht$.
\par

Let $\ff$ be the cusp form on $\GL_2(\A_\Q)$ given by $f$.
We define $\varphi=\otimes_v \varphi_v\in\mathcal{S}(V(\A_\Q))$ as follows:
\begin{itemize}
\item If $v=p<\infty$ and $p\not=2$, then $\varphi_p$ is the characteristic function of $V(\Z_p)$.
\item If $v=2$, then $\varphi_2$ is the characteristic function of 
\[
L=\left\{
\begin{split}
&V(\Z_2)\cup(\Z_2e_1+\Z_2e_2+2^{-1}\Z_2^\times e_3+2^{-1}\Z_2^\times e_4+\Z_2e_5+\Z_2e_6) \iif D\ {\rm is\ odd},\\
&V(\Z_2)+2^{-1}\Z_2e_4\iif D\ {\rm is\ even}.
\end{split}\right.
\]
\item If $v=\infty$, then
\[
\varphi_\infty(x_1,\dots,x_6)=(-\I x_1+x_2-x_5+\I x_6)^{2\kappa+2}e^{-\pi(x_1^2+x_2^2+2x_3^2+2Dx_4^2+x_5^2+x_6^2)}.
\]
\end{itemize}
We consider the theta lift $\theta(\ff,\varphi)$ on $\GO(V)(\A_\Q)$.
Recall that $\theta(\ff,\varphi)$ is defined by
\[
\theta(h;\ff,\varphi)=\int_{[\SL_2]}\theta(\alpha\cdot \alpha_h,h;\varphi)\ff(\alpha\cdot \alpha_h)d\alpha
\]
for $h\in\GO(V)(\A_\Q)$, where $\alpha_h$ is an element in $\GL_2(\A_\Q)$ satisfying
$\det(\alpha_h)=\nu(h)$.
In particular, we can take $\alpha_h=d(\nu(h))$.
Note that $\nu(\GSO(V)(\A_\Q))=N_{K/\Q}(\A_K^\times)$.
Since the central character of $\ff$ is $\bchi=\chi_V$, by \cite{HK} Lemma 5.1.9, 
the center of $\GSO(4,2)$ acts on $\theta(\ff,\varphi)$ trivially, i.e., 
the quantity
\[
\theta(\phi(g);\ff,\varphi)=\int_{[\SL_2]}\theta(\alpha\cdot d(\nu(h_g)),h_g;\varphi)\ff(\alpha\cdot d(\nu(h_g)))d\alpha
\]
does not depend on the choice of a representative 
$h_g\in\GSO(V)(\A_\Q)$ of $\phi(g)\in \PGSO(V)(\A_\Q)$.
Hence we may consider the function $\theta(\phi(g);\ff,\varphi)$ on $\GU(\A_\Q)$.
\begin{prop}\label{theta}
For all $g\in \U(\A_\Q)$, we have
\[
\theta(\phi(g);\ff,\varphi)=
(\prod_{p|D}q_p^{-1})\cdot
2^{2\kappa+2}a_f(D)^{-1}\Lift(g),
\]
where $q_p=(\SL_2(\Z_p):K_0(D;\Z_p))$.
\end{prop}
The rest of this section is devoted to the proof of Proposition \ref{theta}.
The proof is similar to that of Lemma 7.1 of \cite{Ichino}.
To prove this proposition, we compare Fourier coefficients of $\theta(\phi(g);\ff,\varphi)$ with those of $\Lift$.

\subsection{Fourier coefficients of $\Lift$}\label{coef of Lift}
For $B\in\Her(\Q)$, we define the $B$-th Fourier coefficient $\W_{\FF,B}$ (resp.~$\W_B$) 
of $\FF=\Lift$ (resp.~$\theta(\ff,\varphi)$) by
\[
\W_{\FF,B}(g)=\int_{[\Her]}\FF(n(X)g)\overline{\psi(\Tr(BX))}dX
\]
for $g\in \U(\A_\Q)$ (resp.~
\[
\W_{B}(g)=\int_{[\Her]}\theta(\phi(n(X)g);\ff,\varphi)\overline{\psi(\Tr(BX))}dX
\]
for $g\in\GU(\A_\Q)$).
Here we put
\[
n(X)=
\begin{pmatrix}
{\bf 1}_2&X\\0&{\bf 1}_2
\end{pmatrix}
\in \U\quad{\rm for\ }X\in\Her.
\]
Then, we have
\[
\FF(g)=\sum_{B\in\Her(\Q)}\W_{\FF,B}(g)
\quad\text{and}\quad
\theta(\phi(g);\ff,\varphi)=\sum_{B\in\Her(\Q)}\W_{B}(g).
\]
\par

First, we compute $\W_{\FF,B}$ in Proposition \ref{coeff of Lift}. 
We define $\mathbb{K}=\{g\in \U(\R)| g\langle\ii\rangle=\ii\}$
and $\K_0=\K\cap\SU(\R)$.
\begin{lem}
We have
\[
\K=\left\{\left.
\begin{pmatrix}
\alpha&\beta\\-\beta&\alpha
\end{pmatrix}\right| 
\alpha, \beta\in \M_2(\C), 
{^t\overline{\alpha}\beta}={^t\overline{\beta}\alpha},\ 
{^t\overline{\alpha}\alpha}+{^t\overline{\beta}\beta}=1
\right\}.
\]
\end{lem}
\begin{proof}
Let 
$k=\begin{pmatrix}
A&B\\C&D
\end{pmatrix}
\in \K$ with $A,B,C,D\in \M_2(\C)$.
We put
\begin{align*}
X&=A\I+B=-C+D\I,\\
Y&=A\I+C=-B+D\I,\\
Z&=B+C=(D-A)\I.
\end{align*}
Then, it is easily to check that
${}^t \overline{X}X={\bf 1}_2$, ${}^t \overline{Y}Y={\bf 1}_2$ 
and ${}^t \overline{Z}Z={}^t \overline{Y}Y-{\bf 1}_2={\bf 0}_2$.
Hence $Z={\bf 0}_2$.
\end{proof}
For $X\in\Her(\R)$ and $A\in \GL_2(\C)$, we put
\[
n(X)=
\begin{pmatrix}
{\bf1}_2&X\\&{\bf1}_2
\end{pmatrix}
\quad\text{and}\quad
m(A)=
\begin{pmatrix}
A&\\&^t\overline{A}^{-1}
\end{pmatrix}.
\]
\begin{lem}
Any element $g$ in $\GUplus$ can be written as
\[g=
z{\bf 1}_4\cdot n(X)\cdot m(A)\cdot k\cdot r_t,
\]
where $z\in\C^\times$, $X\in\Her(\R)$, 
$A\in \GL_2(\C)$ with $\det(A)\in \R^\times$, 
$k\in \K_0$ and $t\in\C^\times$with $|t|=1$.
\end{lem}
\begin{proof}
Since $\GUplus$ is generated by its center and $\U(\R)$, 
it suffices to consider $g\in \U(\R)$.
Since $Z=g\langle \ii\rangle\in \Ht$, 
we fined that $(Z+{^t\overline{Z}})/2$ is a hermitian matrix and
$(Z-{^t\overline{Z}})/(2\I)$ is a positive definite hermitian matrix.
So there exist $A\in \GL_2(\C)$ and $X\in\Her(\R)$ such that
\[
g\langle \ii\rangle = A^t\overline{A}\ii+X
=\left(n(X)\cdot m(A)\right)\langle\ii\rangle.
\]
Hence we find
$k=\left(n(X)\cdot m(A)\right)^{-1}g\in\K$.
Since $\det(m(A))=\det(A)\overline{\det(A)}^{-1}$,
there exists $\theta\in\R$ such that
$e^{4\I\theta}=\det(m(A))$.
Then, we find that $\det(e^{-\I\theta}A)\in\R$ and $e^{\I\theta}{\bf1}_4\in\K$. 
Therefore we have the decomposition
\[
g=n(X)\cdot m(e^{-\I\theta}A)\cdot(e^{\I\theta}k).
\]
We put $k'=e^{\I\theta}k\in\K$.
Since $N_{\C/\R}(\det(k'))=1$, there exists $t\in \C^\times$ so that
$|t|=1$ and $\det(k')=t^2=\det(r_t)$.
Hence we have $k'(r_t)^{-1}\in\K_0$.
\end{proof}

The action of $r_t$ on $\FF$ is trivial for all $t\in\C^\times$ with $|t|=1$.
Let
\[
k=
\begin{pmatrix}
\alpha&\beta\\-\beta&\alpha
\end{pmatrix}
\in \K_0.
\]
Since $\det(k)=\det(\alpha+\I\beta)(\alpha-\I\beta)=1$,
the action of $k$ on $\FF$ is the scalar multiplication of
\[
j(k,\ii)^{-2\kappa-2}\det(k)^{\kappa+1}=\det(\alpha-\I\beta)^{-2\kappa-2}
=\det(\alpha+\I\beta)^{2\kappa+2}.
\]
\begin{prop}\label{coeff of Lift}
Let
$x_\infty=n(X)\cdot m(A)\in \U(\R)$,
where $X\in\Her(\R)$ and $A\in \GL_2(\C)$ with $\det(A)\in\R^\times$ .
Put $Y=A{^t{\overline{A}}}$.
Then, 
\[
\W_{\FF,B}(\gamma_ix_\infty)=
C_i^{-\kappa-1}\det(Y)^{\kappa+1}A_{F_{\cc_i}}(B)\exp(2\pi\I\Tr(B(X+Y\I))).
\]
Here, we set $A_{F_{\cc_i}}(B)=0$ if $B\not \in\Lam_2^{\cc_i}(\OO)^+$.
\end{prop}
\begin{proof}
We fix a $\Z$-basis $\{\lam_1,\lam_2\}$ of $\bar{\cc}_i$, i.e., 
$\bar{\cc}_i=\lam_1\Z+\lam_2\Z$.
As a coordinate of $\Her$, we use
\[
Z=\begin{pmatrix}
n&\lam_1a+\lam_2b\\\overline{\lam_1a+\lam_2b}&C_im
\end{pmatrix}
\in\Her.
\]
Then,
\begin{align*}
\W_{\FF,B}(\gamma_ix_\infty)
&=\int_{[\Her]}\FF(n(Z)\gamma_ix_\infty)\overline{\psi(\Tr(BZ))}dZ\\
&=\int_{(\hat{\Z}\times[0,1))^4}\FF(\gamma_in(d(t_i)^{-1}Zd(\overline{t_i})^{-1})x_\infty)\overline{\psi(\Tr(BZ))}dZ.
\end{align*}
Write $Z=Z_\fin+Z_\infty$ with $Z_\fin\in\Her(\A_{\Q,\fin})$ and $Z_\infty\in\Her(\R)$.
If $a,b,n,m\in\hat{\Z}$, then we find that $n(d(t_i)^{-1}Z_\fin d(\bar{t_i})^{-1})\in\KK_0$.
Hence $\W_{\FF,B}(\gamma_ix_\infty)$
is equal to
\[
\left(\int_{\hat{\Z}^4}\overline{\psi(\Tr(BZ_\fin))}dZ_\fin\right)
\left(\int_{[0,1)^4}\FF(\gamma_in(Z_\infty)x_\infty)\overline{\psi(\Tr(BZ_\infty))}dZ_\infty\right).
\]
The first integral is not vanish if and only if
$\Tr(BZ_\fin)\in\hat{\Z}$ whenever $a,b,n,m\in\hat{\Z}$.
This condition is equivalent to the condition that $B\in\Lam_2^{\cc_i}(\OO)$.
In this case, the first integral is equal to $1$.
\par

Since 
$\FF(\gamma_in(Z_\infty)x_\infty)=C_i^{-\kappa-1}\det(Y)^{\kappa+1}F_{\cc_i}(X+Y\I+Z_\infty)$,
the second integral is equal to 
\begin{align*}
C_i^{-\kappa-1}\det(Y)^{\kappa+1}
\sum_{H\in\Lam_2^{\cc_i}(\OO)^+}
&A_{F_{\cc_i}}(H)\exp(2\pi\I\Tr(H(X+Y\I)))\\
&\times\int_{[0,1)^4}\exp(2\pi\I\Tr((H-B)Z_\infty))dZ_\infty.
\end{align*}
Since $H-B\in\Lam_2^{\cc_i}(\OO)$, we find that
\[
\int_{[0,1)^4}\exp(2\pi\I\Tr((H-B)Z_\infty))dZ_\infty
=\left\{
\begin{split}
&1\iif B=H,\\
&0\iif B\not=H.
\end{split}
\right.
\]
This completes the proof.
\end{proof}

\subsection{Fourier coefficients of $\theta(\ff,\varphi)$}\label{coeff of theta}
Next we compute $\W_{B}(g)$. 
Let $V_0=\langle e_1,e_6 \rangle$ and $V_1=\pair{e_2,e_3,e_4,e_5}$
be subspaces of $V$.
We identify $V_0$ (resp.~$V_1$) with 
the space of column vectors $\Q^2$ (resp.~$\Q^4$) via this basis.
For 
\[
n(X)=
\begin{pmatrix}
{\bf 1}_2&X\\
0&{\bf 1}_2
\end{pmatrix}\in\SU
\quad { \rm with}\quad X=
\begin{pmatrix}
x_1&x_2+\sqrt{-D}x_3\\
x_2-\sqrt{-D}x_3&x_4
\end{pmatrix}\in \Her,
\]
we have
\[
\phi_1(n(X))=
\begin{pmatrix}
1&x_4&2x_2&2Dx_3&-x_1&x_1x_4-x_2^2-Dx_3^2\\
0&1   &0	 &0	     &0     &x_1\\
0&0   &1	 &0	     &0     &-x_2\\
0&0   &0	 &1	     &0     &-x_3\\
0&0   &0	 &0	     &1     &-x_4\\
0&0   &0	 &0	     &0     &1
\end{pmatrix}.
\]
There are two isomorphisms of the $\Q$-vector spaces
\begin{align*}
v\colon\Her(\Q)\rightarrow V_1(\Q)\cong\Q^4,\ 
&X=\begin{pmatrix}
x_1&x_2+\sqrt{-D}x_3\\
x_2-\sqrt{-D}x_3&x_4
\end{pmatrix}
\mapsto
\begin{pmatrix}
x_1\\-x_2\\-x_3\\-x_4
\end{pmatrix},\\
\beta\colon\Her(\Q)\rightarrow V_1(\Q)\cong\Q^4,\ 
&B=\begin{pmatrix}
b_1&b_2+\sqrt{-D}b_3\\
b_2-\sqrt{-D}b_3&b_4
\end{pmatrix}
\mapsto
\begin{pmatrix}
-b_4\\-b_2\\-b_3\\b_1
\end{pmatrix}.
\end{align*}
Put $v=v(X)$ and $\beta=\beta(B)$. 
Then we have
\[
\phi_1(n(X))=\ell(v),\quad
\Tr(BX)=(\beta,v),\quad
-Q_1[\beta]=\det(B).
\]
Hence we have $\W_{B}(g)=\W_{\beta}(h_g)$, 
where $\W_\beta$ is the $\beta$-th Fourier coefficient of $\theta(\ff,\varphi)$
defined in Sect.~\ref{change}.
\par

For $B=0$, by Paul's result \cite{P}, we get the following lemma.
\begin{lem}\label{W=0}
For all $g\in\GU(\A_\Q)$, we have
\[
\W_0(g)=0.
\]
\end{lem}
\begin{proof}
Let $\pi_{f}\cong\otimes_v'\pi_{f,v}$ be 
the automorphic representation of $\GL_2(\A_\Q)$ generated by $\ff$.
We write $\ff=\otimes_v'\ff_v$ with $\ff_v\in\pi_{f,v}$.
As we noted in Sect.~\ref{GL2}, we have $\pi_{f,\infty}|_{\SL_2(\R)}\cong \pi_f^+\oplus\pi_f^-$ and
$\ff_\infty\in\pi_f^+$,
where $\pi_f^+$ (resp.~$\pi_f^-$) is the holomorphic (resp.~anti-holomorphic) discrete series of $\SL_2(\R)$
with minimal weight $2\kappa+1$ (resp.~$-(2\kappa+1)$).
Let $g\in\GU(\A_\Q)$ and take a representative 
$h_g\in\GSO(V)(\A_\Q)$ of $\phi(g)\in\PGSO(V)(\A_\Q)$.
Fix $\varphi_0\in\mathcal{S}(V_0(\A_\Q))$ and $\varphi_1'\in\otimes_{p<\infty}'\mathcal{S}(V_1(\Q_p))$.
Consider the linear map
\begin{align*}
\Phi\colon\omega_{V_1(\R),\psi_\infty}\otimes \pi_f^+ \rightarrow \C,\ 
\varphi_1\otimes f_\infty\mapsto\W_0(g;\theta(f_\infty\otimes(\otimes_{p<\infty}'\ff_p),
\varphi_0\otimes(\varphi_1\otimes\varphi_1'))),
\end{align*}
where 
$\W_0(g;\theta(\ff',\varphi'))=\W_0(h_g;\theta(\ff',\varphi'))$
is the $0$-th Fourier coefficient of $\theta(\ff',\varphi')$ for 
$\ff'\in\pi$ and $\varphi'\in\mathcal{S}(V(\A_\Q))$.
Note that this integral does not depend on the choice of $h_g$
since $\ff'=f_\infty\otimes(\otimes_{p<\infty}'\ff_p)$ has the central character $\bchi$.
It suffices to show that $\Phi=0$.
\par

Now, we claim that
the linear map $\Phi\colon\omega_{V_1(\R),\psi_\infty}\otimes \pi_f^+\rightarrow \C$
is $\SL_2(\A_\Q)$-invariant.
Indeed, by Lemma \ref{Wxi} (2), we find that 
$\W_0(g;\theta(\ff',\varphi'))$
is equal to
\begin{align*}
\int_{[\SL_2]}
\sum_{x_1\in V_1(\Q)}\hat{\omega}\left(\alpha
\cdot d(\nu (h_g))
,h_g\right)
(\varphi')\hat{}\ (x_1;0,0)
\ff'\left(\alpha
\cdot d(\nu(h_g))
\right)d\alpha.
\end{align*}
We put
\[
\hat{\varphi}_0(y_1,y_2)=
\int_{\A_\Q}\varphi_0(z,y_1)\psi(y_2z)dz.
\]
Then we have
$(\varphi_0\otimes(\varphi_1\otimes\varphi_1'))\hat{}\ (x_1;y_1,y_2)
=\hat{\varphi}_0(y_1,y_2)\cdot(\varphi_1\otimes\varphi_1')(x_1)$.
As we noted in Sect.~\ref{change}, we then find that
\[
\hat{\omega}(\alpha_0,1)(\varphi_0\otimes(\varphi_1\otimes\varphi_1'))\hat{}\ (x_1;y)
=\hat{\varphi}_0(y\alpha_0)\cdot
[\omega_{V_1(\A_\Q),\psi}(\alpha_0,1)(\varphi_1\otimes\varphi_1')](x_1)
\]
for $y\in V_0(\A_\Q)\cong \A_\Q^2$ and $\alpha_0\in\SL_2(\A_\Q)$.
In particular, for $\alpha_0\in\SL_2(\R)$, we have
\begin{align*}
\hat{\omega}(\alpha_0,1)(\varphi_0\otimes(\varphi_1\otimes\varphi_1'))\hat{}\ (x_1;0,0)
&=\hat{\varphi}_0(0,0)\cdot
[\omega_{V_1(\A_\Q),\psi}(\alpha_0,1)(\varphi_1\otimes\varphi_1')](x_1)\\
&=(\varphi_0\otimes\left([\omega_{V_1(\R),\psi_\infty}(\alpha_0,1)\varphi_1]\otimes\varphi_1'\right))\hat{}\ (x_1;0,0).
\end{align*}
This implies that
\[
\Phi([\omega_{V_1(\R),\psi_\infty}(\alpha_0,1)\varphi_1]\otimes[\pi_f^+(\alpha_0)f_\infty])=\Phi(\varphi_1\otimes f_\infty)
\]
for all $\alpha_0\in\SL_2(\R)$ as desired.
\par

We put $\mathfrak{g}=\mathrm{Lie}(\SL_2(\R))\otimes_\R \C$ and $K=\SO(2)$.
By \cite{P} Theorem 15 and Corollary 23, we find that
\[
\mathrm{Hom}_{(\mathfrak{g},K)}(\omega_{V_1(\R),\psi_\infty}, (\pi_f^{+})^{\vee})=0,
\]
since the Harish-Chandra parameter of $(\pi_f^+)^{\vee}\cong\pi_f^-$ is $(-(2\kappa+1))$
and $\O(V_1)(\R)\cong\O(3,1)$.
Therefore, we have $\Phi=0$.
\end{proof}
\par

Next, we consider the case $B\not=0$.
By Lemma \ref{Wxi} (1), we have
\begin{align*}
\W_{B}(g)=\int_{N(\A_\Q)\bs\SL_2(\A_\Q)}\hat{\omega}(\alpha\cdot d(\nu(h_g)),h_g)
\hat{\varphi}(-\beta;0,1)W_{\ff,\det(B)}(\alpha\cdot d(\nu(h_g)))d\alpha,
\end{align*}
where $h_g=(h_{g,v})_v\in\GSO(V)(\A_\Q)$ is a representative of $\phi(g)\in\PGSO(V)(\A_\Q)$.
We put $\xi=\det(B)$. 
If $\xi=0$, we find that $\W_B(g)=0$ since $\ff$ is a cusp form.
If $\xi\not=0$, we find that $\W_B(g)=\prod_v\W_{B,v}(g_v)$, where
for $g\in \GU(\Q_v)$, we put
\[
\W_{B,v}(g)=\int_{N(\Q_v)\bs\SL_2(\Q_v)}\hat{\omega}(\alpha\cdot d(\nu(h_{g})),h_{g})
\hat{\varphi}_v(-\beta;0,1)W_{\ff,v}(a(\xi)\cdot\alpha\cdot d(\nu(h_{g})))d\alpha.
\]
The following lemma will be proved in the last section.
Here, we use the coordinate
\[
B=\begin{pmatrix}
b_1&b_2+\sqrt{-D}b_3\\
b_2-\sqrt{-D}b_3&b_4
\end{pmatrix}.
\]
\begin{lem}\label{local}
Fix $B\in \Her(\Q)$.
\begin{enumerate}
\item\label{arch}
Let $x_\infty=
(z{\bf 1}_4)\cdot
n(X)\cdot
m(A)\cdot
k\cdot r_t
\in \GUplus$
for  $z\in\C^\times$, $X\in\Her(\R)$, 
$A\in \GL_2(\C)$ with $\det(A)\in \R^\times$, $t\in \C^\times$ with $|t|=1$, and
\[
k=\begin{pmatrix}
\alpha&\beta\\-\beta&\alpha
\end{pmatrix}
\in\K_0=\K\cap\SU(\R).
\]
Put
$Y=A{^t\overline{A}}$ and $Z=X+Y\I$. Then, $\W_{B,\infty}\left(x_\infty\right)$ is equal to
\[
\left\{
\begin{aligned}
&2^{2\kappa+2}\det(Y)^{\kappa+1}\xi^{\kappa+(1/2)}e^{2\pi\I{\rm Tr}(BZ)}
\det(\alpha+\I\beta)^{2\kappa+2}\iif B>0,\\
&0&\quad&{\rm otherwise}.
\end{aligned}\right.
\]
\item\label{unram}
Assume that $p\nmid D$. We put
\[
m_0=\left\{
\begin{aligned}
&\min(\ord_p(b_1),\ord_p(b_2),\ord_p(b_3),\ord_p(b_4))\iif p\not=2,\\
&\min(\ord_p(b_1),\ord_p(b_2)+1,\ord_p(b_3)+1,\ord_p(b_4))\iif p=2.\\
\end{aligned}
\right.
\]
Then for all $k\in\GU(\Z_p)$, we have
\[
\W_{B,p}(k)=
|\xi|_p^{\kappa+1/2}\sum_{n=0}^{m_0}
 (p^n)^{2\kappa+1} a_f\left((\xi p^{-2n})_p\right).
\]
\item\label{ram}
Assume that $p\mid D$. We put
\[
m_0=\left\{
\begin{aligned}
&\min(\ord_p(b_1),\ord_p(b_2),\ord_p(b_3)+1,\ord_p(b_4))\iif \ord_p(D)=1,\\
&\min(\ord_p(b_1),\ord_p(b_2)+1,\ord_p(b_3)+2,\ord_p(b_4))\iif \ord_p(D)=2,\\
&\min(\ord_p(b_1),\ord_p(b_2)+1,\ord_p(b_3)+3,\ord_p(b_4))\iif \ord_p(D)=3.\\
\end{aligned}
\right.
\]
Then for all $k\in\GU(\Z_p)$, we have
\[
\W_{B,p}(k)=
q_p^{-1}a_f(D_p)^{-1}|\xi|_p^{\kappa+1/2}
\sum_{n=0}^{m_0}(p^n)^{2\kappa+1}X_{n,p},
\]
where $q_p=(\SL_2(\Z_p):K_0(D;\Z_p))$ and
\[
X_{n,p}=\left\{
\begin{aligned}
&1&\quad&{\rm if\ }(D\xi p^{-2n})_p=1,\\
&a_f\left((D\xi p^{-2n})_p\right)+\underline{\chi}_p(-\xi)
\overline{a_f\left((D\xi p^{-2n})_p\right)}
&\quad&{\rm otherwise}.
\end{aligned}\right.
\]
\item\label{divC}
Assume that $\ord_p(C_i)\not=0$. 
Write $t_{i,p}=x_p\otimes 1+y_p\otimes\sqrt{-D}\in \Q_{p}\otimes K$.
Then for all $k\in\GU(\Z_p)$, we have
\[
\W_{B,p}(\gamma_{i,p}k)=|C_i|_p^{\kappa+1}|\xi|_{p}^{\kappa+1/2}
\sum_{n=0}^\infty \hat{\varphi}_p(-\phi'(t_{i,p})\beta p^{-n};0,p^n)(p^n)^{2\kappa+1}
a_f\left((\xi C_i p^{-2n})_{p}\right).
\]
Here $\phi'(t_{i,p})\in \GL(V_1)\cong\GL_4(\Q_p)$ is given by
\[
\phi'(t_{i,p})=
\begin{pmatrix}
N_{K_p/\Q_p}(t_{i,p})&&&\\
&x_p&-Dy_p&\\
&y_p&x_p&\\
&&&1
\end{pmatrix},
\]
which makes the following diagram commutative.
\[
\begin{CD}
\Her(\Q_p) @>\beta>> V_1(\Q_p)\\
@Vd(\bar{t}_{i,p})Xd(t_{i,p})VV @VV\phi'(t_{i,p})V\\
\Her(\Q_p) @>\beta>> V_1(\Q_p)@..
\end{CD}
\]
\end{enumerate}
\end{lem}

\subsection{Proof of Proposition \ref{theta}}
The ring of integers $\OO$ of $K=\Q(\sD)$ is given by
\[
\OO=\left\{
\begin{aligned}
&\Z+\frac{1+\sD}{2}\Z\iif D\ {\rm is\ odd},\\
&\Z+\frac{\sD}{2}\Z \iif D\ {\rm is\ even}.
\end{aligned}\right.
\]
So we change the coordinate $B=B(h_1,\dots,h_4)$ of $\Her(\A_\Q)$ such that
$d(\bar{t}_i)Bd(t_i)$ is equal to
\[
\left\{
\begin{aligned}
&\begin{pmatrix}
h_1&(h_2+h_3/2)/\sD+h_3/2\\
-(h_2+h_3/2)/\sD+h_3/2&h_4
\end{pmatrix}\iif D\ {\rm is\ odd},\\
&\begin{pmatrix}
h_1&h_2/\sD+h_3/2\\
-h_2/\sD+h_3/2&h_4
\end{pmatrix}\iif D\ {\rm is\ even}
\end{aligned}
\right.
\]
for $h_1,\dots,h_4\in\A_\Q$.
Note that $B\in\Lam_2^{\cc_i}(\OO)$ if and only if 
$\ord_p(h_i)\geq 0$ for all $i=1,\dots,4$ and all prime $p$.
\begin{lem}
We put $m_i=\ord_p(b_i)$ and $l_i=\ord_p(h_i)$.
\begin{enumerate}
\item
If $p\nmid 2DC_i$, then we have
\[
\min(m_1,m_2,m_3,m_4)=\min(l_1,l_2,l_3,l_4).
\]
\item
If $p\not=2$ and $p\mid D$, then we have
\[
\min(m_1,m_2,m_3+1,m_4)=\min(l_1,l_2,l_3,l_4).
\]
\item
Let $p=2\nmid D$. Unless 
$l_2<\min(l_1,l_3,l_4)$, 
we have
\[
\min(m_1,m_2+1,m_3+1,m_4)=\min(l_1,l_2,l_3,l_4).
\]
In this case, we have
\[
\min(m_1,m_2+1,m_3+1,m_4)=l_2+1 ,\quad \min(l_1,l_2,l_3,l_4)=l_2,
\]
and we have $\xi/2^{2(l_2+1)}\not\in\Z_2$.
\item
If $p=2\mid D$, then we have 
\[
\min(m_1,m_2+1,m_3+\ord_2(D),m_4)=\min(l_1,l_2,l_3,l_4).
\]
\item
If $p\mid C_i$, we have
\[
\sum_{n=0}^\infty \hat{\varphi}_p(-\phi'(t_{i,p})\beta p^{-n};0,p^n)(p^n)^{2\kappa+1}
a_f\left((\xi C p^{-2n})_{p}\right)
=\sum_{n=0}^{\min(l_i)} (p^n)^{2\kappa+1}a_f\left((\xi Cp^{-2n})_{p}\right).
\]
\end{enumerate}
\end{lem}
\begin{proof}
(5) is easily followed from the last assertion of Lemma \ref{local} (\ref{divC}).
If $D$ is even, the assertions (1) to (4) are obvious.
So we assume $D$ is odd.
\par

If $p\nmid 2DC_i$, then we have $m_1=l_1$, $m_2=l_3$, $m_4=l_4$ and
$m_3=\ord_p(h_2+h_3/2)$.
If $l_2\geq l_3$, then we have
\[
m_3=\ord_p(h_2+h_3/2)\geq l_3=m_2.
\]
Hence we have
\begin{align*}
\min(m_1,m_2,m_3,m_4)&=\min(m_1,m_2,m_4)
=\min(l_1,l_3,l_4)=\min(l_1,l_2,l_3,l_4).
\end{align*}
If $l_2<l_3$, then we have $m_3=l_2$.
So we have (1).
\par

If $p\not=2$ and $p\mid D$, then we have
$m_1=l_1$, $m_2=l_3$, $m_4=l_4$ and $m_3+1=\ord_p(h_2+h_3/2)$.
Then, the proof of (2) is similar to that of (1).
\par

If $p=2\nmid D$, then we have 
$m_1=l_1$, $m_2+1=l_3$, $m_4=l_4$ and $m_3+1=\ord_p(2h_2+h_3)$.
If $l_2+1> l_3$, then we have
\[
m_3+1= l_3=m_2+1.
\]
Hence we have
\begin{align*}
\min(m_1,m_2+1,m_3+1,m_4)&=\min(m_1,m_2+1,m_4)\\
&=\min(l_1,l_3,l_4)=\min(l_1,l_2,l_3,l_4).
\end{align*}
If $l_2+1< l_3$, then we have
\[
m_3+1= l_2+1<m_2+1.
\]
Hence we have
\begin{align*}
\min(m_1,m_2+1,m_3+1,m_4)&=\min(m_1,m_3+1,m_4)
=\min(l_1,l_2+1,l_4),\\
\min(l_1,l_2,l_3,l_4)&=\min(l_1,l_2,l_4).
\end{align*}
So if
\[
\min(m_1,m_2+1,m_3+1,m_4)\not=\min(l_1,l_2,l_3,l_4),
\]
then we have $l_2<\min(l_1,l_4)$ and 
\begin{align*}
\min(m_1,m_2+1,m_3+1,m_4)=l_2+1\quad {\rm and}\quad 
\min(l_1,l_2,l_3,l_4)=l_2.
\end{align*}
In this case, we have $l_1+l_4>2l_2$, $2l_3-2>2l_2$, and
$\ord_2(h_2+h_3/2)^2=2l_2$. 
Now, we find that
\[
\xi=\det(H)=h_1h_4-\frac{h_3^2}{4}-\frac{1}{D}\left(h_2+\frac{h_3}{2}\right)^2.
\]
Hence we have $\ord_2(\xi)=2l_2<2(l_2+1)$.
If $l_2+1=l_3$, then we find $m_3+1\geq l_2+1$ and $l_2<l_3$. 
So we have
\begin{align*}
\min(m_1,m_2+1,m_3+1,m_4)&=\min(l_1,l_2+1,l_4),\\
\min(l_1,l_2,l_3,l_4)&=\min(l_1,l_2,l_4).
\end{align*}
So if
\[
\min(m_1,m_2+1,m_3+1,m_4)\not=\min(l_1,l_2,l_3,l_4),
\]
then we have $l_2<\min(l_1,l_4)$ and 
\begin{align*}
\min(m_1,m_2+1,m_3+1,m_4)=l_2+1\quad {\rm and}\quad 
\min(l_1,l_2,l_3,l_4)=l_2.
\end{align*}
Since the rational numbers $h_2,h_3$ satisfy that $\ord_2(h_2)=\ord_2(h_3/2)$,
we have 
\[
\ord_2(h_2+h_3/2)>\ord_2(h_2)=l_2.
\]
However we have $\ord_2(h_1h_4)=l_1+l_4>2l_2$ and $\ord_2(h_3^2/4)=2l_3-2=2l_2$.
Hence we have $\ord_2(\xi)=2l_2<2(l_2+1)$.
We get (3). 
This completes the proof.
\end{proof}

Now, we start to prove Proposition \ref{theta}.
By Sect.~\ref{coef of Lift} and Lemma \ref{local}, 
the compact group $\K\cdot\KK_0\subset \U(\A_\Q)$ acts on 
$\FF=\Lift$ and on $\theta(\ff,\varphi)$ by the same character.
Hence, it suffices to show that
\[
\W_B(\gamma_ix_\infty)=\left(\prod_{p|D}q_p^{-1}\right)
2^{2\kappa+2}a_f(D)^{-1}\W_{\FF,B}(\gamma_ix_\infty)
\]
for all $B\in\Her(\Q)$ and $x_\infty=n(X)\cdot m(A)$,
where $X\in\Her(\R)$, $A\in\GL_2(\C)$ with $\det(A)\in\R^\times$ and $q_p=(\SL_2(\Z_p):K_0(D;\Z_p))$.
We may assume that $B\in\Lam_2^{\cc_i}(\OO)^+$.
Then $\W_B(\gamma_ix_\infty)$ is equal to
\begin{align*}
&2^{2\kappa+2}\det(Y)^{\kappa+1}e^{2\pi\I\Tr(BZ)}C_i^{-\kappa-1}
\left(\prod_{p|D}q_p^{-1}\right)a_f(D)^{-1}\\
&\times
\left(\prod_{p\nmid D}\sum_{n=0}^{\min(l_i)}(p^n)^{2\kappa+1}a_f\left((\xi C_ip^{-2n})_{p}\right)\right)
\left(\prod_{p\mid D}\sum_{n=0}^{\min(l_i)}(p^n)^{2\kappa+1}X_{n,p}\right)\\
&=2^{2\kappa+2}\det(Y)^{\kappa+1}e^{2\pi\I\Tr(BZ)}C_i^{-\kappa-1}
\left(\prod_{p|D}q_p^{-1}\right)a_f(D)^{-1}
\\&\times
\sum_{d\mid (h_1,h_2,h_3,h_4)}d^{2\kappa+1}\alpha_{F_{\cc_i}}\left(\frac{C_iD\det(B)}{d^2}\right)\\
&=2^{2\kappa+2}\left(\prod_{p|D}q_p^{-1}\right)
a_f(D)^{-1}\W_{\FF,B}(\gamma_ix_\infty).
\end{align*}
This completes the proof of Proposition \ref{theta}
using Lemma \ref{local}.
\begin{cor}\label{Cor}
The automorphic form $\Lift$ on $\U(\A_\Q)$ can be extended to $\GU(\A_\Q)$ by
\[
\Lift(u r_t x k)=C^{-\kappa-1}(F_{\cc}\parallel_{2\kappa+2} x)(\ii)
=C^{-\kappa-1}F_{\cc}(x\pair{\ii})j(x,\ii)^{-2\kappa-2}(\det x)^{\kappa+1}
\]
for all $u\in \GU(\Q)$, $x\in\GUplus$, $k\in\KK$ and
$t\in\A_{K,\fin}$ such that the ideal $\cc$ given by $t$ is in $J_K^D$. 
Here, $C=N(\cc)\in\Q_{>0}$ is the ideal norm of $\cc$.
Moreover, it satisfies
\[
\theta(\phi(g);\ff,\varphi)=
\left(\prod_{p|D}q_p^{-1}\right)
2^{2\kappa+2}a_f(D)^{-1}\Lift(g)
\]
for all $g\in\GU(\A_\Q)$ with $q_p=(\SL_2(\Z_p):K_0(D;\Z_p))$.
\end{cor}
\begin{proof}
This follows from Proposition \ref{theta} and Lemma \ref{Lem1}.
\end{proof}

\section{Theta correspondence for $(\GL_2,\GO(2,2))$}\label{GL-GO(2,2)}
In this section, we recall the theta correspondence for $(\GL_2,\GO(2,2))$ 
following \cite{Ichino} \S6.
For the proof of Theorem \ref{main}, 
we need Lemma \ref{G->g} which calculates a certain theta lift on $\GL_2(\A_\Q)$.
\par

Let $\M_2$ be the vector space of $2 \times 2$ matrices.
We regard $\M_2$ as the quadratic space with the quadratic form
\[
Q[x]=\det(x).
\]
Let
\[
\mathrm{G}(\SL_2\times\SL_2)=\{(h_1,h_2)\in\GL_2\times\GL_2|\det(h_1)=\det(h_2)\}.
\]
For $(h_1,h_2)\in \mathrm{G}(\SL_2\times\SL_2)$, 
we put $\Delta(h_1,h_2)=\det(h_1)=\det(h_2)$.
Recall that there is a following exact commutative diagram:
\[\begin{CD}
1@>>>\GL_1@>\iota>>\GL_2\times\GL_2@>\rho>>\GSO(\M_2)@>>>1\\
@.@|@AAA@AAA\\
1@>>>\GL_1@>\iota>>\mathrm{G}(\SL_2\times\SL_2)@>\rho>>\SO(\M_2)@>>>1,
\end{CD}\]
where 
\[
\iota(a)=(a{\bf 1}_2,a{\bf 1}_2)
\quad{\rm and}\quad
\rho(h_1,h_2)x=h_1xh_2^{-1},
\]
for $a\in\GL_1$, $h_1,h_2\in\GL_2$ and $x\in\M_2$.
\par

Let $V'$ be the subspace of $V$ generated by $\{e_1,e_2,e_5,e_6\}$.
Then there is an isomorphism of quadratic spaces over $\Q$ as follows:
\begin{align*}
V'(\Q)\xrightarrow{\cong} \M_2(\Q),\quad
x_1e_1+x_2e_2+x_5e_5+x_6e_6
\mapsto
\begin{pmatrix}
x_2&-x_1\\x_6&x_5
\end{pmatrix}.
\end{align*}
Via this isomorphism, we regard $\rho$ as a map $\rho\colon\GL_2\times\GL_2\rightarrow\GSO(V')$.
\par

Let $\mu_2$ be the subgroup of $\O(\M_2)$ generated be the involution $*$ on $\M_2$ given by
\[
x=\begin{pmatrix}
x_1&x_2\\x_3&x_4
\end{pmatrix}
\mapsto
x^*=\begin{pmatrix}
x_4&-x_2\\-x_3&x_1
\end{pmatrix}.
\]
Via the isomorphism $\M_2\cong V'$, the involution $*$ on $\M_2$ corresponds to
an element $h_0'\in\O(V')$ given by
\[
x_1\mapsto -x_1,\ x_2\mapsto x_5,\ 
x_5\mapsto x_2,\ x_6\mapsto -x_6.
\]
\par

For $\alpha\in K^\times$, we have $\phi_2(\alpha)V'\subset V'$.
We put $\phi_2'(\alpha)=\phi_2(\alpha)|_{V'}$ for $\alpha\in K^\times$.
Then we find that
\[
\phi_2'(\alpha)
=
\begin{pmatrix}
N_{K/\Q}(\alpha)&&&\\
&1&&\\
&&N_{K/\Q}(\alpha)&\\
&&&1
\end{pmatrix}
=\rho\left({\bf 1}_2,d(N_{K/\Q}(\alpha)^{-1})
\right)
\in\GSO(V').
\]
We extend $g\in\SO(V')$ to $\SO(V)$ by $ge_3=e_3$ and $ge_4=e_4$.
We define the inclusion $\tau\colon\mathrm{G}(\SL_2\times\SL_2)\hookrightarrow\GSU$ by
\[
\tau\left(
\begin{pmatrix}
a_1&b_1\\c_1&d_1
\end{pmatrix}
,
\begin{pmatrix}
a_2&b_2\\c_2&d_2
\end{pmatrix}
\right)\mapsto
\begin{pmatrix}
a_1&0&b_1&0\\
0&a_2&0&b_2\\
c_1&0&d_1&0\\
0&c_2&0&d_2
\end{pmatrix}.
\]
Then we find that the following diagram is commutative:
\[
\begin{CD}
\mathrm{G}(\SL_2\times\SL_2) @>(\rho,\Delta)>> \SO(V)\times \GL_1\\
@V\tau VV @VVV\\
\GSU@>\phi_1>> \GSO(V).
\end{CD}
\]
Here, the map $\SO(V)\times\GL_1\rightarrow \GSO(V)$ is the multiplication $(h,a)\mapsto ah$.
\par

For a normalized Hecke eigenform $g\in S_{2\kappa+2}(\SL_2(\Z))$, 
let $\g$ denote the cusp form on $\GL_2(\A_\Q)$ given by $g$.
Following \cite{Ichino} \S6.3, we extend the cusp form $\g\otimes\g$ on $\GSO(\M_2)(\A_\Q)$
to a cusp form $\G$ on $\GO(\M_2)(\A_\Q)$ by
$\G(hh')=\G(h)$ for $h\in\GO(\M_2)(\A_\Q)$ and $h'\in\mu_2(\A_\Q)$.
\par

We define $\varphi'=\otimes \varphi'_v\in\mathcal{S}(V'(\A_\Q))$ as follows:
\begin{itemize}
\item If $v=p<\infty$, then $\varphi'_p$ is the characteristic function of $V'(\Z_p)$.
\item If $v=\infty$, then
\[
\varphi'_\infty(x_1e_1+x_2e_2+x_5e_5+x_6e_6)
=(-\I x_1+x_2-x_5+\I x_6)^{2\kappa+2}e^{-\pi(x_1^2+x_2^2+x_5^2+x_6^2)}.
\]
\end{itemize}
We may regard $\G$ as a cusp form on $\GO(V')(\A_\Q)$ via the isomorphism $\M_2\cong V'$.
So we can consider the theta lift $\theta(\overline{\G},\varphi')$ on $\GL_2(\A_\Q)$.
The following lemma is Lemma 6.3 in \cite{Ichino}.
\begin{lem}\label{G->g}
For all $\alpha\in\GL_2(\A_\Q)$, we have
\[
\theta(\alpha;\overline{\G},\varphi')=2^{2\kappa+1}\xi_\Q(2)^{-2}\langle g,g\rangle \overline{\g(\alpha)}.
\]
\end{lem}

\section{Theta correspondence for $(\SL_2,\O(2))$ and the Siegel--Weil formula}\label{SW}
In this section, we study the theta correspondence for $(\SL_2,\O(2))$.
For the proof of Theorem \ref{main}, we need the Siegel--Weil formula (Proposition \ref{1->E})
which calculates a certain theta lift on $\SL_2(\A_\Q)$.
\par

We regard $K$ as the quadratic space over $\Q$ with the bilinear form
\[
(x,y)=\Tr_{K/\Q}(x\bar{y}).
\]
Then the associated quadratic form is given by
\[
Q[x]=N_{K/\Q}(x).
\]
\par

Let $V''$ be the subspace of $V$ generated by $\{e_3,e_4\}$.
Then there is an isomorphism of quadratic spaces over $\Q$ as follows:
\[
\ell\colon K\xrightarrow{\cong} V'',\ x+y\sD \mapsto xe_3+ye_4.
\]
This isomorphism induces an isomorphism of algebraic groups as follows:
\[
K^\times\xrightarrow{\cong} \GSO(V'')\cong\GSO(2),\ x+y\sD\mapsto
\begin{pmatrix}
x&-Dy\\
y&x
\end{pmatrix}.
\]
The restriction of this map gives an isomorphism
$K^1\xrightarrow{\cong}\SO(V'')$.
\par

For $\alpha\in K^\times$, we have $\phi_2(\alpha)V''\subset V''$.
We put $\phi_2''(\alpha)=\phi_2(\alpha)|_{V''}$.
Then for $\alpha\in K^\times$ and $x\in\ K$, we have
\[
\phi_2''(\alpha)\ell(x)=\ell(\overline{\alpha}x),
\]
i.e., the following diagram is commutative:
\[
\begin{CD}
K@>\ell>>V''\\
@V\overline{\alpha}VV@VV\phi_2''(\alpha)V\\
K@>\ell>>V''.
\end{CD}
\]
\par

For $\varphi''\in\mathcal{S}(V''(\A_\Q))$, we define
\[
\Phi(\alpha,s)=\Phi(\alpha,s,\varphi'')=[\omega(\alpha,1)\varphi''](0)|a(\alpha)|^s
\]
where if $\alpha=n(x)t(a)k\in\SL_2(\A_\Q)$
with $x\in\A_\Q$, $a\in\A_\Q^\times$ and $k\in\SL_2(\hat{\Z})\SO(2)$, 
then we put 
\[
|a(\alpha)|=|a|.
\]
Note that the quantity $|a(\alpha)|$ is well-defined.
It satisfies that
\[
\Phi(n(x)t(a)g,s)=\chi_{V''}(a)|a|^{s+1}\Phi(g,s)
\]
for all $x\in\A_\Q$, $a\in\A_\Q^\times$ and $g\in\SL_2(\A_\Q)$,
i.e., $\Phi\in\mathrm{Ind}_{B(\A_\Q)}^{\SL_2(\A_\Q)}(\chi_{V''}|\cdot|^s)$.
Note that $\Phi(\alpha,s)=\prod_v\Phi_v(\alpha_v,s)$.
\par

We define the Eisenstein series $E(\alpha,s)$ by
\[
E(\alpha,s)=\sum_{\gamma\in B(\Q)\backslash \SL_2(\Q)}
\Phi(\gamma\alpha,s).
\]
This is absolutely convergent for $\re(s)>1$.
The Eisenstein series
$E(\alpha,s)$ has a meromorphic continuation to the whole $s$-plane 
and is holomorphic at $s=0$.
The following proposition is Main Theorem of \cite{KR}.
\begin{prop}[Siegel-Weil formula]\label{1->E}
For $\varphi''\in\mathcal{S}(V''(\A_\Q))$, 
we have
\[
\theta(\alpha;{\bf1}_{\O(V'')(\A_\Q)},\varphi'')
=\frac{1}{2}E(\alpha,0).
\]
\end{prop}
\par

We define $\varphi''=\otimes_v\varphi_v''\in\mathcal{S}(V''(\A_\Q))$ as follows:
\begin{itemize}
\item If $v=p<\infty$ and $p\not=2$, then $\varphi''_p$ is the characteristic function of $V''(\Z_p)$.
\item If $v=2$, then $\varphi''_2$ is the characteristic function of 
\[
\left\{
\begin{aligned}
&V''(\Z_2)\cup(2^{-1}\Z_2^\times e_3+2^{-1}\Z_2^\times e_4) \iif D\ {\rm is\ odd},\\
&\Z_2e_3+2^{-1}\Z_2e_4\iif D\ {\rm is\ even}.
\end{aligned}\right.
\]
\item If $v=\infty$, then
\[
\varphi''_\infty(x)=e^{-2\pi Q''[x]}.
\]
\end{itemize}
Here, $Q''$ is the quadratic form of $V''$ given by
\[
Q''[x_3e_3+x_4e_4]=x_3^2+Dx_4^2.
\]
\par

We will apply Proposition \ref{1->E} to $L(\phi_2''(t))\varphi''$ for some $t\in\A_{K,\fin}^\times$.
\begin{lem}\label{invariant}
Via the isomorphism $\ell\colon K\cong V''$, the function $\varphi_p''$ is the characteristic function
of the maximal compact subring $\OO\otimes_{\Z}\Z_p$ of $K\otimes \Q_p$ for each $p$.
In particular, for $\alpha\in(\OO\otimes_\Z\Z_p)^\times$, we have
\[
L(\phi_2''(\alpha))\varphi_p''=\varphi_p''.
\]
\end{lem}
\begin{proof}
This is obvious.
\end{proof}
\par

We define 
\[
\hat{\OO}_K^\times=\prod_\p\OO_{K_\p}^\times,\quad 
\hat{\OO}_K^1=\{\beta\overline{\beta}^{-1}|\beta\in\hat{\OO}_K^\times\}
\]
and
\[
\C^1=K^1(\R)=\{x+y\sqrt{-D}\in\C|x,y\in\R, x^2+Dy^2=1\}.
\]
Note that $\C^1=\{\alpha\in\C^\times|\ |\alpha|=1\}$.
We easily find that $L(\phi_2''(\alpha))\varphi_\infty''=\varphi_\infty''$ for $\alpha\in\C^1$.
This fact and Lemma \ref{invariant} imply that the function
\[
\alpha\mapsto\theta(g
,\phi_2''(\alpha);\varphi'')
\]
on $[K^1]$ is right $\hat{\OO}_K^1\C^1$-invariant for all $g\in\SL_2(\A_\Q)$.

\section{Proof of Theorem \ref{main}}\label{main proof}
In this section, we prove Theorem \ref{main} 
using Lemma \ref{local} and Lemma \ref{I and J},
which will be proved in the last section.

\subsection{Seesaw identity}
Let $\cc\in J_\OO^{D}$ and put $C=N(\cc)\in\Z_{>0}$.
Assume that $C$ is a square free integer.
Let $t\in\A_{K,\fin}^\times$ such that $\ord_\p(t_\p)=\ord_\p(\cc)$ for each prime ideal $\p$ of $K$.
Note that $CN_{K/\Q}(t_p)^{-1}\in\Z_p^\times$ for each $p$.
We consider the integral
\[
I(\cc)=\int_{[\O(V')]}\int_{[\O(V'')]}\theta(h'h''\phi_2(t);\ff,\varphi)\overline{\G(h'\phi_2'(t))}dh''dh'.
\]
Note that this integral does not depend on the choice of $t\in\A_{K,\fin}^\times$.
First, using Proposition \ref{theta}, we get the following expression of $I(\cc)$ in terms of the period integrals.
\begin{prop}\label{period}
The integral $I(\cc)$ is equal to
\[
2^{-1}\xi_\Q(2)^{-2}\left(\prod_{p\mid D}q_p^{-1}\right)
2^{2\kappa+1}a_f(D)^{-1}\frac{1}{\# ({\it Cl}_K^2)}\sum_{[\mathfrak{a}]\in {\it Cl}_K^2}
\frac{\pair{F_{\mathfrak{a}\cc}|_{\Ha\times\Ha},g\times g_{AC}}\pair{g,g}}{\pair{g_{AC},g_{AC}}}.
\]
Here, we put $q_p=(\SL_2(\Z_p):K_0(D;\Z_p))$,
$A=N(\mathfrak{a})$ and ${\it Cl}_K^2=\{[\mathfrak{b}]^2|\ [\mathfrak{b}]\in {\it Cl}_K\}$.
\end{prop}
\begin{proof}
We define $h'_0\in\O(V')$ by
\[
x_1\mapsto -x_1,\ x_2\mapsto x_5,\ 
x_5\mapsto x_2,\ x_6\mapsto -x_6,
\] 
and $h''_0\in\O(V'')$ by
\[
x_3\mapsto -x_3,\ x_4\mapsto x_4.
\]
We put $h_1'=\phi_2'(t)h_0'\phi_2'(t)^{-1}$, $h_1''=\phi_2''(t)h_0''\phi_2''(t)^{-1}$.
Then we find that
\[
\theta(h'h'_1h''h''_1\phi_2(t);\ff,\varphi)\overline{\G(h'h'_1\phi_2'(t))}
=\theta(h'h''\phi_2(t);\ff,\varphi)\overline{\G(h'\phi_2'(t))}
\]
for all $(h',h'')\in \O(V')(\A_\Q)\times \O(V'')(\A_\Q)$.
Hence we have
\[
I(\cc)=2^{-2}\int_{[\SO(V')]}\int_{[\SO(V'')]}\theta(h'h''\phi_2(t);\ff,\varphi)\overline{\G(h'\phi_2'(t))}dh''dh'.
\]
Since the isomorphism $K^1\xrightarrow{\cong}\SO(V'')\hookrightarrow\SO(V)$ is given by 
$\alpha\mapsto\phi_2(\overline{\alpha})$, 
we find that
\[
I(\cc)=2^{-1}\int_{[\SO(V')]}\int_{K^1\bs\A_K^1/\hat{\OO}_K^1\C^1}\theta(h'\phi_2(\alpha)\phi_2(t);\ff,\varphi)\overline{\G(h'\phi_2'(t))}d\alpha dh'.
\]
We see that the map $\beta\mapsto \beta\overline{\beta}^{-1}$ gives an isomorphism
\[
K^\times\bs\A_K^\times/\hat{\OO}_K^\times\C^\times\xrightarrow{\cong} K^1\bs\A_K^1/\hat{\OO}_K^1\C^1.
\]
Since there is a canonical isomorphism $K^\times\bs\A_K^\times/\hat{\OO}_K^\times\C^\times\cong {\it Cl}_K$, we have
\[
I(\cc)=\frac{1}{2h_K}\sum_{i=1}^{h_K}\int_{[\SO(V')]}
\theta(h'\phi_2(\beta_i\overline{\beta_i}^{-1}t);\ff,\varphi)\overline{\G(h'\phi_2'(t))}dh',
\]
where $\{\beta_1,\dots,\beta_{h_K}\}$ is a complete system of representatives of $K^\times\bs\A_K^\times/\hat{\OO}_K^\times\C^\times$.
We may take $\beta_i$ such that the ideal $\mathfrak{b}_i$ defined by $\beta_i$ 
is in $J_K^D$ for $i=1,\dots,h_K$.
We identify $\mathrm{G}(\SL_2\times\SL_2)$ as a subgroup of $\GU$ 
via the inclusion $\tau$ defined in Sect.~\ref{GL-GO(2,2)}.
If $h'=\rho(x_1,x_2)$ with $(x_1,x_2)\in\mathrm{G}(\SL_2\times\SL_2)$, 
by Corollary \ref{Cor}, we have
\begin{align*}
\theta(h'\phi_2(\beta\overline{\beta}^{-1}t);\ff,\varphi)
=\left(\prod_{p|D}q_p^{-1}\right)
2^{2\kappa+2}a_f(D)^{-1}\Lift(\tau(x_1,x_2)r_{\beta\overline{\beta}^{-1}t}),
\end{align*}
and
\[
\G(h'\phi_2'(t))=\g(x_1)\g(x_2\cdot d(N_{K/\Q}(t)^{-1})).
\]
Moreover we see that the automorphic form $\g(x\cdot d(N_{K/\Q}(t)^{-1}))$
is given by $g|_{2\kappa+2}d(C)=C^{-\kappa-1}g_C$.
So we have
\begin{align*}
&2^{-1}\int_{[\SO(V')]}
\Lift(\tau(x_1,x_2)r_{\beta\overline{\beta}^{-1}t})\overline{\G(h'\phi_2'(t))}dh'\\
&=2^{-1}\int_{[{\rm G}(\SL_2\times\SL_2)/\GL_1]}
C^{-\kappa-1}[F_{\mathfrak{b_i}\overline{\mathfrak{b}_i}^{-1}\cc}\parallel_{2\kappa+2}\tau(x_1,x_2)](\ii)
\\&\quad\times
\overline{\g(x_1)\g(x_2\cdot d(N_{K/\Q}(t)^{-1}))}d(x_1,x_2)\\
&=(3/\pi)^2C^{-\kappa-1}\pair{F_{\mathfrak{b_i}
\overline{\mathfrak{b}_i}^{-1}\cc}|_{\Ha\times\Ha},g\times (C^{-\kappa-1}g_C)}
\\&
=(2\xi_\Q(2))^{-2}C^{-2\kappa-2}\pair{F_{\mathfrak{b_i}\overline{\mathfrak{b}_i}^{-1}\cc}|_{\Ha\times\Ha},g\times g_C}.
\end{align*}
Now we have
\[
\pair{g_{C},g_{C}}=C^{2(\kappa+1)}\pair{g|d(C),g|d(C)}=C^{2(\kappa+1)}\pair{g,g}.
\]
Since $[\overline{\mathfrak{b}_i}]=[\mathfrak{b}_i]^{-1}$ in ${\it Cl}_K$, by Lemma \ref{indep}, we have
\begin{align*}
&I(\cc)=\left(\prod_{p|D}q_p^{-1}\right)
2^{2\kappa+2}a_f(D)^{-1}(2\xi_\Q(2))^{-2}
\frac{1}{h_K}\sum_{[\mathfrak{b}]\in {\it Cl}_K}
\frac{\pair{F_{\mathfrak{b}\overline{\mathfrak{b}}^{-1}\cc}|_{\Ha\times\Ha},g\times g_{C}}\pair{g,g}}{\pair{g_{C},g_{C}}}\\
&=2^{-1}\xi_\Q(2)^{-2}\left(\prod_{p|D}q_p^{-1}\right)
2^{2\kappa+1}a_f(D)^{-1}
\frac{1}{\# ({\it Cl}_K^2)}\sum_{[\mathfrak{a}]\in {\it Cl}_K^2}
\frac{\pair{F_{\mathfrak{a}\cc}|_{\Ha\times\Ha},g\times g_{AC}}\pair{g,g}}{\pair{g_{AC},g_{AC}}}.
\end{align*}
This completes the proof using Lemma \ref{local}.
\end{proof}

Next, we give another expression of $I(\cc)$.
To do this, we need the following lemmas.
\begin{lem}
We have
\[\varphi=\varphi'\otimes\varphi''.\]
In particular, we have
\begin{align*}
&
\theta(g\cdot d(\nu\circ\phi_2(t)),h'h''\phi_2(t);\varphi)
\\&
=\theta(g\cdot d(\nu\circ\phi_2'(t)),h'\phi_2'(t);\varphi')
\theta(g\cdot d(\nu\circ\phi_2''(t)),h''\phi_2''(t);\varphi'').
\end{align*}
\end{lem}
\begin{proof}
This is obvious.
\end{proof}

\begin{lem}\label{seesaw}
The integral
\[
\int_{[\O(V')]}\int_{[\O(V'')]}\int_{[\SL_2]}
\left|\theta(g\cdot d(\nu\circ\phi_2(t)),h'h''\phi_2(t);\varphi)\ff(g\cdot d(\nu\circ\phi_2(t)))\overline{\G(h'\phi_2'(t))}\right|dgdh''dh'
\]
is finite.
\end{lem}
\begin{proof}
By the above lemma, this integral is equal to
\begin{align*}
&\int_{[\O(V')]}\int_{[\SL_2]}
\left|\theta(g\cdot d(N_{K/\Q}(t)),h'\phi_2'(t);\varphi')\ff(g\cdot d(N_{K/\Q}(t)))\overline{\G(h'\phi_2'(t))}\right|
\\&\quad\times
\left(\int_{[\O(V'')]}
\left|\theta(g\cdot d(N_{K/\Q}(t)),h''\phi_2''(t);\varphi'')\right|dh''\right)dgdh'.
\end{align*}
Since $V''$ is anisotropic over $\Q$, the space $[\O(V'')]$ is compact.
Hence the inner integral converges and it is slowly increasing function on $[\SL_2]$.
Since $\ff$ and $\G$ are cusp forms, they are rapidly decreasing.
Moreover the function
\[
(h',g)\mapsto
\theta(g\cdot d(N_{K/\Q}(t)),h'\phi_2'(t);\varphi')
\]
is slowly increasing on ${[\O(V')]\times[\SL_2]}$.
Therefore the function
\begin{align*}
&\left|\theta(g\cdot d(N_{K/\Q}(t)),h'\phi_2'(t);\varphi')\ff(g\cdot d(N_{K/\Q}(t)))
\overline{\G(h'\phi_2'(t))}\right|
\\&\quad\times
\left(\int_{[\O(V'')]}
\left|\theta(g\cdot d(N_{K/\Q}(t)),h''\phi_2''(t);\varphi'')\right|dh''\right)
\end{align*}
is  bounded on ${[\O(V')]\times[\SL_2]}$.
This completes the proof.
\end{proof}

\begin{prop}\label{L(1/2)}
Put $q_p=(\SL_2(\Z_p):K_0(D;\Z_p))$. Then 
\begin{align*}
I(\cc)&=2^{-2}\xi_\Q(2)^{-2}\left(\prod_{p\mid D}q_p^{-1}\right)
(2\pi)^{-(2\kappa+1)}a_f(D)^{-1}(2\kappa)!\pair{g,g}
\\&\quad\times
L(1,\chi)^{-1}\sum_{Q\subset Q_D}\chi_Q(-C)a_{f_Q}(D)L(1/2,f_Q\times g).
\end{align*}
\end{prop}
\begin{proof}
By Lemma \ref{G->g}, Proposition \ref{1->E} and Lemma \ref{seesaw}, we have
\begin{align*}
I(\cc)&=
\int_{[\SL_2]}\left(2^{2\kappa+1}\xi_\Q(2)^{-2}\langle g,g\rangle \overline{\g(\alpha\cdot d(N_{K/\Q}(t)))}\right)
\left(\frac{1}{2}E(\alpha,0)\right)\ff(\alpha\cdot d(N_{K/\Q}(t)))d\alpha\\
&=2^{-1}2^{2\kappa+1}\xi_\Q(2)^{-2}\langle g,g\rangle
\int_{[\SL_2]}E(\alpha,0)\ff(\alpha\cdot d(N_{K/\Q}(t)))\overline{\g(\alpha\cdot d(N_{K/\Q}(t)))}d\alpha.
\end{align*}
Here, the Eisenstein series $E(\alpha,s)$ is defined by using the section
\[
\Phi(\alpha,s,L(\phi_2''(t)\varphi'')).
\]
We put
\[
I(\cc,s)=\int_{[\SL_2]}E(\alpha,s)\ff(\alpha\cdot d(N_{K/\Q}(t)))\overline{\g(\alpha\cdot d(N_{K/\Q}(t)))}d\alpha.
\]
Then, we have
\[
I(\cc)=2^{-1}2^{2\kappa+1}\xi_\Q(2)^{-2}\langle g,g\rangle
I(\cc,s)|_{s=0}.
\]
We formally compute $I(\cc,s)$ for $\re(s)\gg 0$.
Note that $C=N(\cc)\in\Q^\times\subset\A_\Q^\times$, 
$\g$ has the trivial central character and 
$\ff,\g$ are left $\GL_2(\Q)$-invariant.
So we have
\begin{align*}
I(\cc,s)&=
\int_{B(\Q)\bs\SL_2(\A_\Q)}\Phi(\alpha,s)\ff(\alpha\cdot d(N_{K/\Q}(t)))\overline{\g(\alpha\cdot d(N_{K/\Q}(t)))}d\alpha\\
&=\int_{B(\Q)\bs\SL_2(\A_\Q)}\Phi(\alpha,s)
\ff(a(C)\alpha d(N_{K/\Q}(t)))\overline{\g(a(C)\alpha d(N_{K/\Q}(t)))}d\alpha\\
&=\int_{B(\Q)\bs\SL_2(\A_\Q)}\Phi(\alpha,s)\ff(a(C)\alpha d(N_{K/\Q}(t)))
\sum_{\xi\in\Q^\times}\overline{W_{\g}(a(\xi C)\alpha d(N_{K/\Q}(t)))}d\alpha.
\end{align*}
Note that
\begin{align*}
\sum_{\xi\in\Q^\times}\overline{W_{\g}(a(\xi C)\alpha d(N_{K/\Q}(t)))}d\alpha
&=\sum_{\xi\in\Q^\times/\Q^{\times 2}}\sum_{\gamma\in\Q_{>0}}
\overline{W_{\g}(a(\xi\gamma^2C)\alpha d(N_{K/\Q}(t)))}d\alpha\\
&=\frac{1}{2}\sum_{\xi\in\Q^\times/\Q^{\times 2}}\sum_{\gamma\in\Q^\times}
\overline{W_{\g}(a(\xi C)t(\gamma)\alpha d(N_{K/\Q}(t)))}d\alpha.
\end{align*}
Hence $I(\cc)$ is equal to
\begin{align*}
&\frac{1}{2}\sum_{\xi\in\Q^\times/\Q^{\times 2}}
\int_{N(\Q)\bs\SL_2(\A_\Q)}\Phi(\alpha,s)\ff(a(C)\alpha d(N_{K/\Q}(t)))
\overline{W_{\g}\left(a(\xi C)\alpha d(N_{K/\Q}(t))\right)}d\alpha\\
&=\frac{1}{2}\sum_{\xi\in\Q^\times/\Q^{\times 2}}
\int_{N(\Q)\bs\SL_2(\A_\Q)}\Phi(\alpha,s)
\\&\quad\times
\sum_{\xi'\in\Q^\times}W_{\ff}(a(\xi' C)\alpha d(N_{K/\Q}(t)))
\overline{W_{\g}\left(a(\xi C)
\alpha d(N_{K/\Q}(t))\right)}d\alpha\\
&=\frac{1}{2}\sum_{\xi\in\Q^\times/\Q^{\times 2}}\sum_{\xi'\in\Q^\times}
\int_{N(\A_\Q)\bs\SL_2(\A_\Q)}
\left(\int_{\Q\bs\A_\Q}\psi(\xi'Cx)\overline{\psi(\xi Cx)}dx\right)
\\&\quad\times
\Phi(\alpha,s)W_{\ff}(a(\xi'C)\alpha d(N_{K/\Q}(t)))
\overline{W_{\g}\left(a(\xi C)\alpha d(N_{K/\Q}(t))\right)}d\alpha\\
&=\frac{1}{2}\sum_{\xi\in\Q^\times/\Q^{\times 2}}
\int_{N(\A_\Q)\bs\SL_2(\A_\Q)}
\Phi(\alpha,s)
\\&\quad\times
W_{\ff}(a(\xi C)\alpha d(N_{K/\Q}(t)))
\overline{W_{\g}\left(a(\xi C)\alpha d(N_{K/\Q}(t))\right)}d\alpha\\
&=\frac{1}{2}\sum_{\xi\in\Q^\times/\Q^{\times 2}}\prod_v I_v(\cc,s,\xi)
=\frac{1}{2}\sum_{\substack{\xi\in\Z\\\text{square free}}}\prod_v I_v(\cc,s,\xi).
\end{align*}
Here, $\displaystyle\sum_{\substack{\xi\in\Z\\\text{square free}}}$ means the sum over all square free integers, 
and we put
\begin{align*}
&I_v(\cc,s,\xi)
\\&=
\int_{N(\Q_v)\bs\SL_2(\Q_v)}
\Phi_v(\alpha,s)W_{\ff,v}(a(\xi C)\alpha d(N_{K/\Q}(t)_v))
\overline{W_{\g,v}\left(a(\xi C)\alpha d(N_{K/\Q}(t)_v)\right)}d\alpha.
\end{align*}
This calculation is justified by the following lemma.
\begin{lem}
We define $J_v(\cc,s,\xi,\xi')$ by
\begin{align*}
\int_{N(\Q_v)\bs\SL_2(\Q_v)}
\left|\Phi_v(\alpha,s)
W_{\ff,v}(a(\xi' C)\alpha d(N_{K/\Q}(t)_v))
\overline{W_{\g,v}\left(a(\xi C)
\alpha d(N_{K/\Q}(t)_v)\right)}\right|d\alpha.
\end{align*}
If $\re(s)>6$, then
\[
\sum_{\substack{\xi\in\Z\\\text{square free}}}\sum_{\xi'\in\Q^\times}
\prod_vJ_v(\cc,s,\xi,\xi')
<\infty.
\]
\end{lem}
\begin{proof}
We need the following lemma which will be proved in the last section.
\begin{lem}\label{I and J}
We put $\sigma=\re(s)>6$.
\begin{enumerate}
\item\label{arch2}
For $v=\infty$, we have
\[
I_{\infty}(\cc,s,\xi)=\left\{
\begin{aligned}
&(4\pi)^{-(s/2+2\kappa+1)}|C\xi|_\infty^{-s/2+1/2}\Gam\left({s/ 2}+2\kappa+1\right)
&\quad&{\rm if\ }\xi>0,\\
&0&\quad&{\rm if\ }\xi<0.
\end{aligned}
\right.
\]
If $\xi,\xi'>0$, then
\[
J_\infty(\cc,s,\xi,\xi')
\leq(4\pi)^{-(\sigma/2+2\kappa+1)}(C\xi)^{-\sigma/4+1/2}
(C\xi')^{-\sigma/4}\Gamma(\sigma/2+2\kappa+1)
\]
and otherwise, $J_\infty(\cc,s,\xi,\xi')=0$.
\item\label{unram2}
For $v=p\nmid CD$, we have
\[
I_p(\cc,s,\xi,\gamma)=|\xi|_p^{-s/2+1/2}\sum_{n=0}^{\infty}|\xi p^{2n}|_p^{s/2+2\kappa+1}
a_f((\xi p^{2n})_p)\overline{a_g((\xi p^{2n})_p)}
\]
and
\[
J_p(\cc,s,\xi,\xi')=|\xi|_p^{\kappa+1}|\xi'|_p^{\kappa+1/2}\sum_{n=0}^\infty
(p^{-n})^{\sigma+4\kappa+2}\left|a_f((\xi'p^{2n})_p)\overline{a_g((\xi p^{2n})_p)}\right|.
\]
\item\label{ram2}
For $v=p\mid D$, we have
\begin{align*}
I_p(\cc,s,\xi)&=\frac{1}{(\SL_2(\Z_p):K_0(D;\Z_p))}
a_f(D_p)^{-1}|\xi|_p^{-s/2+1/2}\sum_{n=0}^{\infty}|\xi p^{2n}|_p^{s/2+2\kappa+1}
\\&\quad \times
[a_f((D\xi p^{2n})_p)+\underline{\chi}_p(-C\xi)\overline{a_f((D\xi p^{2n})_p)}]
\overline{a_g((\xi p^{2n})_p)}
\end{align*}
and
\begin{align*}
J_p(\cc,s,\xi,\xi')\leq&\frac{2|a_F(D_p)|^{-1}}{(\SL_2(\Z_p):K_0(D))}
|\xi|_p^{\kappa+1}|\xi'|_p^{\kappa+1/2}
\\&\quad\times
\sum_{n=0}^\infty
(p^{-n})^{\sigma+4\kappa+2}\left|a_f((D\xi'p^{2n})_p)\overline{a_g((\xi p^{2n})_p)}\right|.
\end{align*}
\item\label{divC2}
For $v=p\mid C$, we have
\begin{align*}
I_p(\cc,s,\xi)
=p^{-1/2}\frac{p^s+p}{1+p}
|\xi|_p^{-s/2+1/2}\sum_{n=0}^{\infty}|\xi p^{2n}|_p^{s/2+2\kappa+1}
a_f((\xi p^{2n})_p)\overline{a_g((\xi p^{2n})_p)}
\end{align*}
and
\begin{align*}
J_p(\cc,s,\xi,\xi')&=p^{-1/2}\frac{p^\sigma+p}{1+p}|\xi|_p^{\kappa+1}|\xi'|_p^{\kappa+1/2}
\\&\quad\times
\sum_{n=0}^\infty
(p^{-n})^{\sigma+4\kappa+2}\left|a_f((\xi'p^{2n})_p)\overline{a_g((\xi p^{2n})_p)}\right|.
\end{align*}
\end{enumerate}
\end{lem}
By this lemma, there is a constant $M_1=M_1(\sigma)>0$ which does not depend on $\xi,\xi'$ and satisfies
\begin{align*}
&\prod_v J_v(\cc,s,\xi,\xi')
\\&\leq
M_1 \xi^{-\sigma/4-\kappa-1/2}(\xi')^{-\sigma/4-\kappa-1/2}
\prod_{p<\infty}\sum_{n=0}^{\infty}(p^{-n})^{\sigma+4\kappa+2}|a_f((D\xi'p^{2n})_p)\overline{a_g((\xi p^{2n})_p)}|\\
&=M_1 \xi^{-\sigma/4-\kappa-1/2}(\xi')^{-\sigma/4-\kappa-1/2}
\sum_{m=1}^{\infty}m^{-(\sigma+4\kappa+2)}|a_f(D\xi'm^2)\overline{a_g(\xi m^2)}|
\end{align*}
if $\xi,\xi'>0$. 
Otherwise, $\prod_v J_v(\cc,s,\xi,\xi')=0$.
We put $t=\sigma/4+\kappa+1/2$.
Then we have
\begin{align*}
&\sum_{\substack{\xi\in\Z\\\text{square free}}}\sum_{\xi'\in\Q^\times}
\prod_vJ_v(\cc,s,\xi,\xi')
\\&\leq M_1  \sum_{m=1}^\infty\sum_{\substack{\xi\in\Z_{>0}\\\text{square free}}}\sum_{\xi'\in\Q_{>0}^\times}
\xi^{-t} (\xi')^{-t}m^{-4t}|a_f(D\xi'm^2)\overline{a_g(\xi m^2)}|.
\end{align*}
Unless $\xi'\in(Dm^2)^{-1}\Z_{>0}$, we have $|a_f(D\xi'm^2)\overline{a_g(\xi m^2)}|=0$.
So we have
\begin{align*}
&\sum_{\substack{\xi\in\Z\\\text{square free}}}\sum_{\xi'\in\Q^\times}
\prod_vJ_v(\cc,s,\xi,\xi')
\\&\leq M_1  \sum_{m=1}^\infty\sum_{\substack{\xi\in\Z_{>0}\\\text{square free}}}\sum_{n\in\Z_{>0}}
\xi^{-t} \left(\frac{n}{Dm^2}\right)^{-t}m^{-4t}|a_f(n)\overline{a_g(\xi m^2)}|.
\end{align*}
It is well-known that
\[
|a_f(n)|=O(n^{\kappa+1/2}),\quad|a_g(n)|=O(n^{\kappa+1}).
\]
Hence, there exists a constant $M_2>0$ such that
\begin{align*}
&\sum_{\substack{\xi\in\Z\\\text{square free}}}\sum_{\xi'\in\Q^\times}
\prod_vJ_v(\cc,s,\xi,\xi')\\
&\leq M_2  \sum_{m=1}^\infty\sum_{\substack{\xi\in\Z_{>0}\\\text{square free}}}\sum_{n\in\Z_{>0}}
\xi^{-t} \left(\frac{n}{Dm^2}\right)^{-t}m^{-4t}n^{\kappa+1/2}(\xi m^2)^{\kappa+1}\\
&\leq M_2 D^{t} \sum_{m=1}^\infty\sum_{\xi=1}^{\infty}\sum_{n=1}^\infty
\xi^{-\sigma/4+1/2}m^{-\sigma/2+1}n^{-\sigma/4}\\
&=M_2D^t\zeta(\sigma/4-1/2)\zeta(\sigma/2-1)\zeta(\sigma/2).
\end{align*}
This is finite if $\sigma>6$.
\end{proof}
We return to the proof of Proposition \ref{L(1/2)}.
By Lemma \ref{I and J}, for $\re(s)\gg0$, we have
\begin{align*}
I(\cc,s)&=\frac{1}{2}\sum_{\substack{\xi\in\Z\\\text{square free}}}\prod_v I_v(\cc,s,\xi)\\
&=\frac{1}{2} (4\pi)^{-(s/2+2\kappa+1)}C^{-s/2}\Gamma(s/2+2\kappa+1)
\left(\prod_{q\mid C}\frac{q^s+q}{1+q}\right)\left(\prod_{p\mid D}q_p^{-1}\right)
\\&\quad \times
a_f(D)^{-1}\sum_{\substack{\xi\in\Z_{>0}\\\text{square free}}}\sum_{n=1}^\infty
(\xi n^2)^{-(s/2+2\kappa+1)}\alpha_{F_\cc}(D\xi n^2)\overline{a_g(\xi n^2)}\\
&=\frac{1}{2} (4\pi)^{-(s/2+2\kappa+1)}C^{-s/2}\Gamma(s/2+2\kappa+1)
\left(\prod_{q\mid C}\frac{q^s+q}{1+q}\right)\left(\prod_{p\mid D}q_p^{-1}\right)
\\&\quad \times
a_f(D)^{-1}\sum_{m=1}^\infty
m^{-(s/2+2\kappa+1)}\alpha_{F_\cc}(Dm)\overline{a_g(m)}.
\end{align*}
Recall that $a_{f^{\cc*}}(Dm)={\bf a}_D^\cc(Dm)\alpha_{F_\cc}(Dm)$.
By definition of ${\bf a}_D^\cc$, we have ${\bf a}_D^\cc(Dm)=1$.
So we have
\[
\alpha_{F_\cc}(Dm)=a_{f^{\cc*}}(Dm)
=\sum_{Q\subset Q_D}\chi_Q(-C)a_{f_Q}(Dm)
=\sum_{Q\subset Q_D}\chi_Q(-C)a_{f_Q}(D)a_{f_Q}(m).
\]
Hence we have
\begin{align*}
&\sum_{m=1}^\infty m^{-(s/2+2\kappa+1)}\alpha_{F_\cc}(Dm)\overline{a_g(m)}\\
&=\sum_{Q\subset Q_D}\chi_Q(-C)a_{f_Q}(D)
\sum_{m=1}^\infty m^{-(s/2+2\kappa+1)}a_{f^{\cc*}}(m)\overline{a_g(m)}\\
&=\sum_{Q\subset Q_D}\chi_Q(-C)a_{f_Q}(D)\D(s/2+2\kappa+1,f_Q,g)\\
&=\sum_{Q\subset Q_D}\chi_Q(-C)a_{f_Q}(D)L(s+1,\chi)^{-1}L(s/2+1/2,f_Q\times g).
\end{align*}
This completes the proof of Proposition \ref{L(1/2)} using Lemma \ref{I and J}.
\end{proof}
By Proposition \ref{period} and \ref{L(1/2)}, we get the following corollary 
using Lemma \ref{local} and Lemma \ref{I and J}.
\begin{cor}\label{L(1/2)=<F,g*g>}
For $\cc\in J_\OO^D$ such that $C=N(\cc)$ is a square free integer, we have
\begin{align*}
&\sum_{Q\subset Q_D}\chi_Q(-C)a_{f_Q}(D)L(1/2,f_Q\times g)
\\&=\frac{2L(1,\chi)(4\pi)^{2\kappa+1}}{(2\kappa)!}
\frac{1}{\# ({\it Cl}_K^2)}\sum_{[\mathfrak{a}]\in {\it Cl}_K^2}
\frac{\pair{F_{\mathfrak{a}\cc}|_{\Ha\times\Ha},g\times g_{AC}}}{\pair{g_{AC},g_{AC}}}.
\end{align*}
\end{cor}

\subsection{Genus theory}
We recall the genus theory (see e.g., Hecke \cite{Hecke} \S48).
Let $\mathfrak{G}={\it Cl}_K/{\it Cl}_K^2$.
Note that the canonical homomorphism
$J_\OO^D\rightarrow \mathfrak{G}$ is surjective by the Chebotarev density theorem.
Let $N\colon J_\OO^D\rightarrow \Z_{>0}$ be the ideal norm map.
For $Q\subset Q_D$, the map $\chi_Q\circ N\colon J_\OO^D\rightarrow \{\pm1\}$
gives a character of $\mathfrak{G}$.
Moreover, for $Q,Q'\subset Q_D$, we see that
$\chi_Q\circ N=\chi_{Q'}\circ N$ if and only if $Q'=Q$ or $Q'=Q_D\setminus Q$.
In particular, $\chi_Q\circ N$ is the trivial character of $\mathfrak{G}$
if and only if $Q=\emptyset$ or $Q=Q_D$.
So if we choose $\cc_1,\dots,\cc_l\in J_\OO^D$ such that 
they give a complete system of $\mathfrak{G}$, then we have
\[
\frac{1}{({\it Cl}_K:{\it Cl}_K^2)}
\sum_{j=1}^l \chi_{Q}(-N(\cc_j))=\left\{
\begin{aligned}
&1\iif Q=\emptyset,\\
&-1\iif Q=Q_D,\\
&0&\quad&\text{otherwise}
\end{aligned}
\right.
\]
since $\chi(-1)=-1$.
We may assume that $C_j=N(\cc_j)$ is a square free integer for each $j=1,\dots,l$.
Considering $({\it Cl}_K:{\it Cl}_K^2)^{-1}\sum_{j=1}^l I(\cc_j)$, 
by Corollary \ref{L(1/2)=<F,g*g>} and Lemma \ref{indep}, we have
\begin{align*}
&
a_{f_\emptyset}(D)L(1/2,f_\emptyset\times g)
-a_{f_{Q_D}}(D)L(1/2,f_{Q_D}\times g)
\\
&=\frac{2L(1,\chi)(4\pi)^{2\kappa+1}}{(2\kappa)!}
\frac{1}{\#({\it Cl}_K^2)}\frac{1}{({\it Cl}_K:{\it Cl}_K^2)}
\sum_{j=1}^l\sum_{[\mathfrak{a}]\in {\it Cl}_K^2}
\frac{\pair{F_{\mathfrak{a}\cc_j}|_{\Ha\times\Ha},g\times g_{AC_j}}}{\pair{g_{AC_j},g_{AC_j}}}\\
&=\frac{2L(1,\chi)(4\pi)^{2\kappa+1}}{(2\kappa)!}
\frac{1}{h_K}\sum_{[\cc]\in {\it Cl}_K}
\frac{\pair{F_{\cc}|_{\Ha\times\Ha},g\times g_{C}}}{\pair{g_{C},g_{C}}}.
\end{align*}
Now we find that
\begin{align*}
&a_{f_\emptyset}(D)L(1/2,f_\emptyset\times g)=a_{f}(D)L(1/2,f\times g)
,\\
&a_{f_{Q_D}}(D)L(1/2,f_{Q_D}\times g)=\overline{a_{f}(D)L(1/2,f\times g)}.
\end{align*}
Since $a_{f}(D)L(1/2,f\times g)\in\I\R$, we have
\[
a_{f_\emptyset}(D)L(1/2,f_\emptyset\times g)
-a_{f_{Q_D}}(D)L(1/2,f_{Q_D}\times g)
=2a_f(D)L(1/2,f\times g).
\]
This completes the proof of Theorem \ref{main} using
Lemma \ref{local} and Lemma \ref{I and J}.

\section{The local integrals}\label{local integrals}
In this section, we prove Lemma \ref{local} and Lemma \ref{I and J}.
We show these lemmas only when $v=\infty$, $v=p=2$ or $v=p\mid C_i$ (or $v=p\mid C$).
The other cases are shown similarly.

\subsection{Preliminaries}
Through this section, we use the notation
\[
B=\begin{pmatrix}
b_1&b_2+\sD b_3\\b_2-\sD b_3&b_4
\end{pmatrix}
\]
and we put $\xi=\det(B)$, $\beta=\beta(B)$, $m_i=\ord_p(b_i)$ and $l=\ord_p(\xi)$
when we prove Lemma \ref{local} as in Sect.~\ref{coeff of theta}.
We fix a place $v$ of $\Q$.
Let $\psi=\psi_v$ be the standard character of $\Q_v$.
\par

Since the anisotropic kernel of $V$ is isomorphic to the one of $V''$ over $\Q$, 
we find that $\gamma_{V(\Q_v)}=\gamma_{V''(\Q_v)}$  and 
$\chi_{V(\Q_v)}=\chi_{V''(\Q_v)}=\bchi_v$.
Moreover, we have $\gamma_{V'(\Q_v)}=1$ and $\chi_{V'(\Q_v)}=1$.
We simply denote $\gamma_{V(\Q_v)}$ by $\gamma_V$ and $\chi_{V(\Q_v)}$ by $\chi_V$. 
It is easily seen that
$\vol(V(\Z_p))=\vol(V''(\Z_p))=|4D|_p^{1/2}$ and $\vol(V'(\Z_p))=1$ if $v=p<\infty$,
and
\[
\gamma_V=\gamma_{V''}=(2,-1)_{\Q_v}\gamma_{\Q_v}(D,\psi)\gamma_{\Q_v}(\psi)^2.
\]
\par

Let $\varphi'\in\mathcal{S}(V'(\Q_v))$ and $\varphi''\in\mathcal{S}(V''(\Q_v))$
and we put $\varphi=\varphi'\otimes\varphi''\in\mathcal{S}(V(\Q_v))$.
Then we have
\[
\omega_V(g,1)\varphi=(\omega_{V'}(g,1)\varphi')\otimes(\omega_{V''}(g,1)\varphi'').
\]
for $g\in\SL_2(\Q_v)$.
If $v=p<\infty$, for open compact subsets $A,B\subset \Q_p$, we denote 
the characteristic function of of 
$Ae_3+Be_4\subset V''(\Q_p)$ by $\varphi''(A,B)\in\mathcal{S}(V''(\Q_p))$.
Let $\varphi'\in \mathcal{S}(V'(\Q_p))$ be the characteristic function of $V'(\Z_p)$. 
Then we find that 
$\varphi'$ is $\SL_2(\Z_p)$-invariant.
We put $\varphi(A,B)=\varphi'\otimes\varphi''(A,B)\in\mathcal{S}(V(\Q_p))$.
\par

Let $v=p<\infty$.
We put $q_p=(\SL_2(\Z_p):K_0(D;\Z_p))$.
It is well-known that for $N\in p\Z_p$,
a complete system of representatives of $\SL_2(\Z_p)/K_0(N;\Z_p)$
is given by
\[
\mathcal{A}(N;\Z_p)=
\left\{\left.
\begin{pmatrix}
1&0\\c&1
\end{pmatrix}
\right|c\in\Z_p/N\Z_p
\right\}\cup
\left\{\left.
\begin{pmatrix}
a&-1\\1&0
\end{pmatrix}
\right|a\in p\Z_p/N\Z_p
\right\}.
\]
Note that
\[
\begin{pmatrix}
1&0\\c&1
\end{pmatrix}
=-w\cdot n(-c)\cdot w
,\quad
\begin{pmatrix}
a&-1\\1&0
\end{pmatrix}
=n(a)\cdot w.
\]
If $c\not=0$, we have
$
-w\cdot n(-c)\cdot w
=t(c^{-1})\cdot n(c)\cdot w\cdot n(c^{-1})$.

\subsection{The archimedean case}
We will prove Lemma \ref{local} (\ref{arch}) and Lemma \ref{I and J} (1).
Let $v=\infty$. 
We define the functions $f$ and $q'$ on $V'(\R)$ and $q''$ on $V''(\R)$ by
\begin{align*}
f(u')&=-\I u_1+u_2-u_5+\I u_6\quad{\rm and }\\
q'(u')&=u_1^2+u_2^2+u_5^2+u_6^2,\quad
q''(u'')=2u_3^2+2Du_4^2=2Q''[u'']
\end{align*}
for $u'=u_1e_1+u_2e_2+u_5e_5+u_6e_6\in V'(\R)$ and $u''=u_3e_3+u_4e_4\in V''(\R)$.
We regard $f$ as a function on $V(\R)$ by
$f(u'+u'')=f(u')$
for $u'\in V'(\R)$ and $u''\in V''(\R)$.
We put $q(u)=q'(u')+q''(u'')$ for $u=u'+u''\in V(\R)$ with $u'\in V'(\R)$ and $u''\in V''(\R)$.
We consider 
\[
\varphi'(u')=f(u')^{2\kappa+2}e^{-\pi q'(u')}\in\mathcal{S}(V'(\R))
,\quad
\varphi''(u'')=e^{-\pi q''(u'')}\in\mathcal{S}(V''(\R)),
\]
and we put $\varphi=\varphi'\otimes \varphi''\in\mathcal{S}(V(\R))$.
\begin{prop}\label{SO(2)-inv}
For $k_\theta\in\SO(2)$, we have
\begin{align*}
\omega_{V'(\R)}(k_\theta,1)\varphi'(u')
=e^{-\I(2\kappa+2)\theta}\varphi'(u'),\quad
\omega_{V''(\R)}(k_\theta,1)\varphi''(u'')
=e^{\I\theta}\varphi''(u'').
\end{align*}
So we have
\[
\omega_{V(\R)}(k_\theta,\phi_1(k')\phi_2(t))\varphi(u)
=e^{-\I(2\kappa+1)\theta}\det(\alpha+\I\beta)^{2\kappa+2}\varphi(u)
\]
for
\[
k'=\begin{pmatrix}
\alpha&\beta\\-\beta&\alpha
\end{pmatrix}
\in \K_0
\]
and $t\in\C^1$.
In particular, the functions
\begin{align*}
&\SL_2(\R)\ni g\longmapsto \hat{\omega}_{V(\R)}(g,\phi_1((z{\bf 1}_4)n(X)m(A)k')\phi_2(t))
\hat{\varphi}(-\beta;0,1)W_{\ff,\infty}(a(\xi)g),\\
&\SL_2(\R)\ni g\longmapsto \Phi(g,s,\varphi'')W_{\ff,\infty}(a(\xi')g)\overline{W_{\g,\infty}(a(\xi)g)}
\end{align*}
are $\SO(2)$-invariant for all $z\in\C^\times$, $X\in\Her(\R)$, $A\in\GL_2(\C)$ with 
$\det(A)\in\R^\times$,
$k'\in\K_0$, $t\in\C^\times$ with $|t|=1$
, $\xi,\xi'\in\Q^\times$ and $s\in\C$.
\end{prop}
\begin{proof}
Note that $k'^{-1}={}^t\overline{k'}$ and $\det(k')=\det(\alpha+\I\beta)\det(\alpha-\I\beta)=1$.
By a simple calculation, we have
\[
f(\phi_2(t)^{-1}\phi_1(k)^{-1}u)=\det(\alpha+\I\beta)f(u),
\quad
q(\phi_2(t)^{-1}\phi_1(k)^{-1}u)=q(u)
\]
for $u\in V(\R)$.
These equations give the actions of $k'$ and $t$.
\par

To study the action of $\SO(2)$ on $\mathcal{S}(V'(\R))$ and on $\mathcal{S}(V''(\R))$, 
consider the action of the Lie algebra ${\rm Lie}(\SL_2(\R))$.
Note that $k_\theta=\exp(\theta(X_1-X_2))$ with
\[
X_1=\begin{pmatrix}
0&1\\0&0
\end{pmatrix},\ 
X_2=\begin{pmatrix}
0&0\\1&0
\end{pmatrix}
\in{\rm Lie}(\SL_2(\R)).
\]
By  the classical Fourier analysis,
we can calculate the actions of $X_1$ and $X_2$.
\end{proof}

Now, we start to prove Lemma \ref{local} (\ref{arch}).
The proof of Lemma \ref{local} (\ref{arch}) is similar to that of Lemma 7.6 of \cite{Ichino}.
By Proposition \ref{SO(2)-inv}, we may assume that $x_\infty=n(X)m(A)$.
By \cite{Ichino} Lemma 7.4, we find that $\hat{\varphi}(x;y_1,y_2)$ is equal to
\begin{align*}
(2\sqrt{\pi})^{-(2\kappa+2)}
&H_{2\kappa+2}(\sqrt{\pi}(x_1-x_4+\I y_1+y_4))
\\&\times
\exp(-\pi((x_1-x_4+y_2)^2-2y_2(x_1-x_4)+2Q_1[x]+y_1^2))
\end{align*}
for $x={}^t(x_1,x_2,x_3,x_4)\in V_1(\R)$ and $y_1,y_2\in\R$.
Here, 
\[
H_n(x)=(-1)^n e^{x^2}\frac{d^n}{dx^n}\left(e^{-x^2}\right)
\]
is the Hermite polynomial.
For $A\in\GL_2(\C)$ with $\det(A)\in\R^\times$, we have
\[
\phi_1(m(A))=
\begin{pmatrix}
\det(A)&0&0\\
0&h_1&0\\
0&0&\det(A)^{-1}
\end{pmatrix}
\]
with some $h_1\in\GSO(3,1)(\R)$.
By a simple calculation, we have
\[
-(h_1^{-1}\beta)_1+(h_1^{-1}\beta)_4=\det(A)^{-1}\Tr(BA{}^t\overline{A})
=\det(A)^{-1}\Tr(BY).
\]
Let $v=v(X)$ as in Sect.~\ref{coeff of theta}.
Then we have $\phi_1(n(X))=\ell(v)$, $(v,\beta)_{V_1}=\Tr(BX)$
and $Q_1[h_1^{-1}\beta]=Q_1[\beta]=-\det(B)=-\xi$.
So we have
\begin{align*}
&\hat{\omega}\left(
t(x),\phi
\left(
n(X)m(A)
\right)
\right)
\hat{\varphi}(-\beta;0,1)\\
&=\chi_V(x)|x|_\infty^{3}
\int_{\R}\varphi
\begin{pmatrix}
x\det(A)^{-1}z-x\det(A)^{-1}(v,\beta)\\
-h_1^{-1}\beta x\\0
\end{pmatrix}
e^{2\pi\I z}dz\\
&=\chi_V(x)|x|_\infty^{2}|\det(A)|_\infty e^{2\pi\I \Tr(BX)}
\hat{\varphi}(-h_1^{-1}\beta x;0,\det(A)x^{-1})\\
&=\chi_V(x)|x|_\infty^{2}|\det(A)|_\infty e^{2\pi\I \Tr(BX)}
H_{2\kappa+2}(\sqrt{\pi}[\det(A)^{-1}x\Tr(BY)+\det(A)x^{-1}])\\
&\quad\times
(4\pi)^{-\kappa-1}
\exp(-\pi[\det(A)^{-1}x\Tr(BY)+\det(A)x^{-1}]^2)
\exp(2\pi\Tr(BY)+2\pi x^2\xi).
\end{align*}
By the formula of $W_{\ff,\infty}$ in Sect.~\ref{Whittaker},
we have $\W_{B,\infty}(n(X)m(A))=0$ if $\xi<0$.
If $\xi>0$, then $\W_{B,\infty}(n(X)m(A))$ is equal to
\begin{align*}
&\int_{\R^\times}\hat{\omega}(t(x),\phi_1(n(X)m(A)))\hat{\varphi}(-\beta;0,1)W_{\ff,\infty}(a(\xi)t(x))|x|_\infty^{-2}d^\times x\\
&=2\int_0^\infty \hat{\omega}(t(x),\phi_1(n(X)m(A)))\hat{\varphi}(-\beta;0,1)W_{\ff,\infty}(a(\xi)t(x))x^{-3}dx\\
&=2|\det(A)|^{2\kappa+2} e^{2\pi\I\Tr(B(X-Y\I))}(4\pi)^{-\kappa-1}\xi^{\kappa+(1/2)}
\\
&\quad\times 
\int_{0}^{\infty}x^{2\kappa}
e^{-\pi[x\Tr(BY)+x^{-1}]^2}
H_{2\kappa+2}(\sqrt{\pi}[x\Tr(BY)+x^{-1}])dx.
\end{align*}
By \cite{Ichino} Lemma 7.5, the last integral is equal to
\[\left\{
\begin{split}
&2^{2(2\kappa+2)-1}\pi^{\kappa+1}e^{4\pi\I\Tr(BY\I)}&\quad&{\rm if\ }B>0,\\
&0&\quad&{\rm if\ }B<0.
\end{split}
\right.\]
Since $Y=A{}^t\overline{A}$, we have $\det(Y)=|\det(A)|^2$.
This completes the proof of Lemma \ref{local} (\ref{arch}).

\par

Lemma \ref{I and J} (\ref{arch2}) follows from the equation
$\Phi(t(x),s,\varphi'')=\chi_{V''(\R)}(x)|x|_\infty^{s+1}$, 
the formulas of $W_{\ff,\infty}$ and $W_{\g,\infty}$ in Sect.~\ref{Whittaker},
and the well-known formula
\[
\int_{\R^\times}|x|^se^{-ax^2}d^\times x
=a^{-s/2}\Gamma(\frac{s}{2})
\]
for $a>0$.

\subsection{The unramified $2$-adic case}
We will prove Lemma \ref{local} (\ref{unram}) and Lemma \ref{I and J} (\ref{unram2}) for $p=2$.
Let $v=2\nmid D$. 
So we find that $D\equiv 3\bmod 4$.

In this case, we have
$(2,-1)_{\Q_2}=1$, $\gamma_{\Q_2}(D,\psi)=\I$ and $
\gamma_{\Q_2}(\psi)=\zeta_8^{-1}$.
Hence we have
$\gamma_{V}=\gamma_{V''}=1$.
\par

Let $\varphi''=\varphi''(\Z_2,\Z_2)+\varphi''(2^{-1}\Z_2^\times,2^{-1}\Z_2^\times)$
and $\varphi''_0=\varphi''(\Z_2,\Z_2)\in\mathcal{S}(V''(\Q_2))$. 
We put $\varphi=\varphi'\otimes\varphi''$ and $\varphi_0=\varphi'\otimes\varphi_0''$.
\begin{lem}\label{unram inv}
We have
\[
\varphi''(u)=2\int_{\SL_2(\Z_2)}\omega(k,1)\varphi''_0(u)dk
\quad {\rm and\ so}\quad
\varphi(u)=2\int_{\SL_2(\Z_2)}\omega(k,1)\varphi_0(u)dk.
\]
In particular, $\varphi$ and $\varphi''$ is $\SL_2(\Z_2)$-invariant.
Moreover $\varphi$ is also $\phi_1(\GSU(\Z_2))\phi_2((\OO\otimes_\Z\Z_2)^\times)$-invariant.
\end{lem}

\begin{proof}
It is easy that $\varphi''_0$ is $K_0(4;\Z_2)$-invariant.
For $a,c\in\Z_2$, we have
\begin{align*}
\omega(-w\cdot n(-c)\cdot w,1)\varphi''_0(u)&=
\omega(-w,1)[(\omega(w,1)\varphi''_0](u))\psi(-cQ''[u]),\\ 
\omega(n(a)w,1)\varphi''_0&=[\omega(w,1)\varphi''_0](u)\psi(aQ''[u]).
\end{align*}
Now we find that
\begin{align*}
\sum_{c\in\Z_2/4\Z_2}\psi(-cQ''[u])
&=\left\{
\begin{aligned}
&4\iif Q''[u]\in\Z_2,\\
&0&\quad&\text{otherwise},
\end{aligned}
\right.
\end{align*}
Hence, we have
\[
\sum_{c\in\Z_2/4\Z_2}(\omega(w,1)\varphi''_0(u))\psi(-cQ''[u])
=4\vol(V''(\Z_2))\varphi''(u).
\]
Similarly, we have
\[
\sum_{a\in2\Z_2/4\Z_2}(\omega(w,1)\varphi''_0(u))\psi(aQ''[u])
=2\vol(V''(\Z_2))\varphi''(u).
\]
On the other hand, $\omega(w,1)\varphi''$ is equal to
\begin{align*}
&\omega(w,1)[\varphi''(2^{-1}\Z_2,2^{-1}\Z_2)-\varphi''(2^{-1}\Z_2,\Z_2)
-\varphi''(\Z_2,2^{-1}\Z_2)+2\varphi''(\Z_2,\Z_2)]\\
&=\vol(V''(\Z_2))[4\varphi''(\Z_2,\Z_2)-2\varphi''(\Z_2,2^{-1}\Z_2)
-2\varphi''(2^{-1}\Z_2,\Z_2)+2\varphi''(2^{-1}\Z_2,2^{-1}\Z_2)].
\end{align*}
Since $\vol(V''(\Z_2))=|\det(Q'')|^{1/2}=2^{-1}$, we have $\omega(w,1)\varphi''=\varphi''$.
Therefore we have
\begin{align*}
&\int_{\SL_2(\Z_2)}\omega(k,1)\varphi''_0(u)dk
\\&=(\SL_2(\Z_2):K_0(4;\Z_2))^{-1}\left[4\vol(V''(\Z_2))\varphi''(u)+2\vol(V''(\Z_2))\varphi''(u)\right]
=2^{-1}\varphi''(u).
\end{align*}
\par

The equation for $\varphi$ implies that 
$\varphi$ is $\SL_2(\Z_2)\times\phi_1(\GSU(\Z_2))$-invariant.
By Lemma \ref{invariant}, we find that $\varphi$ is $\phi_2((\OO\otimes_\Z\Z_p)^\times)$-invariant.
\end{proof}
\par

Now, we start to prove Lemma \ref{local} (\ref{unram}) for $p=2$. 
By Lemma \ref{GU(Zp)} and Lemma \ref{unram inv}, we find that $\W_{B,2}$ is right $\GU(\Z_2)$-invariant.
Moreover, we have
\begin{align*}
\W_{B,2}(1)&=\int_{\Q_2^\times}\hat{\omega}(t(x),1)\hat{\varphi}(-\beta;0,1)W_{\ff,2}(a(\xi)t(x))|x|_2^{-2}d^\times x\\
&=\int_{\Q_2^\times}[\chi_V(x)|x|_2^2\hat{\varphi}(-\beta x;0,x^{-1})][\bchi_2(x^{-1})W_{\ff,2}(a(\xi x^2))]|x|_2^{-2}d^\times x\\
&=\sum_{n\in\Z}\hat{\varphi}(-\beta 2^{-n};0,2^n)|\xi (2^{-n})^2|_2^{\kappa+1/2}a_f\left(\left(\xi (2^{-n})^2\right)_2\right).
\end{align*}
Note that $\varphi=\varphi(2^{-1}\Z_2,2^{-1}\Z_2)
-\varphi(2^{-1}\Z_2^\times,\Z_2)-\varphi(\Z_2,2^{-1}\Z_2^\times)$.
\begin{lem}
We have
\begin{align*}
\sum_{n\in\Z}\varphi(2^{-1}\Z_2^\times,\Z_2)\hat{}\ (-\beta 2^{-n};0,2^n)
|\xi 2^{-2n}|_2^{\kappa+1/2}a_f\left(\left(\xi 2^{-2n}\right)_2\right)
&=0,
\end{align*}
and
\begin{align*}
\sum_{n\in\Z}\varphi(\Z_2,2^{-1}\Z_2^\times)\hat{}\ (-\beta 2^{-n};0,2^n)
|\xi 2^{-2n}|_2^{\kappa+1/2}a_f\left(\left(\xi 2^{-2n}\right)_2\right)
&=0.
\end{align*}
\end{lem}
\begin{proof}
Since $\beta={}^t(-b_4,-b_2,-b_3,b_1)$, 
we have $\varphi(2^{-1}\Z_2^\times,\Z_2)\hat{}\ (-\beta 2^{-n};0,2^n)=0$ unless
\begin{align*}
n\geq 0,\quad m_4-n\geq 0,\quad m_2-n= -1,\ 
m_3-n\geq0,\quad \text{and}\quad m_1-n\geq0.
\end{align*}
In this case, we have $n=m_2+1\leq\min(m_1,m_3,m_4)$.
So we find that
\[
l=\ord_2(\xi)=\ord_2(b_1b_4-b_2^2-Db_3^2)=2m_2<2n.
\]
Therefore we have
$a_f\left(\left(\xi 2^{-2n}\right)_2\right)=0$.
The second assertion is proved similarly.
\end{proof}
Let $m_0=\min(m_1,m_2+1,m_3+1,m_4)$.
Then by this lemma, we have
\begin{align*}
\W_{B,2}(1)
&=\sum_{n\in\Z}\varphi(2^{-1}\Z_2,2^{-1}\Z_2)\hat{}\ (-\beta 2^{-n};0,2^n)
|\xi 2^{-2n}|_2^{\kappa+1/2}a_f\left(\left(\xi 2^{-2n}\right)_2\right)\\
&=|\xi|_2^{\kappa+1/2}\sum_{n=0}^{m_0}(2^{n})^{2\kappa+1}a_f\left(\left(\xi 2^{-2n}\right)_2\right).
\end{align*}
This completes the proof of Lemma \ref{local} (\ref{unram}) for $p=2$.
\par

Next, we start to prove Lemma \ref{I and J} (\ref{unram2}) (for any $p$).
Then we have
\begin{align*}
J_p(\cc,s,\xi,\xi')&=\int_{\Q_p^\times}
\left|\Phi(t(x),s,\varphi'')W_{\ff,p}(a(\xi' C)t(x))\overline{W_{\g,p}(a(\xi C)t(x))}\right||x|_p^{-2}d^\times x\\
&=\int_{\Q_p^\times}\left|[\chi_{V''(\Q_p)}(x)|x|_p^{s+1}][\bchi_p(x^{-1})|\xi'Cx^2|_p^{\kappa+1/2}a_f\left((\xi' C x^2)_p\right)]
\right.\\&\quad\times\left.
\left[\overline{|\xi C x^2|_p^{\kappa+1}a_g\left((\xi C x^2)_p\right)}\right]\right||x|_p^{-2}d^\times x\\
&=|\xi|_p^{\kappa+1}|\xi'|_p^{\kappa+1/2}\int_{\Q_p^\times}|x|_p^{\sigma+4\kappa+2}
\left|a_f\left((\xi' x^2)_p\right)\overline{a_g\left((\xi x^2)_p\right)}\right|d^\times x.
\end{align*}
Since $\ord_p(\xi)=0$ or $1$, this is equal to
\[
|\xi|_p^{\kappa+1}|\xi'|_p^{\kappa+1/2}\sum_{n=0}^\infty (p^{-n})^{\sigma+4\kappa+2}
\left|a_f\left((\xi' p^{2n})_p\right)\overline{a_g\left((\xi p^{2n})_p\right)}\right|.
\]
This is the desired formula for $J_p(\cc,s,\xi,\xi')$.
The formula for $I_p(\cc,s,\xi)$ can be proved similarly.
These complete the proof of Lemma \ref{I and J} (\ref{unram2}) (for any $p$).

\subsection{The ramified $2$-adic case}
Let $v=2\mid D$.
We will prove Lemma \ref{local} (\ref{ram}) and Lemma \ref{I and J} (\ref{ram}) for this case.
We put $d=\ord_2(D)\in\{2,3\}$ and $D'=D/2^d$.
Write $D'=\delta u_0$ with $\delta\in\{\pm1\}$ and $u_0\in 1+4\Z_2$.
Note that if $d=2$, then $\delta=1$.
In this case, we have
$(2,-1)_{\Q_2}=1$, $\gamma_{\Q_2}(\psi)=\zeta_8^{-1}$.
If $d=2$, then $\gamma_{\Q_2}(D,\psi)=\gamma_{\Q_2}(D',\psi)=1$.
If $d=3$, then
\[
\bchi_2(2)=\left\{
\begin{aligned}
&1\iif D'\equiv\pm1\bmod8,\\
&-1\iif D'\equiv\pm3\bmod8,
\end{aligned}
\right.
\ 
\gamma_{\Q_2}(D,\psi)
=\left\{
\begin{aligned}
&1     \iif D'\equiv 1\bmod 8,\\
&-\I  \iif D'\equiv 3\bmod 8,\\
&-1   \iif D'\equiv 5\bmod 8,\\
&\I   \iif D'\equiv 7\bmod 8.
\end{aligned}
\right.
\]
Hence we have
$\gamma_{V}\varepsilon
=\gamma_{V''}\varepsilon
=\bchi_2(2^d)$.
\par

Let $\varphi''=\varphi''(\Z_2,2^{-1}\Z_2)$, $\varphi''_0=\varphi''(\Z_2,\Z_2)$, 
$\varphi''_1=\varphi''(2^{-1}\Z_2,(2D)^{-1}\Z_2)\in\mathcal{S}(V''(\Q_2))$.
We put $\varphi=\varphi'\otimes\varphi''$ and $\varphi_i=\varphi'\otimes\varphi''_i\in\mathcal{S}(V(\Q_2))$ for $i=0,1$.
Clearly, $\varphi_0$ is $\phi_1(\GSU(\Z_2))$-invariant.
We find that
\[
\omega(w,1)\varphi''_0=\gamma_{V''}^{-1}\vol(V''(\Z_2))\varphi''_1
\quad {\rm and\ so}\quad
\omega(w,1)\varphi_0=\gamma_{V}^{-1}\vol(V(\Z_2))\varphi_1.
\]
For 
\[
k=\begin{pmatrix}
a&b\\c&d
\end{pmatrix}
=w^{-1}\cdot n(-c/a)\cdot w\cdot t(a)\cdot n(b/a)
\in K_0(4D;\Z_p), 
\]
we have
\[
\omega(k,1)\varphi_0''=\chi_{V''}(a)\varphi''_0
\quad {\rm and \ so}\quad
\omega(k,1)\varphi_0=\chi_{V}(a)\varphi_0.
\]
A complete system of representatives of $K_0(D;\Z_2)/K_0(4D;\Z_2)$ is given by
\[
\left\{\left.
\begin{pmatrix}
1&0\\c&1
\end{pmatrix}
=-w\cdot n(-c)\cdot w\right|c\in D\Z_2/4D\Z_2\right\}.
\]
\begin{lem}\label{mean}
We have
\[
\sum_{c\in D\Z_2/4D\Z_2}\omega(-w\cdot n(-c)\cdot w,1)\varphi_0''=2\varphi''
\quad {\rm and\ so}\quad
\sum_{c\in D\Z_2/4D\Z_2}\omega(-w\cdot n(-c)\cdot w,1)\varphi_0=2\varphi.
\]
In particular, $K_0(D;\Z_2)$ acts on $\varphi''$ $($resp.~$\varphi)$ by
the scalar multiplication of the character
\[
\begin{pmatrix}
a&b\\c&d
\end{pmatrix}
=w^{-1}\cdot n(-c/a)\cdot w\cdot t(a)\cdot n(b/a)\mapsto\chi_{V''}(a)
\quad ({\rm resp.~}
\chi_V(a)).
\]
Moreover, we find that $\varphi$ is $\phi_1(\GSU(\Z_2))\phi_2((\OO\otimes_\Z\Z_2)^\times)$-invariant.
\end{lem}
\begin{proof}
We have
\[
\omega(n(-c)\cdot w,1)\varphi''(u)
=\gamma_{V''}^{-1}\vol(V(\Z_2))\varphi''_1(u)\psi(-cQ''[u]).
\]
Note that
\[
\sum_{c\in D\Z_2/4D\Z_2}\psi(-cQ''[u])
=\left\{
\begin{aligned}
&4\iif Q''[u]\in D^{-1}\Z_2,\\
&0\iif Q''[u]\not\in D^{-1}\Z_2.
\end{aligned}
\right.
\]
We find that
\[
\{u\in V''(\Q_p)|\varphi_1''(u)\not=0,\ Q''[u]\in D^{-1}\Z_2\}
=2^{-1}\Z_2e_3+D^{-1}\Z_2e_4.
\]
Hence, we have
\[
\sum_{c\in D\Z_2/4D\Z_2}\omega(n(-c)\cdot w,1)\varphi_0''
=4\gamma_{V''}^{-1}\vol(V(\Z_2))\varphi''(2^{-1}\Z_2,D^{-1}\Z_2).
\]
Since
\begin{align*}
\omega(w,1)\varphi''(2^{-1}\Z_2,D^{-1}\Z_2)
&=\gamma_{V''}^{-1}\vol(2^{-1}\Z_2e_3+D^{-1}\Z_2e_4)\varphi''
\\&=2^{1+d}\gamma_{V''}^{-1}\vol(V(\Z_2))\varphi'',
\end{align*}
we have
\[
\sum_{c\in D\Z_2/4D\Z_2}\omega(-w\cdot n(-c)\cdot w,1)\varphi_0''(u)
=2^{3+d}\vol(V(\Z_2))^{2}\varphi''(-u)
=2\varphi''(u).
\]
Let $k=w^{-1}\cdot n(-c/a)\cdot w\cdot t(a)\cdot n(b/a)\in K_0(D;\Z_2)$
and $c_1,c_2\in D\Z_2$.
Then
\[
\begin{pmatrix}
1&0\\c_1&1
\end{pmatrix}^{-1}
k
\begin{pmatrix}
1&0\\c_2&1
\end{pmatrix}
=\begin{pmatrix}
a+bc_2&*\\ *&*
\end{pmatrix}
\in K_0(D;\Z_2)
\]
and 
$\chi_{V''}(a+bc_2)=\chi_{V''}(a)$.
Therefore we have
\[
\omega(k,1)\varphi''=\chi_{V''}(a)\varphi''.
\]
The last assertion follows from Lemma \ref{invariant}.
\end{proof}
By this lemma, we find that the functions
\begin{align*}
&\SL_2(\Q_2)\ni g\longmapsto 
\hat{\omega}(g,1)\hat{\varphi}(-\beta;0,1)W_{\ff,2}(a(\xi)g),\\
&\SL_2(\Q_2)\ni g\longmapsto 
\Phi(g,s,\varphi'')W_{\ff,2}(a(\xi' C)g)\overline{W_{\g,2}(a(\xi C)g)}
\end{align*}
are right $K_0(D;\Z_2)$-invariant.
Hence we have
\begin{align*}
\W_{B,2}(1)&=(\SL_2(\Z_2):K_0(D;\Z_2))^{-1}\sum_{k\in\SL_2(\Z_2)/K_0(D;\Z_2)}\Omega_{B,2}(k),
\\
I_2(\cc,s,\xi)&=(\SL_2(\Z_2):K_0(D;\Z_2))^{-1}\sum_{k\in\SL_2(\Z_2)/K_0(D;\Z_2)}\Lam_{B,2}(k),
\end{align*}
where
\begin{align*}
\Omega_{B,2}(k)&=\int_{\Q_2^\times}\hat{\omega}(t(x)k,1)\hat{\varphi}(-\beta;0,1)W_{\ff,2}(a(\xi)t(x)k)|x|_2^{-2}d^\times x,\\
\Lam_{B,2}(k)&=\int_{\Q_2^\times}\Phi(t(x)k,s,\varphi'')W_{\ff,2}(a(\xi C)t(x)k)\overline{W_{\g,2}(a(\xi C)t(x)k)}|x|_2^{-2}d^\times x.
\end{align*}
\par

A complete system of representatives of $\SL_2(\Z_2)/K_0(D;\Z_2)$ is given by
\[
\mathcal{A}(D;\Z_2)=
\left\{-w\cdot n(-c)\cdot w,\ n(a)\cdot w
|c\in\Z_2/D\Z_2, a\in2\Z_2/D\Z_2
\right\}.
\]

\begin{lem}\label{AD}
Let $x\in\Q_2^\times$.
We put $n=\ord_2(x)$ and $l=\ord_2(\xi)$.
\begin{enumerate}
\item
For $k=-w\cdot n(-c)\cdot w$ with $c\in\Z_2^\times$, we have
\[
\omega(t(x)k,1)\varphi''(u)=2^{-d/2}\gamma_{V''}^{-1}\chi_{V''}(xc^{-1})
|x|_2\varphi''(2^{-1}\Z_2,D^{-1}\Z_2)(ux)\psi(c^{-1}Q''[ux])
\]
so that
\[
\omega(t(x)k,1)\varphi(u)=2^{-d/2}\gamma_{V}^{-1}\chi_{V}(xc^{-1})
|x|_2^3\varphi(2^{-1}\Z_2,D^{-1}\Z_2)(ux)\psi(c^{-1}Q[ux]),
\]
and $W_{\ff,2}(a(\xi)t(x)k)$ is equal to
\[
2^{-d/2}\varepsilon^{-1}\psi(\xi x^2c^{-1})\bchi_2(x^{-1}c)\bchi_2(-2^d\xi)
a_f(D_2)^{-1}|\xi x^2|_2^{\kappa+1/2}\overline{a_f\left((D\xi x^2)_2\right)}.
\]
\item
For $k=n(a)\cdot w$ with $a\in2\Z_2$, we have
\[
\omega(t(x)k,1)\varphi''(u)=2^{-d/2}\gamma_{V''}^{-1}\chi_{V''}(x)|x|_2
\varphi''(2^{-1}\Z_2,D^{-1}\Z_2)(ux)\psi(aQ''[ux])
\]
so that
\[
\omega(t(x)k,1)\varphi(u)=2^{-d/2}\gamma_{V}^{-1}\chi_{V}(x)|x|_2^3
\varphi(2^{-1}\Z_2,D^{-1}\Z_2)(ux)\psi(aQ[ux]),
\]
and $W_{\ff,2}(a(\xi)t(x)k)$ is equal to
\[
2^{-d/2}\varepsilon^{-1}\psi(a\xi x^2)\bchi_2(x^{-1})\bchi(-2^d\xi)a_f(D_2)^{-1}
|\xi x^2|_2^{\kappa+1/2}\overline{a_f\left((D\xi x^2)_2\right)}.
\]
\item
For $k=-w\cdot n(-2^{d-1})\cdot w=k_1$, we have
\[
\omega(t(x)k,1)\varphi''(u)=\chi_{V''}(x)|x|_2
\times\left\{
\begin{aligned}
&\varphi''(2^{-1}\Z_2^\times,2^{-2}\Z_2^\times)(ux),\iif d=2,\\
&\varphi''(\Z_2,2^{-2}\Z_2^\times)(ux),\iif d=3
\end{aligned}
\right.
\]
so that
\[
\omega(t(x)k,1)\varphi(u)=\chi_{V}(x)|x|_2^3
\times\left\{
\begin{aligned}
&\varphi(2^{-1}\Z_2^\times,2^{-2}\Z_2^\times)(ux),\iif d=2,\\
&\varphi(\Z_2,2^{-2}\Z_2^\times)(ux),\iif d=3
\end{aligned}
\right.
\]
and
\[
W_{\ff,2}(a(\xi)t(x)k)=
\left\{
\begin{aligned}
&\bchi_2(x^{-1})|\xi x^2|_2^{\kappa+1/2}a_f(2)^{-1}\iif l+2n=-1,\\
&0&\quad&{\rm otherwise}.
\end{aligned}
\right.
\]
\item
Assume that $d=3$ and
$k=-w\cdot n(\pm2)\cdot w$.
Then $\omega(t(x)k,1)\varphi''(u)$ is equal to
\begin{align*}
&
2^{-1}\chi_{V''}(x)|x|_2
\\&\times
\left[\varphi''(2^{-1}\Z_2^\times,2^{-2}\Z_2)
\pm\delta\I\left(\varphi''(2^{-1}\Z_2^\times,2^{-2}\Z_2^\times)-\varphi''(2^{-1}\Z_2^\times,2^{-1}\Z_2)\right)\right](ux)
\end{align*}
so that $\omega(t(x)k,1)\varphi(u)$ is equal to
\begin{align*}
&2^{-1}\chi_{V}(x)|x|_2^3
\\&\times
\left[\varphi(2^{-1}\Z_2^\times,2^{-2}\Z_2)
\pm\delta\I\left(\varphi(2^{-1}\Z_2^\times,2^{-2}\Z_2^\times)-\varphi(2^{-1}\Z_2^\times,2^{-1}\Z_2)\right)\right](ux),
\end{align*}
and $W_{\ff,2}(a(\xi)t(x)k)$ is equal to
\[
\left\{
\begin{aligned}
&\bchi_2(x^{-1})2^{-1/2}\bchi_2(\mp \eta)\varepsilon \zeta_8^{\pm \eta}
|\xi x^2|_2^{\kappa+1/2}a_f(2)^{-2}
\iif l+2n=-2,\\
&0&\quad&{\rm otherwise}.
\end{aligned}\right.
\]
\end{enumerate}
\end{lem}
\begin{proof}
(1)
Let $k=-w\cdot n(-c)\cdot w=t(c^{-1})\cdot n(c)\cdot w\cdot n(c^{-1})$.
Then we have
\begin{align*}
&\omega(t(x)k,1)\varphi''(u)
=\omega(t(xc^{-1})\cdot n(c)\cdot w,1)\varphi''(u)\\
&=\gamma_{V''}^{-1}\vol(\Z_2e_3+2^{-1}\Z_2e_4)
\omega(t(xc^{-1})\cdot n(c),1)\varphi''(2^{-1}\Z_2,D^{-1}\Z_2)(u)\\
&=2^{-d/2}\gamma_{V''}^{-1}\chi_{V''}(xc^{-1})|xc^{-1}|_2
\varphi''(2^{-1}\Z_2,D^{-1}\Z_2)(uxc^{-1})\psi(cQ''[uxc^{-1}])\\
&=2^{-d/2}\gamma_{V''}^{-1}\chi_{V''}(xc^{-1})|x|_2\varphi''(2^{-1}\Z_2,2^{-2}\Z_2)(ux)
\psi(c^{-1}Q''[ux]).
\end{align*}
On the other hand, since $a(\xi)\cdot t(xc^{-1})\cdot n(c)=n(\xi x^2c^{-1})a(\xi)t(xc^{-1})$, we have
\begin{align*}
&W_{\ff,2}(a(\xi)t(x)k)
=\psi(\xi x^2 c^{-1})\bchi_2(x^{-1} c)W_{\ff,2}(a(\xi x^2c^{-2})w)\\
&=\psi(\xi x^2 c^{-1})\bchi_2(x^{-1} c)\bchi_2(-D_2\xi x^2)\varepsilon^{-1}2^{-d/2}a_f(D_2)^{-1}
|\xi x^2|_2^{\kappa+1/2}\overline{a_f\left((D\xi x^2)_2\right)}.
\end{align*}
\par

(2) 
Let $k=n(a)\cdot w$ with $a\in2\Z_2$. Then we have
\begin{align*}
\omega(t(x)k,1)\varphi''(u)
&=\gamma_{V''}^{-1}2^{-d/2}\chi_{V''}(x)|x|_2
\varphi''(2^{-1}\Z_2,D^{-1}\Z_2)(ux)\psi(aQ''[ux]).
\end{align*}
On the other hand, since $a(\xi)t(x)n(a)=n(a\xi x^2)a(\xi)t(x)$,
we have
\begin{align*}
&W_{\ff,2}(a(\xi)t(x)k)
\\&=\psi(a\xi x^2)\bchi_2(x^{-1})\bchi(-D_2 \xi x^2)\varepsilon^{-1}2^{-d/2}a_f(D_2)^{-1}
|\xi x^2|_2^{\kappa+1/2}\overline{a_f\left((D\xi x^2)_2\right)}.
\end{align*}
\par

(3) 
We assume that $d=3$.
The proof for $d=2$ is similar.
Let $k=-w\cdot n(-4)\cdot w=k_1$.
The formula for $W_{\ff,2}(a(\xi)t(x)k_1)$ follows from Lemma \ref{explicit} immediately.
On the other hand, we have
\begin{align*}
\omega(n(-4)\cdot w,1)\varphi''(u)
&=\gamma_{V''}^{-1}\vol(\Z_2e_3+2^{-1}\Z_2e_4)\varphi''(2^{-1}\Z_2,2^{-3}\Z_2)(u)\psi(-4Q''[u]).
\end{align*}
Now we find that
\[
\psi(-4Q''[u_3e_3+u_4e_4])
=\left\{
\begin{aligned}
&1 \iif u_3\in2^{-1}\Z_2,u_4\in2^{-2}\Z_2,\\
&-1\iif u_3\in2^{-1}\Z_2,u_4\in2^{-3}\Z_2^\times.
\end{aligned}
\right.
\]
Hence we find that $\omega(n(-4)\cdot w,1)\varphi''(u)$ is equal to
\[
2^{-3/2}\gamma_{V''}^{-1}\left[\varphi''(2^{-1}\Z_2.2^{-2}\Z_2)-\varphi''(2^{-1}\Z_2,2^{-3}\Z_2^\times)\right](u).
\]
Note that
\begin{align*}
\omega(w,1)\varphi''(2^{-1}\Z_2,2^{-2}\Z_2)(u)
&=2^{1/2}\gamma_{V''}^{-1}\left[\varphi''(\Z_2,2^{-1}\Z_2)+\varphi''(\Z_2,2^{-2}\Z_2^\times)\right](u),\\
\omega(w,1)\varphi''(2^{-1}\Z_2,2^{-3}\Z_2^\times)(u)
&=2^{1/2}\gamma_{V''}^{-1}\left[\varphi''(\Z_2,2^{-1}\Z_2)-\varphi''(\Z_2,2^{-2}\Z_2^\times)\right](u).
\end{align*}
So we have
\[
\omega(w\cdot n(-4)\cdot w,1)\varphi''=\gamma_{V''}^{-2}\varphi''(\Z_2,2^{-2}\Z_2^\times).
\]
Since $\gamma_{V''}^2=\bchi_2(-1)$, we have
\[
\omega(t(x)k_1,1)\varphi''(u)=\chi_{V''}(x)|x|_2\varphi''(\Z_2,2^{-2}\Z_2^\times)(ux).
\]
\par

(4)
Assume that $d=3$.
Let $k=-w\cdot n(\pm2)\cdot w=a(\mp1)\cdot k_2\cdot a(\mp1)$.
We write $\xi =2^l\eta$ and $x=2^nu$ with $\eta, u\in\Z_2^\times$.
Then $W_{\ff,2}(a(\xi)t(x)k)$ is equal to
\[
\left\{
\begin{aligned}
&\bchi_2(x^{-1})2^{-1/2}\bchi_2(\mp \eta u^2)\varepsilon \zeta_8^{\pm \eta u^2}
|\mp\xi x^2|_2^{\kappa+1/2}a_f(2)^{-2}
\iif l+2n=-2,\\
&0&\quad&{\rm otherwise}.
\end{aligned}
\right.
\]
Since $u^2\equiv 1\bmod 8$ for $u\in\Z_2^\times$, we get the desired formula for $W_{\ff,2}(a(\xi)t(x)k)$.
On the other hand, we have
\begin{align*}
\omega(n(\pm2)\cdot w,1)\varphi''(u)
=\gamma_{V''}^{-1}\vol(\Z_2e_3+2^{-1}\Z_2e_4)\varphi''(2^{-1}\Z_2,2^{-3}\Z_2)(u)\psi(\pm2Q''[u]).
\end{align*}
Since $\pm2Q''[u_3e_3+u_4e_4]=\left(\pm2(u_3^2+8\delta u_0u_4^2)\right)$,
we have
\[
\pm2Q''[u_3e_3+u_4e_4]
\equiv\left\{
\begin{aligned}
&0                            &&\bmod \Z_2   \iif u_3\in\Z_2,u_4\in2^{-2}\Z_2,\\
&\pm\delta\cdot 4^{-1}  &&\bmod \Z_2    \iif u_3\in\Z_2,u_4\in2^{-3}\Z_2^\times,\\
&\pm2^{-1}             &&\bmod \Z_2   \iif u_3\in2^{-1}\Z_2^\times,u_4\in2^{-2}\Z_2,\\
&\pm2^{-1}\pm\delta\cdot4^{-1}
&&\bmod \Z_2    \iif u_3\in2^{-1}\Z_2^\times,u_4\in2^{-3}\Z_2^\times.
\end{aligned}
\right.
\]
Since $\psi(1/4)=-\I$, we find that $\omega(n(\pm2)\cdot w,1)\varphi''(u)$ is equal to
\begin{align*}
2^{-3/2}\gamma_{V''}^{-1}\Big[
&\varphi''(\Z_2,2^{-2}\Z_2)
-\varphi''(2^{-1}\Z_2^\times,2^{-2}\Z_2)
\\&
\pm\delta\I\varphi''(2^{-1}\Z_2^\times,2^{-3}\Z_2^\times)
\mp\delta\I\varphi''(\Z_2,2^{-3}\Z_2^\times)
\Big](u).
\end{align*}
Note that
\begin{align*}
&\omega(w,1)\varphi''(\Z_2,2^{-2}\Z_2)=
2^{-1/2}\gamma_{V''}^{-1}
\\&\times
\Big[\varphi''(\Z_2,2^{-1}\Z_2)+\varphi''(2^{-1}\Z_2^\times,2^{-1}\Z_2)
+\varphi''(\Z_2,2^{-2}\Z_2^\times)+\varphi''(2^{-1}\Z_2^\times,2^{-2}\Z_2^\times)\Big](u),\\
&\omega(w,1)\varphi''(2^{-1}\Z_2^\times,2^{-2}\Z_2)=
2^{-1/2}\gamma_{V''}^{-1}
\\&\times
\Big[\varphi''(\Z_2,2^{-1}\Z_2)-\varphi''(2^{-1}\Z_2^\times,2^{-1}\Z_2)
+\varphi''(\Z_2,2^{-2}\Z_2^\times)-\varphi''(2^{-1}\Z_2^\times,2^{-2}\Z_2^\times)\Big](u),\\
&\omega(w,1)\varphi''(2^{-1}\Z_2^\times,2^{-3}\Z_2^\times)=2^{-1/2}\gamma_{V''}^{-1}
\\&\times
\Big[\varphi''(\Z_2,2^{-1}\Z_2)-\varphi''(2^{-1}\Z_2^\times,2^{-1}\Z_2)
-\varphi''(\Z_2,2^{-2}\Z_2^\times)+\varphi''(2^{-1}\Z_2^\times,2^{-2}\Z_2^\times)\Big](u),\\
&\omega(w,1)\varphi''(\Z_2,2^{-3}\Z_2^\times)=2^{-1/2}\gamma_{V''}^{-1}
\\&\times
\Big[\varphi''(\Z_2,2^{-1}\Z_2)+\varphi''(2^{-1}\Z_2^\times,2^{-1}\Z_2)
-\varphi''(\Z_2,2^{-2}\Z_2^\times)-\varphi''(2^{-1}\Z_2^\times,2^{-2}\Z_2^\times)\Big](u).
\end{align*}
So we find that $\omega(w\cdot n(\pm2)\cdot w,1)\varphi''(u)$ is equal to
\[
2^{-1}\gamma_{V''}^{-2}
\left[\varphi''(2^{-1}\Z_2^\times,2^{-2}\Z_2)
\pm\delta\I\left(\varphi''(2^{-1}\Z_2^\times,2^{-2}\Z_2^\times)-\varphi''(2^{-1}\Z_2^\times,2^{-1}\Z_2)\right)\right](u).
\]
Since $\gamma_{V''}^{2}=\chi_{V''}(-1)$, we find that $\omega(t(x)k,1)\varphi''(u)$ is equal to
\begin{align*}
&2^{-1}\chi_{V''}(x)|x|_2
\\&\times
\left[\varphi''(2^{-1}\Z_2^\times,2^{-2}\Z_2)
\pm\delta\I\left(\varphi''(2^{-1}\Z_2^\times,2^{-2}\Z_2^\times)-\varphi''(2^{-1}\Z_2^\times,2^{-1}\Z_2)\right)\right](ux).
\end{align*}
This completes the proof.
\end{proof}
\begin{lem}\label{ord=d}
We put $l=\ord_2(\xi)$, $m_i=\ord_2(b_i)$ for $i=1,\dots,4$, 
\begin{align*}
m_0&=\min(\ord_2(b_1),\ord_2(b_2)+1,\ord_2(b_3)+d,\ord_2(b_4)),\\
m_0'&=\min(\ord_2(b_1),\ord_2(b_2),\ord_2(b_3)+1,\ord_2(b_4)),
\end{align*}
and
\begin{align*}
\delta_2&=\left\{
\begin{aligned}
&1\iif m_1,m_4\geq m_2+1=m_3+2,\\
&0&\quad&{\rm otherwise},
\end{aligned}
\right.\\
\delta_3&=\left\{
\begin{aligned}
&1\iif m_1,m_2,m_4\geq m_3+2 {\rm \ and\ } 
\ord_2(b_1b_4-b_2^2)>\ord_2(Db_3^2),\\
&0&\quad&{\rm otherwise},
\end{aligned}
\right.\\
\delta_0&=\left\{
\begin{aligned}
&1\iif m_1,m_3+2,m_4\geq m_2+1,\\
&0&\quad&{\rm otherwise}.
\end{aligned}
\right.
\end{align*}
\begin{enumerate}
\item
For $k={\bf 1}_2$, we have
\begin{align*}
\Omega_{B,2}({\bf 1}_2)&=|\xi|_2^{\kappa+1/2}\sum_{n=0}^{m_0'}(2^n)^{2\kappa+1}a_f\left((\xi 2^{-2n})_2\right),\\
\Lam_{B,2}({\bf 1}_2)&=|\xi|_2^{-s/2+1/2}\sum_{n=0}^\infty
|\xi 2^{2n}|_2^{s/2+2\kappa+1}a_f\left((\xi 2^{2n})_2\right)
\overline{a_g\left((\xi 2^{2n})_2\right)}.
\end{align*}
\item
For $k=-w\cdot n(-c)\cdot w$ with $c\in\Z_2^\times$
or $k=n(a)\cdot w$ with $a\in2\Z_2$, we have
\begin{align*}
\Omega_{B,2}(k)&=2^{-d}a_f(D_2)^{-1}|\xi|_2^{\kappa+1/2}\bchi_2(-\xi)
\sum_{n=0}^{m_0}(2^{n})^{2\kappa+1}\overline{a_f\left((D\xi 2^{-2n})_2\right)},\\
\Lam_{B,2}(k)&=2^{-d}a_f(D_2)^{-1}|\xi|_2^{-s/2+1/2}\bchi_2(-\xi C)
\\&\quad\times
\sum_{n=0}^\infty|\xi 2^{2n}|_2^{s/2+2\kappa+1}
\overline{a_f\left((D\xi 2^{2n})_2\right)a_g\left((\xi 2^{2n})_2\right)}.
\end{align*}
\item
For $k=-w\cdot n(-2^{d-1})\cdot w=k_1$, we have
\begin{align*}
\Omega_{B,2}(k)&=
\delta_d|\xi|_2^{\kappa+1/2}(2^{\ord_2(b_3)+2})^{2\kappa+1}a_f(2)^{-1},\\
\Lam_{B,2}(k)&=0.
\end{align*}
\item
Assume that $d=3$. 
Let $k_\pm=-w\cdot n(\pm2)\cdot w$. 
Then we have
\begin{align*}
\Omega_{B,2}(k_+)+\Omega_{B,2}(k_-)&=\delta_0|\xi|_2^{\kappa+1/2}(2^{\ord_2(b_2)+1})^{2\kappa+1}a_f(2)^{-2},\\
\Lam_{B,2}(k_\pm)&=0.
\end{align*}
\end{enumerate}
\end{lem}
\begin{proof}
(1) is easy.
(2) follows from Lemma \ref{AD}.
So we show (3) and (4).
\par

We show (3) only when $d=2$.
Let $k=-w\cdot n(-2)\cdot w=k_1$.
Then $\Omega_{B,2}(k)$ is equal to
\begin{align*}
\int_{\Q_2^\times}\left[\chi_{V}(x)|x|_2^2\varphi(2^{-1}\Z_2^\times,2^{-2}\Z_2^{\times})\hat{}\ (-\beta x;0,x^{-1})\right]
W_{\ff,2}(a(\xi)t(x)k_1)
|x|_2^{-2}d^\times x.
\end{align*}
Let $x\in\Q_2^\times$ and we put $n=\ord_2(x)$. 
Since $\beta={}^t(-b_4,-b_2,-b_3,b_1)$, we have 
$\varphi(2^{-1}\Z_2^\times,2^{-2}\Z_2^{\times})\hat{}\ (-\beta x;0,x^{-1})\not=0$
if and only if
\[
n\leq 0,\quad m_4+n\geq 0,\quad m_2+n=-1,\quad 
m_3+n=-2,\quad m_1+n\geq 0. 
\]
In this case, we find that $n=-\ord_2(b_2)-1=-\ord_2(b_3)-2$ and
\[
m_1,m_4\geq m_2+1=m_3+2.
\]
Moreover, since $\ord_2(b_2)=\ord_2(2b_3)$ and $D'=D/4\equiv 1\bmod4$, we have
\[
\ord_2(b_2^2+Db_3^2)=\ord_2(b_2^2+D'(2b_3)^2)
=2\ord_2(b_2)+1<\ord_2(b_1b_4).
\]
Hence, we have
$l=\ord_2(b_1b_4-b_2^2-Db_3^2)=2\ord_2(b_2)+1=-2n-1$.
So in this case, we have
\[
W_{\ff,2}(a(\xi)t(x)k_1)=\bchi_2(x^{-1})|\xi x^2|_2^{\kappa+1/2}a_f(2)^{-1}.
\]
Therefore, we have
\begin{align*}
\Omega_{B,2}(k_1)
&=\delta_2 |\xi|_2^{\kappa+1/2}(2^{-n})^{2\kappa+1}a_f(2)^{-1}.
\end{align*}
On the other hand, since
$\Phi(t(x)k_1,s)=0$
for all $x\in\Q_2^\times$ and $s\in\C$, 
we have
$\Lam_{B,2}(k_1)=0$.
The proof of (3) for $d=3$ is similar.
\par

(4) We assume that $d=3$.
Let $k_\pm=-w\cdot n(\pm2)\cdot w$. 
Then we find that $\Omega_{B,2}(k_\pm)$ is equal to
\begin{align*}
\int_{\Q_2^\times}
2^{-1}\chi_V(x)|x|_2^2
&\Big[
\varphi(2^{-1}\Z_2^\times,2^{-2}\Z_2)\pm\delta\I
\Big(\varphi(2^{-1}\Z_2^\times,2^{-2}\Z_2^\times)
\\&
-\varphi(2^{-1}\Z_2^\times,2^{-1}\Z_2)\Big)
\Big]\hat{}\ (-\beta x;0,x^{-1})
W_{\ff,2}(a(\xi)t(x)k_\pm)|x|_2^{-2}d^\times x.
\end{align*}
Write $\xi=2^l\eta$ with $\eta\in\Z_2^\times$.
We assume that $l+2\in2\Z$ and put $n=(-l-2)/2$.
Then we have
\begin{align*}
\Omega_{B,2}(k_\pm)
&=2^{-3/2}\bchi_2(\mp\eta)\varepsilon\zeta_8^{\pm\eta}|\xi 2^{2n}|_2^{\kappa+1/2}a_f(2)^{-2}
\Big[
\varphi(2^{-1}\Z_2^\times,2^{-2}\Z_2)
\\&
\pm\delta\I
\left(\varphi(2^{-1}\Z_2^\times,2^{-2}\Z_2^\times)-\varphi(2^{-1}\Z_2^\times,2^{-1}\Z_2)\right)
\Big]\hat{}\ (-\beta 2^n;0,2^{-n}).
\end{align*}
We put
$\rho=\varepsilon(\bchi_2(-\eta)\zeta_8^{\eta}+\bchi_2(\eta)\zeta_8^{-\eta})$ and 
$\rho'=\varepsilon\I\delta(\bchi_2(-\eta)\zeta_8^{\eta}-\bchi_2(\eta)\zeta_8^{-\eta})$.
Then by a simple calculation, we have
\[
\rho=\sqrt{2}\quad{\rm and }\quad
\rho'=\left\{
\begin{aligned}
&\sqrt{2} \iif \eta\equiv 1\bmod 4\Z_2,\\
&-\sqrt{2}\iif \eta\equiv 3\bmod 4\Z_2.
\end{aligned}
\right.
\]
Since $\varphi(2^{-1}\Z_2^\times,2^{-2}\Z_2)=
\varphi(2^{-1}\Z_2^\times,2^{-2}\Z_2^\times)+\varphi(2^{-1}\Z_2^\times,2^{-1}\Z_2)$,
 we find that
$\Omega_{B,2}(k_+)+\Omega_{B,2}(k_-)$ is equal to
\[
|\xi 2^{2n}|_2^{\kappa+1/2}a_f(2)^{-2}
\times\left\{
\begin{aligned}
&\varphi(2^{-1}\Z_2^\times,2^{-2}\Z_2^\times)\hat{}\ (-\beta 2^n;0,2^{-n})
\iif\eta\equiv 1\bmod 4\Z_2,\\
&\varphi(2^{-1}\Z_2^\times,2^{-1}\Z_2)\hat{}\ (-\beta 2^n;0,2^{-n})
\iif\eta\equiv 3\bmod 4\Z_2.
\end{aligned}\right.
\]
However, if $\varphi(2^{-1}\Z_2^\times,2^{-2}\Z_2^\times)\hat{}\ (-\beta 2^n;0,2^{-n})\not=0$, then we have
\[
n\leq0,\quad m_4+n\geq0,\quad m_2+n=-1,\quad m_3+n=-2,\quad m_1+n\geq0.
\]
In this case, we have
\[
l=\ord_2(b_1b_4-b_2^2-Db_3^2)=\ord_2(b_2^2)=-2n-2.
\]
So we have
\begin{align*}
\eta&=\xi 2^{-l}=4(2^nb_1)(2^nb_4)-(2^{-\ord_2(b_2)}b_2)^2-8D'(2^{n+1}b_3)^2\\
&\equiv -(2^{-\ord_2(b_2)}b_2)^2-2\delta(2^{n+2}b_3)^2\bmod 4\Z_2\\
&\equiv -1+2\equiv 1\bmod 4\Z_2.
\end{align*}
Note that $u^2\equiv 1\bmod4\Z_2$ for all $u\in\Z_2^\times$ and $\pm2\equiv 2\bmod4\Z_2$.
Similarly, if $\varphi(2^{-1}\Z_2^\times,2^{-1}\Z_2)\hat{}\ (-\beta 2^n;0,2^{-n})\not=0$, then
we have $l=-2n-2$ and $\eta\equiv -1\bmod4\Z_2$.
Therefore we have
\begin{align*}
\Omega_{B,2}(k_+)+\Omega_{B,2}(k_-)
&=|\xi 2^{2n}|_2^{\kappa+1/2}a_f(2)^{-2}\varphi(2^{-1}\Z_2^\times,2^{-2}\Z_2)\hat{}\ (-\beta 2^n;0,2^{-n})\\
&=\delta_0|\xi|_2^{\kappa+1/2}(2^{\ord_2(b_2)+1})^{2\kappa+1}a_f(2)^{-2}.
\end{align*}
On the other hand, since
$\Phi(t(x)k_\pm,s)=0$
for all $x\in\Q_2^\times$ and $s\in\C$, we have
$\Lam_{B,2}(k_\pm)=0$.
This completes the proof.
\end{proof}
\par

Now, we start to prove Lemma \ref{local} (\ref{ram}) for the case when $p=2\mid D$.
First, we assume that $d=\ord_2(D)=2$.
By Lemma \ref{GU(Zp)} and Lemma \ref{mean},
the function $\W_{B,2}$ is right $\GU(\Z_2)$-invariant.
We put $q_2=(\SL_2(\Z_2):K_0(D;\Z_2))$.
Then
\begin{align*}
\W_{B,2}(1)&=q_2^{-1}|\xi|_2^{\kappa+1/2}a_f(D_2)^{-1}
\left[a_f(D_2)\sum_{n=0}^{m_0'}(2^n)^{2\kappa+1}a_f\left((\xi 2^{-2n})_2\right)
\right.\\&\left.
+\sum_{n=0}^{m_0}(2^n)^{\kappa+1/2}\bchi_2(-\xi)\overline{a_f\left((D\xi 2^{-2n})_2\right)}
+a_f(D_2)\delta_2 (2^{m_3+2})^{2\kappa+1}a_f(2)^{-1}
\right].
\end{align*} 
For $0\leq n\leq m_0'$, we find that
$\xi 2^{-2n}\in\Z_2$. 
Note that 
\[
m_0'=\min(m_1,m_2,m_3+1,m_4)\leq 
m_0=\min(m_1,m_2+1,m_3+2,m_4)
\]
and $m_0'=m_0$ 
 if and only if 
$m_0'=\min(m_1,m_4)$.
In this case, we have $\delta_2=0$.
So we consider the case when $m_0\not=m_0'$.
We divide the case into the three cases as follows.
\par

If $m_1,m_4>m_2=m_3+1$, 
then we find that $\delta=1$, $m_0=m_2+1$ and 
\[
l= \ord_2(b_1b_4-b_2^2-Db_3^2)=\ord_2(b_2^2+D'(2b_3)^2)=2m_2+1.
\]
So we have $(D\xi 2^{-2m_0})_2= 2$. 
Moreover in this case we have
\begin{align*}
a_f(D_2)\delta_2 (2^{m_3+2})^{2\kappa+1}a_f(2)^{-1}
&=(2^{m_0})^{2\kappa+1}a_f(2).
\end{align*}
\par

If $m_1,m_4>m_2>m_3+1$, then we have 
$\delta_2=0$, $l=2(m_3+1)$ and $m_0=m_3+2$. 
So we have $(D\xi 2^{-2m_0})_2=1$. 
In this case, since
\begin{align*}
-D\xi 2^{-2m_0}\equiv (Db_3 2^{-m_0})^2\equiv 1\bmod 4\Z_2
\end{align*}
and $(\Z_2^\times,1+4\Z_2)_{\Q_2}=1$,
we have
\begin{align*}
\underline{\chi}_2(-\xi)&=(-D,-\xi)_{\Q_2}
=\left(-D,-D\xi 2^{-2m_0}\right)_{\Q_2}
=\left(-D',-D\xi 2^{-2m_0}\right)_{\Q_2}=1.
\end{align*}
\par

If $m_1,m_4>m_3+1>m_2$, by the same calculation, we have
$\delta=0$, $m_0=m_2+1$ and $(D\xi 2^{-2m_0})_2=1$.
Moreover, since
\[
-D\xi 2^{-2m_0}\equiv D'(2b_2 2^{-m_0})^2\equiv 1\bmod 4\Z_2,
\]
we have $\underline{\chi}_2(-\xi)=1$.
This completes the proof of Lemma \ref{ram} for the case when $p=2$ and $\ord_2(D)=2$.
The of Lemma \ref{ram} for the case when $p=2$ and $\ord_2(D)=3$ is similar.
\par

Next, we start to prove Lemma \ref{I and J} (\ref{ram2}) for the case when $p=2\mid D$.
We put $q_2=(\SL_2(\Z_2):K_0(D;\Z_2))$.
By Lemma \ref{ord=d}, we have
\begin{align*}
I_2(\cc,s,\xi)&=q_2^{-1}a_f(D_2)^{-1}|\xi|_2^{-s/2+1/2}
\\&\times
\sum_{n=0}^{\infty}|\xi 2^{2n}|_2^{s/2+2\kappa+1}
\left[a_f\left((D\xi 2^{2n})_2\right)+\bchi_2(-\xi C)\overline{a_f\left((D\xi 2^{2n})_2\right)}\right]
\overline{a_g\left((\xi 2^{2n})_2\right)}.
\end{align*}
This is the desired formula for $I_2(\cc,s,\xi)$.
By a calculation similar to that of $I_2(\cc,s,\xi)$, we find that $J_2(\cc,s,\xi,\xi')$ is equal to
\begin{align*}
&q_2^{-1}|\xi'|_2^{\kappa+1/2}|\xi|_2^{\kappa+1}|a_f(D_2)|^{-1}
\\&\times
\sum_{n=0}^\infty |2^{2n}|_2^{\sigma/2+2\kappa+1}
\left[\left|a_f(D_2)a_f\left((\xi' 2^{2n})_2\right)\right|+\left|\overline{a_f\left((D\xi' 2^{2n})_2\right)}\right|\right]
\left|\overline{a_g\left((\xi 2^{2n})_2\right)}\right|.
\end{align*}
Since 
$\left|a_f(D_2)a_f\left((\xi' 2^{2n})_2\right)\right|
\leq \left|a_f\left((D\xi' 2^{2n})_2\right)\right|$
for all $\xi'\in\Q^\times$, we get the desired estimation.
This completes the proof  of Lemma \ref{I and J} (\ref{ram2}) for the case when $p=2$.

\subsection{The case when $\ord_p(C)\not=0$}
Finally, we will prove Lemma \ref{local} (\ref{divC}) and Lemma \ref{I and J} (\ref{divC2}).
\begin{lem}\label{Haar}
For $a\in\Q_p^\times$ and $f\in L^1(\SL_2(\Q_p))$, we have
\[
\int_{\SL_2(\Q_p)}f(x)dx
=|a|_p\int_{B(\Q_p)}\int_{d(a)\SL_2(\Z_p)d(a)^{-1}}f(bk)dkdb.
\]
\end{lem}
\begin{proof}
This lemma is proved by a change of variables immediately.
\end{proof}
Now, we start to prove Lemma \ref{local} (\ref{divC}).
Let $\cc$ be an integral ideal of $K$.
We assume that $\cc$ is prime to $2D$ and $p\mid C=N(\cc)$.
Let $t\in\A_{K,\fin}^\times$ such that
$\ord_\p(t_\p)=\ord_\p(\cc)$ for all prime ideals $\p$ of $K$.
Put $\gamma=r_t$.
Write $t_p=x_p\otimes 1+y_p\otimes \sD\in\Q_p\otimes K$. 
Note that $\varphi_p$ is $\phi_1(\GSU(\Z_p))\phi_2((\OO\otimes_\Z\Z_p)^\times)$-invariant by Lemma \ref{invariant}.
Therefore, by Lemma \ref{GU(Zp)}, we only have to calculate $\W_{B,p}(\gamma_p)$.
Recall that
$\W_{B,p}(\gamma_p)$ is equal to
\[
\int_{N(\Q_p)\bs\SL_2(\Q_p)}
\hat{\omega}(\alpha\cdot d(N_{K/\Q}(t_p)),\phi_2(t_p))\hat{\varphi}_p(-\beta;0,1)
W_{\ff,p}(a(\xi)\alpha\cdot d(N_{K/\Q}(t_p)))d\alpha.
\]
Since the integrand is right $d(N_{K/\Q}(t_p))\SL_2(\Z_p)d(N_{K/\Q}(t_p))^{-1}$-invariant,
by Lemma \ref{Haar}, we have
\begin{align*}
\W_{B,p}(\gamma_p)=|N_{K/\Q}(t_p)|_p\int_{\Q_p^\times}
&\hat{\omega}(t(x)d(N_{K/\Q}(t_p)),\phi_2(t_p))\hat{\varphi}_p(-\beta;0,1)
\\&\times
W_{\ff,p}(a(\xi)t(x)d(N_{K/\Q}(t_p)))|x|_p^{-2}d^\times x.
\end{align*}
Note that $|N_{K/\Q}(t_p)|_p=|C|_p$.
We have
\begin{align*}
&\hat{\omega}(t(x)d(N_{K/\Q}(t_p)),\phi_2(t))\hat{\varphi}_p(-\beta;0,1)
\\&=\int_{\Q_p}\chi_V(x)|x|_p^3|N_{K/\Q}(t_p)|_p^{-3/2}\varphi_p
\left(\phi_2(t_p)^{-1}
\begin{pmatrix}
zx\\-\beta x\\0
\end{pmatrix}\right)
\psi(z)dz.
\end{align*}
Since
\[
\phi_2(t_p)^{-1}=N_{K/\Q}(t_p)^{-1}
\begin{pmatrix}
1&&\\
&\phi'(t_p)&\\
&&N_{K/\Q}(t_p)
\end{pmatrix},
\]
we find that $\hat{\omega}(t(x)\cdot d(N_{K/\Q}(t_p)),\phi_2(t))\hat{\varphi}_p(-\beta;0,1)$ is equal to
\begin{align*}
\chi_V(x)|x|_p^2|C|_p^{-1/2}\hat{\varphi}_p(-N_{K/\Q}(t_p)^{-1}\phi'(t_p)\beta x;0,N_{K/\Q}(t_p)x^{-1}).
\end{align*}
On the other hand, we have
\begin{align*}
&W_{\ff,p}(a(\xi)t(x)d(N_K/\Q(t_p)))
\\&=\bchi_p(x^{-1}N_{K/\Q}(t_p))|\xi x^2 N_{K/\Q}(t_p)^{-1}|_p^{\kappa+1/2}
a_f\left((\xi x^2 N_{K/\Q}(t_p)^{-1})_p\right)\\
&=\bchi_p(x^{-1})|\xi x^2 N_{K/\Q}(t_p)^{-2}C|_p^{\kappa+1/2}
a_f\left((\xi x^2 N_{K/\Q}(t_p)^{-2}C)_p\right).
\end{align*}
Hence we have
\begin{align*}
\W_{B,p}(\gamma_p)
&=|C|_p^{\kappa+1}|\xi|_p^{\kappa+1/2}\int_{\Q_p^\times}
\hat{\varphi}_p(-\phi'(t_p)\beta x;0,x^{-1})|x^2|_p^{\kappa+1/2}a_f\left((\xi x^2C)_p\right)d^\times x\\
&=|C|_p^{\kappa+1}|\xi|_p^{\kappa+1/2}\sum_{n=0}^\infty
\hat{\varphi}_p(-\phi'(t_p)\beta p^{-n};0,p^n)(p^n)^{2\kappa+1}a_f\left((\xi Cp^{-2n})_p\right)
\end{align*}
as desired. 
The last assertion of Lemma \ref{local} (\ref{divC}) is easy.
\par

Next, we start to prove Lemma \ref{I and J} (\ref{divC2}).
Let $\cc\in J_\OO^D$.
We assume that $C=N(\cc)$ is a square free integer and $p\mid C$.
Let $t\in\A_{K,\fin}^\times$ such that
$\ord_\p(t_\p)=\ord_\p(\cc)$ for all prime ideals $\p$ of $K$.
Hence, we have $|N_{K/\Q}(t_p)|_p=|C|_p=p^{-1}$ and $C\cdot N_{K/\Q}(t_p)^{-1}\in\Z_p^\times$.
Note that for $k'\in K'=d(N_{K/\Q}(t_p))\SL_2(\Z_p)d(N_{K/\Q}(t_p))^{-1}$, we have
\begin{align*}
&\Phi(t(x)k',s,L(\phi_2''(t_p))\varphi''_p)\\
&=\chi_{V''}(x)|x|_p^{s+1}[\omega(k'\cdot d(N_{K/\Q}(t_p)),\phi_2''(t_p))\varphi_p''](0)|a(k')|_p^{s}\\
&=\chi_{V''}(x)|x|_p^{s+1}[L(\phi_2''(t_p))\varphi_p''](0)|a(k')|_p^{s}
=\chi_{V''}(x)|x|_p^{s+1}|N_{K/\Q}(t_p)|_p^{-1/2}|a(k')|_p^{s}.
\end{align*}
By Lemma \ref{Haar}, $J_p(\cc,s,\xi,\xi')$ is equal to
\begin{align*}
&\int_{\Q_p^\times}
\left(|N_{K/\Q}(t_p)|_p\int_{K'}\left|\chi_{V''}(x)|x|_p^{s+1}|N_{K/\Q}(t_p)|_p^{-1/2}|a(k')|_p^s\right| dk'\right)
\\&\times
\left|W_{\ff,p}(a(\xi' C)t(x)d(N_{K/\Q}(t_p)))\overline{W_{\g,p}(a(\xi C)t(x)d(N_{K/\Q}(t_p)))}
\right| |x|_p^{-2}d^\times x.
\end{align*}
Now $|a(k')|$ is right $\left(K'\cap\SL_2(\Z_p)\right)$-invariant.
Since $\ord_p(N_{K/\Q}(t_p))=1$, 
a complete system of representatives of $K'/K'\cap\SL_2(\Z_p)$ is given by
\[
\left\{\left.
\begin{pmatrix}
1&b/p\\0&1
\end{pmatrix},\ 
\begin{pmatrix}
0&p^{-1}\\-p&0
\end{pmatrix}
\right|b\in\Z_p/p\Z_p
\right\}.
\]
Since
\[
n(b/p)\in N(\Q_p)
\quad{\rm and}\quad
\begin{pmatrix}
0&p^{-1}\\-p&0
\end{pmatrix}
=\begin{pmatrix}
p^{-1}&0\\0&p
\end{pmatrix}\begin{pmatrix}
0&1\\-1&0
\end{pmatrix},
\]
we have
\[
\left|a\begin{pmatrix}
1&b/p\\0&1
\end{pmatrix}\right|=1
\quad{\rm and }\quad 
\left|a\begin{pmatrix}
0&p^{-1}\\-p&0
\end{pmatrix}\right|=|p^{-1}|_p=p.
\]
Hence, we have
\[
\int_{K'}|a(k')|^sdk=\frac{1}{p+1}(p+p^s)=\frac{p^s+p}{1+p}
\]
for $s\in\C$.
Therefore we find that  $J_p(\cc,s,\xi,\xi')$ is equal to
\begin{align*}
&p^{-1/2}\int_{\Q_p^\times}\left|\chi_{V''}(x)|x|_p^{\sigma+1}\frac{p^\sigma+p}{1+p}\right|
\\&\quad\times
\Big|\bchi_p(x^{-1}N_{K/\Q}(t_p))|\xi' Cx^2N_{K/\Q}(t_p)^{-1}|_p^{\kappa+1/2}
a_f\left(\xi' Cx^2N_{K/\Q}(t_p)^{-1}\right)\Big|
\\&\quad\times
\left|\overline{|\xi Cx^2N_{K/\Q}(t_p)^{-1}|_p^{\kappa+1}a_g\left(\xi Cx^2N_{K/\Q}(t_p)^{-1}\right)}\right|
|x|_p^{-2}d^\times x\\
&=p^{-1/2}\frac{p^\sigma+p}{1+p}|\xi'|_p^{\kappa+1/2}|\xi|_p^{\kappa+1}
\\&\quad\times
\int_{\Q_p^\times}|x|_p^{\sigma+4\kappa+2}
\left|a_f\left(\xi' Cx^2N_{K/\Q}(t_p)^{-1}\right)\overline{a_g\left(\xi Cx^2N_{K/\Q}(t_p)^{-1}\right)}\right|d^\times x\\
&=p^{-1/2}\frac{p^\sigma+p}{1+p}|\xi'|_p^{\kappa+1/2}|\xi|_p^{\kappa+1}
\sum_{n=0}^\infty (p^{-n})^{\sigma+4\kappa+2}
\left|a_f\left(\xi' p^{2n}\right)\overline{a_g\left(\xi p^{2n}\right)}\right|.
\end{align*}
This is the desired formula for $J_p(\cc,s,\xi,\xi')$.
The formula for $I_p(\cc,s,\xi)$ is proved similarly.
This completes the proof of Lemma \ref{I and J} (\ref{divC2}).

\end{document}